\documentclass[11pt]{amsart}

\usepackage{graphicx} 
\usepackage[utf8]{inputenc}
\usepackage{amsmath}
\usepackage{amssymb}
\usepackage{xcolor}
\usepackage{hyperref}
\usepackage{cleveref}
\usepackage{pgfplots}
\usepackage{float}
\pgfplotsset{compat=1.15}
\usepackage{mathrsfs}
\usetikzlibrary{arrows} 
\usepackage[margin = 2.5 cm]{geometry}

\newcommand{\II}{\mathcal{I}}
\newcommand{\n}{\mathrm{n}}

\newcommand{\depth}{\operatorname{depth}}
\newcommand{\G}{\mathfrak{G}}
\newcommand{\D}{\mathcal{D}}
\newcommand{\p}{\mathfrak{p}}
\newcommand{\new}[1]{{\color{black}#1}}

\newtheorem{theorem}{Theorem}[section]
\newtheorem{lemma}[theorem]{Lemma}
\newtheorem{proposition}[theorem]{Proposition}
\newtheorem{corollary}[theorem]{Corollary}
\newtheorem{definition}[theorem]{Definition}
\theoremstyle{definition}
\newtheorem{notation}[theorem]{Notation}
\newtheorem{example}[theorem]{Example}
\theoremstyle{remark}
\newtheorem{remark}[theorem]{Remark}

\title{Classification of unmixed parity binomial edge ideals of cactus and chordal graphs}
\author{Deblina Dey}
\email{ma20d750@smail.iitm.ac.in}
\author{A. V. Jayanthan}
\email{jayanav@iitm.ac.in}
\author{Sarang Sane}
\email{sarangsanemath@gmail.com}
\address{Department of Mathematics, Indian Institute of Technology Madras, Chennai, Tamil Nadu, India - 600036.}

\begin{document}
\begin{abstract}
In this article, we characterize all unmixed and Cohen-Macaulay parity binomial edge ideals of cactus and chordal graphs in terms of the structural properties of the graph.
\end{abstract}



\maketitle

\section{Introduction}
Let $G$ be a finite simple graph on the vertex set $V(G) = [n] = \{1,\ldots,n\}$ and the edge set $E(G)$. Kahle, Sarmiento and Windisch in 2016, \cite{Kahle2016}, introduced \textit{parity binomial edge ideals}: 
\[
\II_G := (x_ix_j-y_iy_j : \{i,j\} \in E(G)) \subset R = K[x_1,\ldots,x_n,y_1,\ldots,y_n],
\]
where $K$ is a field. This ideal is a subtle variation of another class of binomial ideals, called \textit{binomial edge ideal}, defined by 
Herzog et al., \cite{HHHKR}:
\[
J_G := (x_iy_j-x_jy_i : \{i,j\} \in E(G)) \subset R = K[x_1,\ldots,x_n,y_1,\ldots,y_n],
\]
A close relation has been observed between several algebraic properties and/or invariants of these edge ideals and combinatorial properties and/or invariants of the graph $G$. 

If $G$ is a bipartite graph, then its parity binomial edge ideal is obtained from its binomial edge ideal by a simple change of variables. As noted in the paper of Kahle et al., parity binomial edge ideals share a number of properties with binomial edge ideals, but the combinatorics is subtler. In \cite{Kahle2016}, the authors described a Gr\"obner basis for $\II_G$. \new{If char$(K) \neq 2$, they show that $\II_G$ is a radical ideal and explicitly obtain the structure of the minimal primes in terms of the disconnector sets (see  Definition \ref{def: disconnector}) of the graph. Our study of the unmixed property of $\II_G$ crucially uses this description of the minimal primes and hence in this article, we assume that char$(K) \neq 2$. } 

In \cite{Kum21}, Kumar studied some of the structural properties of $\II_G$. He proved that for a bipartite graph $G$, $\II_G$ is a complete intersection if and only if $G$ is a disjoint union of paths, and, for a non-bipartite graph, $\II_G$ is a complete intersection if and only if $G$ is a cycle on an odd number of vertices. \new{He also classifies graphs with almost complete intersection parity binomial edge ideals and establishes when their associated Rees algebra $\mathcal{R}(I_G)$ is Cohen-Macaulay.}

Among the homological invariants, Betti numbers, Castelnuovo-Mumford regularity and depth are well studied for many classes of homogeneous ideals of polynomial rings. While computation of depth is in general challenging, one particular instance, of maximal depth, i.e., being Cohen-Macaulay, has received a lot of attention.

An ideal $I$ in a Noetherian ring $R$ is said to be Cohen-Macaulay if $R/I$ is Cohen-Macaulay. Understanding the Cohen-Macaulayness of an ideal is a classical topic. When it comes to ideals associated with certain geometric/topological/combinatorial structures, the attempt is to characterize the Cohen-Macaulayness in terms of the structure associated with the ideal. When $G$ is a bipartite graph, Bolognini, Macchia and Strazzanti completely characterized when the binomial edge ideal $J_G$ (and hence the parity binomial edge ideal $\II_G$) is Cohen-Macaulay in terms of the structure of $G$, \cite{BMS18}. For an arbitrary graph $G$, Bolognini et al. conjectured that $J_G$ is Cohen-Macaulay if and only if $G$ is accessible (see \cite{BMS22} for the details). They proved the `only if' part of this conjecture and proved the `if' part for chordal graphs and graphs containing a Hamiltonian path. The general case is still open. For the case of parity binomial edge ideals, no such results, and not even a conjecture is available in this direction. 

Any Cohen-Macaulay ideal must be unmixed. Hence, to understand the Cohen-Macaulayness, one should understand the unmixed property of the ideal as well. In this article, we first investigate the unmixedness of parity binomial edge ideals of non-bipartite chordal graphs. As a consequence of our investigation, we are also able to characterize their Cohen-Macaulayness as well. 

The unmixed property of the parity binomial edge ideal is much more complex than that of the binomial edge ideals. In the case of binomial edge ideals, the combinatorial object that determines the minimality of an associated prime is called a cut set and there is a unique prime ideal associated to each cut set. In the case of parity binomial edge ideals, the corresponding object is called a disconnector set and there are several prime ideals associated with one disconnector set. Our investigation of the unmixedness property crucially depends on the construction of suitable disconnector sets.

We briefly describe the structure of the article. After discussing preliminaries in \Cref{sec:prelim}, we study the unnmixedness of parity binomial edge ideals of cactus graphs in \Cref{cactus}. We prove that for a connected non-bipartite cactus graph $G$, $\II_G$ is unmixed only if it is Cohen-Macaulay if and only if it is Gorenstein if and only if it is a complete intersection if and only if $G$ is an odd cycle, \Cref{cm Cactus graph}.

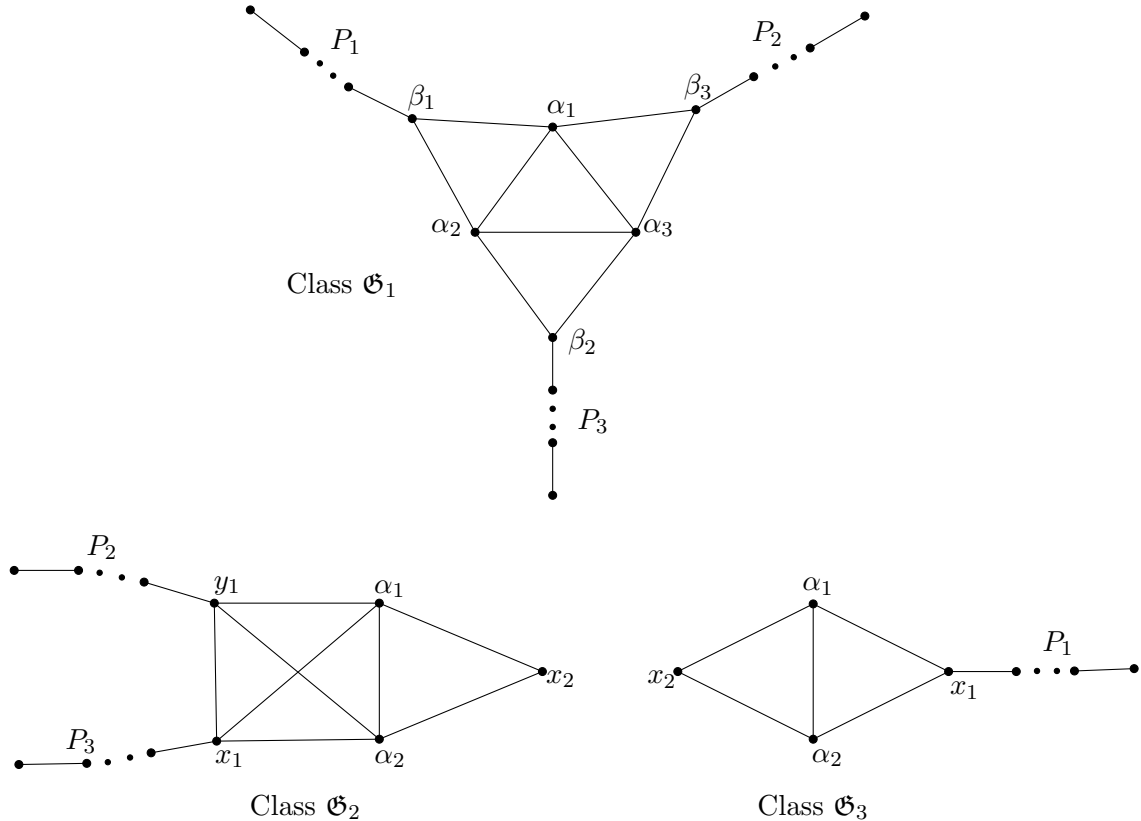
\begin{figure}[H]
\centering   
   \begin{tikzpicture}[line cap=round,line join=round,>=triangle 45,x=0.7cm,y=0.7cm]
\clip(-7,-5.5) rectangle (8,6);
\draw  (0,2)-- (-1.4733843159543591,0);
\draw  (-1.4733843159543591,0)-- (1.58,0);
\draw  (0,2)-- (1.58,0);
\draw  (0,2)-- (2.72,2.33);
\draw  (1.58,0)-- (2.72,2.33);
\draw  (0,2)-- (-2.666934000319876,2.1600537638719857);
\draw  (-1.4733843159543591,0)-- (-2.666934000319876,2.1600537638719857);
\draw  (-1.4733843159543591,0)-- (0,-2);
\draw  (1.58,0)-- (0,-2);
\draw  (2.72,2.33)-- (3.818707583389515,2.9570443593589433);
\draw  (5.929403904181309,4.109516794109229)-- (4.890144585497376,3.5025031967956703);
\draw  (-2.666934000319876,2.1600537638719857)-- (-3.8761581590214846,2.7622376317029693);
\draw  (-4.7138270879421755,3.4245805057332808)-- (-5.741801018811592,4.22603464691781);
\draw  (0,-2)-- (0,-3);
\draw  (0,-4)-- (0,-5);
\draw (-4.4174735373893,4.043209932537573) node[anchor=north west] {$P_1$};
\draw (3.590420905807785,4.2) node[anchor=north west] {$P_2$};
\draw (0.2771552211150657,-3.1711348262939736) node[anchor=north west] {$P_3$};
\draw(-4,-1) node {Class $\mathfrak{G}_1$};
\draw (-0.3063579579080759,2.7) node[anchor=north west] {$\alpha_1$};
\draw (-2.5,0.4360375531217998) node[anchor=north west] {$\alpha_2$};
\draw (1.52375155811905,0.4360375531217998) node[anchor=north west] {$\alpha_3$};
\draw (-2.952597284554512,3) node[anchor=north west] {$\beta_1$};
\draw (0.1,-1.6471186155437878) node[anchor=north west] {$\beta_2$};
\draw (2.3,3.2) node[anchor=north west] {$\beta_3$};
\begin{scriptsize}
\draw [fill=black] (0,2) circle (1.5pt);
\draw [fill=black] (-1.4733843159543591,0) circle (1.5pt);
\draw [fill=black] (1.58,0) circle (1.5pt);
\draw [fill=black] (2.72,2.33) circle (1.5pt);
\draw [fill=black] (-2.666934000319876,2.1600537638719857) circle (1.5pt);
\draw [fill=black] (0,-2) circle (1.5pt);
\draw [fill=black] (3.818707583389515,2.9570443593589433) circle (1.5pt);
\draw [fill=black] (4.227801711467062,3.1518510870149172) circle (1pt);
\draw [fill=black] (4.578453821247816,3.327177141905294) circle (1pt);
\draw [fill=black] (5.929403904181309,4.109516794109229) circle (1.5pt);
\draw [fill=black] (4.890144585497376,3.5025031967956703) circle (1.5pt);
\draw [fill=black] (-3.8761581590214846,2.7622376317029693) circle (1.5pt);
\draw [fill=black] (-4.168368250505447,2.976525032124541) circle (1pt);
\draw [fill=black] (-4.421616996458214,3.2102931053117096) circle (1pt);
\draw [fill=black] (-4.7138270879421755,3.4245805057332808) circle (1.5pt);
\draw [fill=black] (-5.741801018811592,4.22603464691781) circle (1.5pt);
\draw [fill=black] (0,-3) circle (1.5pt);
\draw [fill=black] (0,-3.7) circle (1pt);
\draw [fill=black] (0,-4) circle (1.5pt);
\draw [fill=black] (0,-5) circle (1.5pt);
\draw [fill=black] (0,-3.3546936166946137) circle (1pt);
\end{scriptsize}
\end{tikzpicture}

\centering
    
\begin{tikzpicture}[line cap=round,line join=round,>=triangle 45,x=0.9cm,y=0.9cm]
\draw  (-3.8463466690239216,2.011325981637946)-- (-3.811773565817807,-0.028487107522841483);
\draw  (-3.8463466690239216,2.011325981637946)-- (-1.408942892992807,2.011325981637946);
\draw  (-1.408942892992807,2.011325981637946)-- (-1.408942892992807,0);
\draw  (-3.811773565817807,-0.028487107522841483)-- (-1.408942892992807,0);
\draw  (-3.8463466690239216,2.011325981637946)-- (-1.408942892992807,0);
\draw  (-3.811773565817807,-0.028487107522841483)-- (-1.408942892992807,2.011325981637946);
\draw  (-1.408942892992807,2.011325981637946)-- (1,1);
\draw  (-1.408942892992807,0)-- (1,1);
\draw  (-3.811773565817807,-0.028487107522841483)-- (-4.77982045558903,-0.2013526235534167);
\draw  (-3.8463466690239216,2.011325981637946)-- (-4.883539765207375,2.3224839104929815);
\draw  (-5.851586654978598,2.4953494265235565)-- (-6.802346993146763,2.4953494265235565);
\draw  (3,1)-- (5,2);
\draw  (3,1)-- (5,0);
\draw  (5,2)-- (5,0);
\draw  (5,2)-- (7,1);
\draw  (5,0)-- (7,1);
\draw  (7,1)-- (8,1);
\draw  (8.85926875922338,1.0087059886606098)-- (9.740882890979314,1.0432790918667247);
\draw  (-5.730580793757195,-0.3569315879809344)-- (-6.733200786734533,-0.3742181395839919);
\draw (8.236952901513307,1.7520277075920812) node[anchor=north west] {$P_1$};
\draw (-5.886159758184713,3.169524939042798) node[anchor=north west] {$P_2$};
\draw (-6.197317687039749,0.3) node[anchor=north west] {$P_3$};
\draw (-1.6336680638325551,2.5) node[anchor=north west] {$\alpha_1$};
\draw (-1.6299487542142098,-0.0161) node[anchor=north west] {$\alpha_2$};
\draw (4.693209822886509,2.53736114896636) node[anchor=north west] {$\alpha_1$};
\draw (4.831502235710969,0) node[anchor=north west] {$\alpha_2$};
\draw (2.4,1.16) node[anchor=north west] {$x_2$};
\draw (6.871315324871761,1) node[anchor=north west] {$x_1$};
\draw (-3.9846390818483823,-0.0161) node[anchor=north west] {$x_1$};
\draw (-4.00192563345144,2.555) node[anchor=north west] {$y_1$};
\draw (0.9074550218169051,1.1642849530881256) node[anchor=north west] {$x_2$};

\draw (-2.5, -1) node {Class $\mathfrak{G}_2$};
\draw (5, -1) node {Class $\mathfrak{G}_3$};
\begin{scriptsize} 
\draw [fill=black] (-3.8463466690239216,2.011325981637946) circle (1.5pt);
\draw [fill=black] (-3.811773565817807,-0.028487107522841483) circle (1.5pt);
\draw [fill=black] (-1.408942892992807,2.011325981637946) circle (1.5pt);
\draw [fill=black] (-1.408942892992807,0) circle (1.5pt);
\draw [fill=black] (1,1) circle (1.5pt);
\draw [fill=black] (-4.77982045558903,-0.2013526235534167) circle (1.5pt);
\draw [fill=black] (-4.883539765207375,2.3224839104929815) circle (1.5pt);
\draw [fill=black] (-5.211984245665469,2.3916301169052114) circle (1pt);
\draw [fill=black] (-5.540428726123562,2.460776323317442) circle (1pt);
\draw [fill=black] (-5.851586654978598,2.4953494265235565) circle (1.5pt);
\draw [fill=black] (-6.802346993146763,2.4953494265235565) circle (1.5pt);
\draw [fill=black] (3,1) circle (1.5pt);
\draw [fill=black] (5,2) circle (1.5pt);
\draw [fill=black] (5,0) circle (1.5pt);
\draw [fill=black] (7,1) circle (1.5pt);
\draw [fill=black] (8,1) circle (1.5pt);
\draw [fill=black] (8.306099107925538,1.0087059886606098) circle (1pt);
\draw [fill=black] (8.599970485177517,1.0087059886606098) circle (1pt);
\draw [fill=black] (8.85926875922338,1.0087059886606098) circle (1.5pt);
\draw [fill=black] (9.740882890979314,1.0432790918667247) circle (1.5pt);
\draw [fill=black] (-5.419422864902159,-0.33964503637787685) circle (1pt);
\draw [fill=black] (-5.090978384444066,-0.2704988299656468) circle (1pt);
\draw [fill=black] (-5.730580793757195,-0.3569315879809344) circle (1.5pt);
\draw [fill=black] (-6.733200786734533,-0.3742181395839919) circle (1.5pt);
\end{scriptsize}
\end{tikzpicture}
\caption{\new{Unmixed chordal graph classes: $\mathfrak{G}_1, \mathfrak{G_2}$ and $\mathfrak{G_3}$, where $P_1, P_2, P_3$ are path graphs of length at least one}}
    \label{fig:fig 13}
\end{figure}
A chordal graph can be realized as a clique sum of complete graphs (immediate consequence of \cite[Theorem 5.3.17]{West}).
We use this structure to describe some disconnector sets which we further use for classifying when the parity binomial edge ideals are unmixed. \new{deleted "}
Let $G$ be the clique sum of $K_{n_1}, K_{n_2}, \ldots, K_{n_t}.$ For $j \geq 2$, we define $K_{r_j} = \left( \cup_{i < j} K_{n_i} \right) \cap K_{n_j}$. Note that $V(K_{r_j})$ is a disconnector set and we use it in several proofs to obtain contradictions or structural results.
First we show that if $G$ is a path, a $K_3$ or one of the graphs in $\G_1, \G_2$ or $\G_3$ (see \Cref{fig:fig 13}), then $\II_G$ is unmixed,  
\Cref{I_G is unmixed only if G has there forms}. The main purpose of this article is to prove the converse of this theorem. 

In \Cref{sec:algo}, we first propose an algorithm to construct a maximal tree such that its vertices are contained in the largest possible number of maximal cliques. The construction ensures that the set of vertices that are not in the tree but adjacent to some vertex of this tree, forms a disconnector set $S(G)$ and the tree is a component of $G\setminus S(G)$. If $\II_G$ is unmixed, then we show that there can be at most one vertex in $S(G)$ that is adjacent only to vertices of the tree. Using this, we prove the main theorem of \new{\Cref{sec:algo}, which states that} if $\II_G$ is unmixed, then any such tree must be a path, \Cref{H_n is path}. All these results are used in the next section to describe those chordal graphs for which $\II_G$ is not unmixed. 

In \Cref{sec:non-unmixed}, we consider those chordal graphs such that all maximal cliques do not have a common point of intersection in $G$. 
Most of this section is a difficult and technical analysis towards showing that if $G\setminus K_{r_i}$ has at least three components, then $\II_{G}$ is not unmixed. As a consequence, if $|V(K_{r_i})| \geq 3$ for some $i$, then $\II_G$ cannot be unmixed. We also show at the end of this section that if $|V(K_{r_i})| = 1$ for all $i$, then $\II_G$ cannot be unmixed.

In \Cref{sec:unmixed}, we complete the characterization of unmixedness of parity binomial edge ideals of non-bipartite chordal graphs. First we show that if all maximal cliques of the given non-bipartite chordal graph $G$ intersect non-trivially, then $\II_G$ is unmixed if and only if it is $K_3$, \Cref{m([t])> 0}. The only chordal graphs that remain to be classified are those where $|V(K_{r_i})| \leq 2$ and $G\setminus K_{r_i} $ has two components. We prove that in this case, $\II_G$ is unmixed only if $G \in \G_1 \cup \G_2 \cup \G_3$, Theorems \ref{I_G unmixed 1}, \ref{I_G unmixed 2}.  To characterize the Cohen-Macaulayness, we need to consider $K_3$ along with graphs in these three classes. We take the help of the computational commutative algebra software, \new{Macaulay2,} \cite{M2}, to prove that the most basic graphs in $\G_1$ and $\G_2$ are not Cohen-Macaulay and the most basic graph in $\G_3$ is Cohen-Macaulay, \Cref{fig:cm_2}. We then finally prove that for a non-bipartite chordal graph $G$, $\II_G$ is Cohen-Macaulay if and only if $G=K_3$ or $G \in \G_3$, \Cref{thm:c-m}.

\section{Preliminaries}\label{sec:prelim}
In this section, we recall all definitions and basic results from graph theory required for the rest of the article.
Throughout this paper $G$ denotes a simple connected graph on finite vertices. The vertex set and the edge set of $G$ are denoted by $V(G)$ and $E(G)$ respectively. A graph $H$ is said to be an induced subgraph of $G$ if $V(H) \subseteq V(G)$ and for $u, v \in V(H)$, $\{u,v\} \in E(H)$ if and only if $\{u,v\} \in E(G)$.  A graph $G$ is said to be \textit{chordal} if $G$ does not contain an induced cycle of length $4$ or more. Chordal graphs are clique sum of complete graphs. 
\begin{definition}\label{def_clique_sum}
  Let $G_1$ and $G_2$ be two graphs. Then the \emph{clique sum} of $G_1$ and $G_2$ represents the graph obtained by identifying the vertices of a clique on $m$ vertices of $G_1$ with a clique of equal size in $G_2$. In such case, we say that $G$ is the clique sum of $G_1$ and $G_2$ along the clique $K_m$, and we denote this by $G = G_1 \cup_{K_m} G_2$. 
  \end{definition}
A simple graph is said to be \textit{cactus} if any two cycles have at most one vertex in common.

Let $H_1$ and $H_2$ be two subgraphs of $G$. Then $H_1\cap H_2$ and $H_1\setminus H_2$ denote the induced subgraph of $G$ on $V(H_1)\cap V(H_2)$ and $V(H_1) \setminus V(H_2)$ respectively. For $S\subseteq V(G)$, $G\setminus S$ denotes the induced subgraph of $G$ on the vertex set $V(G) \setminus S$ and $G[S]$ denotes the induced subgraph of $G$ on the vertex set $S$.
For $S \subseteq V(G)$, we fix the following notation:
\begin{enumerate}
\item $\mathcal{C}_G(S)$ denotes the collection of all connected components of $G \setminus S$.
\item $\mathcal{B}_G(S)$ denotes the collection of all bipartite connected components of $G\setminus S$.
\item $c_G(S) = |\mathcal{C}_G(S)|$.
\item $b_G(S) = |\mathcal{B}_G(S)|$.
\item For $s \in S$, 
$\mathcal{C}_{G\setminus S}(s)$ denotes the collection of all those connected components of $G\setminus S$ which become one connected component in $G\setminus (S\setminus \{s\})$, i.e., all those components that get reconnected when we add $s$ back to $G\setminus S$.

\item \new{ We use the notations $b(G)$ and $c(G)$ to denote the number of bipartite components and total number of components of $G$, respectively.}
\end{enumerate}

\new{We now define cut vertices, cut sets and disconnector sets.  Cut sets determine the minimal primes of a binomial edge ideal while disconnector sets generalize cut sets and are used to describe the minimal primes of $\II_G$. We elaborate these connections after the definitions.}

\begin{definition} \cite{HHHKR}
     A subset $S\subseteq V(G)$ is said to be a cut set of $G$ if for every $s\in S$, $s$ connects at least two components of $G\setminus S$. That is $c_G(S) > c_G({S\setminus \{s\}})$ for every $s\in S$. A vertex $s\in V(G)$ is said to be a cut vertex if $\{s\}$ is a cut set of $G$.
\end{definition}
\new{A \textit{block} of a graph is a maximal connected induced subgraph that has no cut vertex. For example the blocks of a cactus graph are cycles or edges.} 

\begin{definition}\cite[Definition 4.5]{Kahle2016} \label{def: disconnector}
A set $S\subseteq V(G)$ is a disconnector of $G$ if $b_G(S) + c_G(S) > b_G({S\setminus \{s\}}) + c_G({S\setminus \{s\}})$ for every $s\in S$. The empty set is always a disconnector of any graph.
\end{definition}
For a graph $G$ on $[n]$, let $\mathcal{J}_G := \II_G : (\prod_{i\in V(G)} x_iy_i)^\infty \subset K[x_1,\ldots, x_n,y_1,\ldots,y_n].$ Throughout this article, we will assume that char$(K) \neq 2$. For a graph $G$ with an odd cycle, let $\p^+(G) := (x_i+y_i ~:~ i \in V(G))$ and $\p^-(G) := (x_i-y_i ~:~ i \in V(G))$. Let $G$ be a graph with bipartite connected components $B_1,\ldots, B_r$ and non-bipartite connected components $N_1,\ldots, N_t$. In \cite{Kahle2016}, it is proved that the minimal primes of $\mathcal{J}_G$ are of the form $\sum_{i=1}^r \mathcal{J}_{B_i} + \sum_{i=1}^t \p^{\sigma_i}(N_i)$, where $\sigma_i \in \{+,-\}$. Let $S \subset V(G)$ be a disconnector of $G$. Let $\mathfrak{q}= \sum_{i=1}^r \mathcal{J}_{B_i'} + \sum_{i=1}^t \p^{\sigma_i}({N_i'})$ be a minimal prime of $\mathcal{J}_{G\setminus S}$. Then $\mathfrak{q}$ is sign-split if for all $s \in S$ such that $C_{G\setminus S}(s)$ contains no bipartite graphs, the prime summands of $\mathfrak{q}$ corresponding to connected components in $C_{G\setminus S}(s)$ are not all equal to $\p^+$ or not all equal to $\p^-$. Not every disconnector set of a graph admits a sign-split minimal prime of $\mathcal{J}_G$, see \cite[Example 4.14]{Kahle2016}. 
For $S \subset [n]$, let $\mathfrak{m}_S = (x_i, y_i ~:~ i\in S)$. Kahle et al. proved that the minimal primes of $\II_G$ are ideals $\mathfrak{m}_S + \p$, where $S \subset V(G)$ is a disconnector set of $G$ and $\p$ is a sign-split minimal prime of $\mathcal{J}_{G\setminus S}$, see \cite[Theorem 4.15]{Kahle2016}. 
\begin{notation}
\new{A disconnector set $S\subseteq V(G)$ is said to be a sign split disconnector set if $\mathcal{J}_{G\setminus S}$ has a minimal prime, that satisfies the sign-split property}. We denote the collection of all sign-split disconnector sets of $G$ by $\mathcal{D}(G)$.
\end{notation}
\new{
\begin{definition}
The ideal $I\subseteq R$ is said to be unmixed if all its associated primes have the same height.
\end{definition} 
 If $char (K) \neq 2$, then $\II_G$ is radical. So, in this case $\II_G$ has same set of associated and minimal primes. The height of the minimal primes of $\II_G$ are given in \cite{Kum21}. Let $Q_S^\sigma(G) = m_S+ \mathfrak{q}$ be a minimal prime of $\II_G$. Then, height $Q_S^\sigma(G) = |S| + n - b_G(S)$. Therefore, $\II_G$ is unmixed if and only if $b_G(S) = |S| + b(G)$ for all $S\in \mathcal{D}(G)$. 
}
\subsection{Disconnector sets and unmixedness}
While a disconnector set may not always satisfy the sign-split property, some restriction on the set allows it to be sign-split as discussed below. 
\begin{remark}\label{cremark_5}\cite[Remark 4.12]{Kahle2016}
Let $S$ be a disconnector set of $G$ such that every vertex of $S$ is connected with some bipartite connected component of $G\setminus S$. Then $S$ satisfies the sign-split property trivially.
\end{remark}

\begin{proposition}\label{cprop_4}
Let $S$ be a disconnector set of $G$ such that every element of $S$ is adjacent to vertices of a fixed non-bipartite connected component of $G\setminus S$. Then $S$ satisfies the sign-split property.
\end{proposition}
\begin{proof}
Let $Q$ be the non-bipartite connected component of $G \setminus S$ such that $\mathcal{C}_{G\setminus S}(s) \ni Q$ for all $s\in S$. Then corresponding to $Q$, we choose $\mathfrak{p}^+(Q)$ and for any other non-bipartite connected component $Q^\prime$ of $G\setminus S$, we take $\mathfrak{p}^-(Q^\prime)$.
Hence, the sign-split property is preserved.
\end{proof}
For a non-bipartite graph $G$, $\II_G$ is unmixed if and only if $|S| = b_G(S)$ for all $S\in \mathcal{D}(G)$. The following result proposes a necessary condition on $S$ so that it satisfies the above equality.

\begin{proposition} \label{prop on the unmixed property}
Let $G$ be a non-bipartite graph and $S\in \mathcal{D}(G)$. If there exists $T \subseteq S$, such that $\mathcal{I}_{G\setminus T}$ is unmixed and $G\setminus T$ has exactly \new{$|T|$ bipartite components $(b(G\setminus T) = |T|)$, then $|S| = b_G(S)$.}
\end{proposition}
\begin{proof} 
Since $S\in \mathcal{D}(G)$,  $S\setminus T\in \mathcal{D}(G\setminus T)$ and $b_G(S) = b_{(G\setminus T)}(S\setminus T)$. Since $\mathcal{I}_{G\setminus T}$ is unmixed, $|S\setminus T| = b_{(G\setminus T)}(S\setminus T) + b (G\setminus T)$. From the given hypothesis, $b (G\setminus T)= |T|$. So, $|S\setminus T| = b_G(S) + |T| \implies |S|= b_G(S)$.
\end{proof}
\new{
The following remark shows that the study of the unmixed and Cohen-Macaulay properties of $\II_G$ can be reduced to the case of connected graphs.

\begin{remark}
 Let $G = G_1\sqcup G_2 \sqcup \cdots \sqcup G_m$ be a graph, where $G_i$ are the connected components. 
 \begin{itemize}
     \item[(i)] By definition it follows that $\mathcal{D}(G) = \{ S_1\cup S_2\cup \cdots \cup S_m: S_i\in \mathcal{D}(G_i) \}$.

     \item[(ii)] By $(\text{i})$, it follows that $\II_G$ is unmixed if and only if $\II_{G_i}$ is unmixed for all $1\le i \le m$.
     
     \item [(iii)] Note that $\frac{R}{\II_G} \equiv \frac{R_1}{\II_{G_1}} \otimes_K \frac{R_2}{\II_{G_2}}\otimes_K \cdots \otimes_K \frac{R_m}{\II_{G_m}}$, where $R_i = K[x_j,y_j: j\in V(G_i)]$. Therefore $\II_G$ is Cohen-Macaulay if and only if $\II_{G_i}$ is Cohen-Macaulay for each $1\le i \le m$ \cite[Theorem 2.1]{bouchiba2002tensor}. 
     
 \end{itemize}

 \end{remark}
}
For the rest of the article, $G$ denotes a connected simple graph unless stated otherwise.

\section{Cohen-Macaulay parity binomial edge ideal of cactus graphs}\label{cactus}
\new{We study the parity binomial edge ideals of cactus graphs in this section. The classification of  Cohen-Macaulay parity binomial edge ideals associated with bipartite cactus graphs are known from the work of Rinaldo \cite{rinaldo2017cohen}.} So, we assume that our graphs are non-bipartite cactus graphs. The parity binomial edge ideal of an odd cycle is a complete intersection and hence Cohen-Macaulay \cite[Theorem 3.5]{Kum21}. We prove that these are the only Cohen-Macaulay cases among the connected non-bipartite cactus graphs.

First of all, we observe that given any two vertices in a cactus graph $G$, there is some kind of uniqueness in the path joining these two vertices. Let $G$ be a cactus graph and $u, v \in V(G)$. For any path $P$ from $u$ to $v$, let $\mathcal{C}_P = \{C : C \text{ is a cycle in } G \text{ such that } C \cap P \text{ is an edge in }G\}$. Then, the set $\mathcal{C}_P$ is independent of the path $P$. Suppose there are two distinct paths $P_1$ and $P_2$ from $u$ to $v$ such that $\mathcal{C}_{P_1} \neq \mathcal{C}_{P_2}$. If $V(P_1) \cap V(P_2) = \{u=x_0,x_1,\ldots,x_{r-1},x_r=v\}$, then there exists an $i \in \{0,\ldots,r-1\}$ such that for the part $\mathcal{P}_{ij}$ of $P_j$, $j = 1, 2$, joining $x_i$ and $x_{i+1}$, $\mathcal{C}_{\mathcal{P}_{i1}} \neq \mathcal{C}_{\mathcal{P}_{i2}}$. Now consider the induced subgraph on $V(\mathcal{P}_{i1}) \cup V(\mathcal{P}_{i2})\cup (\cup_{C \in \mathcal{C}_{\mathcal{P}_{i1}} \cup \mathcal{C}_{\mathcal{P}_{i2}}} V(C))$ is a block in $G$ which is neither an edge, nor a cycle. This contradicts the assumption that $G$ is a cactus. Therefore, given any two cycles $C$ and $C'$ of $G$, there exist two unique vertices $u \in V(C)$ and $u'\in V(C')$ such that any path connecting $C$ and $C'$ always passes through both $u$ and $u'$. We say that $u$ and $u'$ are the connecting vertices of $C$ and $C'$ in $G$.

\begin{definition}
An odd cycle $C$ of a graph $G$ is said to be pendant if there exists $v\in V(C)$ so that any path in $G$ from a vertex of $C$ to any vertex of an odd cycle in $G \setminus \{v\}$ always passes through $v$. We denote such a pendant odd cycle by $(C,v)$.
\end{definition}

    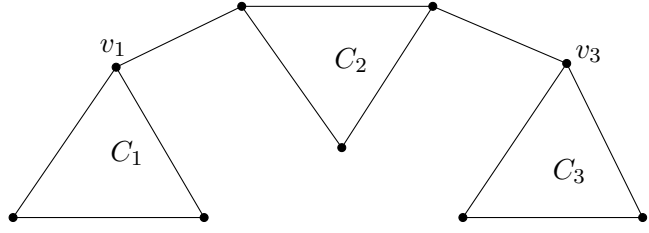
\begin{figure}[H]
        \centering
       \begin{tikzpicture}[line cap=round,line join=round,>=triangle 45,x=1cm,y=1cm]
\clip(-6,-0.6) rectangle (5,4);
\draw  (-3,2)-- (-4.37,0);
\draw  (-3,2)-- (-1.83,0);
\draw  (-4.37,0)-- (-1.83,0);
\draw  (2.99,2.05)-- (1.61,0);
\draw  (2.99,2.05)-- (4,0);
\draw  (1.61,0)-- (4,0);
\draw  (-3,2)-- (-1.33,2.81);
\draw  (-1.33,2.81)-- (1.21,2.81);
\draw  (0,0.93)-- (1.21,2.81);
\draw  (0,0.93)-- (-1.33,2.81);
\draw  (1.21,2.81)-- (2.99,2.05);
\draw (-3.21,1.15) node[anchor=north west] {$C_1$};
\draw (-0.23,2.39) node[anchor=north west] {$C_2$};
\draw (2.67,0.91) node[anchor=north west] {$C_3$};
\draw (-3.35,2.51) node[anchor=north west] {$v_1$};
\draw (2.97,2.43) node[anchor=north west] {$v_3$};
\begin{scriptsize}
\draw [fill=black] (-3,2) circle (1.5pt);
\draw [fill=black] (-4.37,0) circle (1.5pt);
\draw [fill=black] (-1.83,0) circle (1.5pt);
\draw [fill=black] (2.99,2.05) circle (1.5pt);
\draw [fill=black] (1.61,0) circle (1.5pt);
\draw [fill=black] (4,0) circle (1.5pt);
\draw [fill=black] (-1.33,2.81) circle (1.5pt);
\draw [fill=black] (1.21,2.81) circle (1.5pt);
\draw [fill=black] (0,0.93) circle (1.5pt);
\end{scriptsize}
\end{tikzpicture}
        \caption{A cactus graph. Note that $(C_1, v_1)$ and $(C_3, v_3)$ are pendant odd cycles of $G$ but $C_2$ is not.}
        \label{fig:cactus 1}
    \end{figure}
\new{
\begin{remark}
Note that the above definition is similar to, but not quite the same as the standard notion of pendant vertex in graph theory. For example in the above \Cref{fig:cactus 1}, there can be even cycles or edges attached to the vertices of $C_1$ or $C_3$, but they still satisfy the definition of pendant odd cycles. It follows from the definition of a pendant odd cycle that if $G$ has a unique odd cycle, then it is pendant.
Suppose there is more than one odd cycle. Recall from the above discussion that any odd cycle $C$ is connected to another odd cycle $C'$ through a unique vertex $v \in V(C)$. Define $\epsilon(C)$ to be the number of distinct vertices in $V(C)$ that connect $C$ with other odd cycles of $G$. Then $C$ is a pendant odd cycle if and only if $\epsilon(C) = 1$. It follows from the structure of cactus graphs that a pendant odd cycle always exists. We give a proof below.
\end{remark}

\begin{proposition}
Let $G$ be a cactus graph with at least two odd cycles. Then $G$ has a pendant odd cycle.
\end{proposition}

\begin{proof}
We assume that $G$ is connected.
Let $C$ and $C'$ be two different odd cycles in $G$, and $u\in V(C)$ and $u'\in V(C')$ be the connecting vertices. Let us define \[\delta(u,u') = \min \{ n(P): P \text{ is a path between $u$ and $u'$} \}, \] where $n(P)$ is number of distinct odd cycles in $G$ intersecting the path $P$. Consider the set $\{ \delta(u,u') \}$ over all pairs of odd cycles and connecting vertices $u\in V(C)$ and $u'\in V(C')$. Suppose the maximum is achieved for the vertices $v_1\in V(C_1)$ and $v_2\in V(C_2)$.

\noindent
Claim: $(C_1,v_1)$ is a pendant odd cycle.\\
Proof: Suppose it is not. Then there exists another odd cycle $C_3 \neq C_1, C_2$, $v_1\neq v_1'\in V(C_1)$ and $ v_3\in V(C_3)$ such that $C_1$ and $C_3$ are connected through $v_1'$ and $v_3$. Therefore, there is a path connecting $v_3$ in $C_3$ and $v_2$ in $C_2$ passing through $v_1'$ in $C_1$. Since $G$ is a cactus graph, any path from $C_3$ to $C_2$ passes through $C_1$. So, any path from $v_3$ to $v_2$ passes through $v_1'$ and $v_1$. Since $C_3$ intersects any path from $v_3$ to $v_2$ and it does not intersect any path from $v_1$ to $v_2$, we get $\delta(v_2, v_3) \geq \delta(v_2,v_1) + 1 > \delta(v_2, v_1)$. This is a contradiction to the selection of $v_1$ and $v_2$. Hence $(C_1,v_1)$ is a pendant odd cycle.
\end{proof}
}

\begin{notation}\label{notn:pend}
Let $G$ be a cactus graph and $(C,v)$ be a pendant odd cycle of $G$. Note that $v$ is a cut vertex and $\{v\} \in \mathcal{D}(G)$. It follows that $C\setminus \{v\}$ must be contained in a bipartite connected component of $G\setminus \{v\}$, say $G_0$. Write $G \setminus \{v\}= G_0 \sqcup G^{\prime}$. Here $G^{\prime}$ is the union of connected components of $G\setminus \{v\}$ except $G_0$. 
\end{notation}

Our aim is to show that unmixedness descends from $\II_G$ to $\II_{G'}$. For this purpose, we first prove a couple of technical lemmas.

\begin{lemma}\label{int_lemma5}
With the notation as in \Cref{notn:pend}, let $S\subseteq V(G^{\prime})$. Then $c_G(S\cup\{v\}) = c_{G^{\prime}}(S) + 1$ and $b_G(S\cup\{v\}) = b_{G^{\prime}}(S) + 1$.
\end{lemma}
\begin{proof}
Since $G\setminus \{v\} = G_0 \sqcup G'$ and $S \subseteq V(G')$, we have
\[\mathcal{C}_G(S\cup \{v\})= \mathcal{C}_{G\setminus \{v\}}(S) = \mathcal{C}_{G_0\sqcup G^{\prime}}(S) = \mathcal{C}_{G^{\prime}}(S) \sqcup\{ G_0\}.\]
Similarly, $B_G(S\cup \{v\})= B_{G^{\prime}}(S) \sqcup \{ G_0 \} $. Hence the assertion follows.
\end{proof}

\begin{lemma}\label{int_lemma6}
With the notation as in \Cref{notn:pend}, if $S \in \mathcal{D}(G^{\prime})$, then $S\cup \{v\} \in \mathcal{D}(G)$. 
\end{lemma}
\begin{proof}
Let $S\in \mathcal{D}(G^{\prime})$. Therefore $S$ is a sign-split disconnector set of $G^{\prime}$ and for every $s\in S$, $c_{G^{\prime}}(S) + b_{G^{\prime}}(S) > c_{G^{\prime}}(S\setminus\{s\}) + b_{G^{\prime}}(S\setminus\{s\})$.
Then it follows from \Cref{int_lemma5} that for every $s \in S,$
\[c_G(S\cup \{v\}) + b_G (S\cup \{v\}) > c_{G}((S\cup\{v\})\setminus\{s\}) + b_{G}((S\cup \{v\})\setminus\{s\}).\]
Since $G_0$ is bipartite and $G$ is non-bipartite, 
$c_G(S\cup \{v\}) + b_G (S\cup \{v\}) > c_{G}(S) + b_{G}(S) $. We now prove that $S\cup \{v\}$ satisfies the sign-split property. Note that $S$ has a sign-split property in $G^\prime$. Let $s\in S\cup\{v\}$. If $s\neq v$, then $\mathcal{C}_{G\setminus (S\cup \{v\})}(s) = \mathcal{C}_{G^\prime \setminus S}(s)$. If $s = v$, then $\mathcal{C}_{G\setminus (S\cup \{v\})}(s) = \mathcal{C}_{G^\prime \setminus S}(s) \cup \{G_0\}$, and $G_0$ is bipartite. Hence in both cases $S\cup \{v\}$ satisfies sign-split property in $G$. Therefore $S\cup \{v\} \in \mathcal{D}(G)$. 
\end{proof}

We now prove that the unmixedness property descends.
\begin{proposition}\label{int_prop2}
With the notation as in \Cref{notn:pend}, if $\mathcal{I}_G$ is unmixed, then so is $\mathcal{I}_{G^{\prime}}$.
\end{proposition}
\begin{proof} 
Since $\{v\} \in \mathcal{D}(G)$ and $\mathcal{ I}_G$ is unmixed, $b_G(\{v\}) =1$. Hence the connected components of $G \setminus \{ v \}$ other than $G_0$ must be non-bipartite.
In particular, $G^{\prime}$ must be non-bipartite. 
It follows from \Cref{int_lemma6} that for all $ S\in \mathcal{D}(G^{\prime})$, $S\cup \{v\}\in \mathcal{D}(G)$. As $\mathcal{I}_G$ is unmixed, $b_G(S\cup \{v\})=|S\cup \{v\}|$. This implies $b_{G^{\prime}}(S) = |S|$, by \Cref{int_lemma5}. Thus $\II_{G'}$ is unmixed.
\end{proof}
We now proceed to characterize the cactus graphs having unmixed parity binomial edge ideals.
\begin{proposition}
\label{int_lemma7}
Let $G$ be a non-bipartite cactus graph with only one odd cycle. Then $\mathcal{I}_G$ is unmixed if and only if $G$ is an odd cycle.
\end{proposition}
\begin{proof}
The `if part' follows from \cite[Theorem 3.5]{Kum21}. Now assume that $\II_G$ is unmixed. Let $C$ be the odd cycle in $G$. If there exists a $v \in V(C)$ such that $\deg_G( v) > 2$, then $v$ is a cut vertex. Thus $\{v\} \in \mathcal{D}(G)$. Moreover, $G \setminus \{v\}$ is bipartite and has at least two components, i.e., $b_G(\{v\}) \geq 2$. This contradicts the unmixedness of $\II_G$. Hence for every $v\in V(C)$, $\deg_G( v) = 2$. Therefore, $G=C$. 
\end{proof}
We now show that if a cactus graph contains more than one odd cycle, then the parity binomial edge ideal is not unmixed.
\begin{theorem}\label{int_lemma8}
Let $G$ be a cactus graph having at least two odd cycles as subgraphs, then $\mathcal{I}_G$ is not unmixed.
\end{theorem}
\begin{proof}
We prove the statement by induction on $r(G)$, the number of odd cycles in $G$. First assume that $r(G)=2$ and $\mathcal{I}_G$ is unmixed. Since $G$ is cactus with only two odd cycles, both must be pendant odd cycles, say $(C_1,v_1)$ and $(C_2,v_2)$. If $v_1 = v_2$, then $\{v_1\} \in \D(G)$ and $b_G(\{v_1\}) = 2$ which is a contradiction to the unmixedness of $\II_G$. Suppose $v_1 \neq v_2$.  First, we take the disconnector set $\{v_1\}$. Suppose $G\setminus \{v_1\} = G_0 \sqcup G^{\prime}$, where $G_0$ is the bipartite component containing $C_1 \setminus\{v_1\}$ and $G^{\prime}$ is the non-bipartite graph containing $C_2$. Since $b_G(\{v_1\}) = 1$ and since there are only two odd cycles in $G$, $G'$ is also connected. \new{Moreover,} \Cref{int_prop2} implies that $\mathcal{I}_{G^ \prime}$ is unmixed. But $G^{\prime}$ is a connected cactus graph with only one odd cycle. So, from \Cref{int_lemma7}, it follows that $G^{\prime}$ is an odd cycle, i.e., $G^{\prime} = C_2$. Using the same argument for $\{v_2\}$, we get $G\setminus \{v_2\} = (C_2\setminus \{v_2\}) \sqcup C_1$. So, $V(G)= V(C_1) \cup V(C_2)$ and $E(G) = E(C_1) \cup E(C_2) \cup \{e\}$, where $e=\{v_1,v_2\}$ is the only edge between $C_1$ and $C_2$.
Choose $x_1\in V(C_1) \setminus \{v_1\}$ and $x_2\in V(C_2) \setminus \{v_2\}$. Then $\{x_1,x_2\} \in \mathcal{D}(G)$, but $b_G(\{x_1,x_2\}) = 1$, a contradiction to the unmixedness of $\II_G$. Therefore $\mathcal{I}_G$ is not unmixed. 

Now let $r(G)\geq 3$ and assume by induction that if $H$ is a connected non-bipartite cactus graph with $2\leq r(H) < r(G)$, then $\II_H$ is not unmixed. Suppose $\mathcal{I}_G$ is unmixed. Let $(C, v)$ be a pendant odd cycle of $G$. Then $\{v\}\in \mathcal{D}(G).$ If $v$ is contained in more than one pendant odd cycles, then $b_G(\{v\}) > 1$ which contradicts the unmixedness of $\II_G$. 
Thus we can write $G\setminus \{v\} = G_0 \sqcup G'$, where $G_0$ is the bipartite component containing $C\setminus \{v\}$ and $G'$ is a non-bipartite  graph with each component a non-bipartite cactus graph. Observe that $G'$ will not have any bipartite component, since otherwise $b_G(\{v\}) > 1$ contradicting the unmixedness of $\II_G$. Note that for any non-bipartite connected component $H$ of $G'$, $r(H) < r(G)$. We know from \Cref{int_prop2} that $\mathcal{I}_{G'}$ is unmixed. Now we have two cases.
\vskip 2mm
\noindent\textbf{Case 1:} If any of the connected components of $G'$ has at least two odd cycles, then by induction hypothesis $\mathcal{I}_{G'}$ is not unmixed, which contradicts \Cref{int_prop2}. 
\vskip 2mm
\noindent \textbf{Case 2:} Suppose every non-bipartite connected component of $G'$ has only one odd cycle. Since $\II_G$ is unmixed, so is $\II_{G'}$, by \Cref{int_prop2}. Therefore, by \Cref{int_lemma7}, we get that all these connected components of $G'$ are cycles. Write  $G' = C_1\sqcup C_2 \sqcup \cdots \sqcup C_k$, where $C_i$ are odd cycles. Let $v_i \in V(C_i)$ be such that $\{v,v_i\} \in E(G)$. Since $G$ is a cactus graph, these $v_i$ are the unique vertices of $C_i$ that are adjacent to $v$. Hence $(C_i,v_i)$ are all pendant odd cycles of $G$, for $1\leq i\leq k$. Now consider $G\setminus \{v_1\} = G_0^\prime \sqcup G''$, where $G_0'$ is the bipartite component containing $C_1 \setminus\{v_1\}$ and $G''$ is the non-bipartite connected component containing $C,C_2,\ldots, C_k$. Since $r(G) \geq 3$, $k \geq 2$. Hence by induction $\II_{G''}$ is not unmixed. But this contradicts the conclusion by \Cref{int_prop2} that $\II_{G''}$ is unmixed. Hence $\II_G$ is not unmixed.
\end{proof}
Summarizing all the above results, we get:
\begin{corollary} \label{cm Cactus graph}
Let $G$ be a non-bipartite cactus graph. Then the following are equivalent:
\begin{enumerate}
    \item $\II_G$ is unmixed;
    \item $\II_G$ is Cohen-Macaulay;
   \new{ \item $\II_G$ is Gorenstein;
   \item $\II_G$ is a complete intersection;}
    \item $G$ is an odd cycle.
\end{enumerate}
\end{corollary}
\new{\begin{proof}
The implications $(4) \Rightarrow (3) \Rightarrow (2) \Rightarrow (1)$ are always true. By \cite{Kum21}, $(4)$ and $(5)$ are equivalent. The implication $(1) \Rightarrow (5)$ follows from \Cref{int_lemma7} and \Cref{int_lemma8}.
\end{proof}}

\section{The algorithm}\label{sec:algo} 

 \new{Unlike cactus graphs, it is not clear how to study the unmixed property of $\II_G$ using hereditary behavior for chordal graphs.
  To understand the unmixed property, we first propose an algorithm for chordal graphs to construct a maximal induced tree $H_{\n}$ along with a disconnector set $S^{\mathrm{n}}(G)$. We show that if $\II_G$ is unmixed, then these trees are always path graphs.
\begin{notation}\label{not:chordal}
Let $G$ be a chordal graph which is a clique sum of $K_{n_1}, \ldots, K_{n_t}.$ We define $K_{r_j} = \left( \cup_{i < j} K_{n_i} \right) \cap K_{n_j}$ for $j \geq 2$, i.e., the graph $G$ is obtained by attaching $K_{n_j}$ on $\cup_{i < j} K_{n_i}$ along $K_{r_j}$.
\end{notation}

We fix the above notation for the rest of the paper.}
\begin{definition} \label{definition of lambda}
\new{Let $j\in \{ 2, \ldots, t\}$.} Then, $\lambda(j) = \{i<j: K_{r_j} \subset K_{n_i}\}$. 
\end{definition}
 \noindent Thus, $\displaystyle{K_{r_j} =K_{n_j}\bigcap_{l\in \lambda(j)} K_{n_l} }$. Without loss of generality we may assume that $K_{r_j} \neq K_{n_i}$ for every $i \leq j$.

\begin{remark} \label{cRemark_1}
Suppose $i,j\in [t]$ such that $i<j$, then \[ K_{n_i} \cap K_{n_j} = K_{n_i} \cap K_{r_j} = K_{n_i} \cap (K_{n_j}\bigcap_{l\in \lambda(j)} K_{n_l}). \]
\end{remark}

\begin{lemma}
For any $i\in [t]\setminus \{1\}$, $V(K_{r_i}) \in \mathcal{D}(G)$. 
\end{lemma}
\begin{proof}
Suppose $j\in \lambda(i)$. Let $G_i$ and $G_j$ be the connected components of $G\setminus V(K_{r_i})$ containing $V(K_{n_i})\setminus V(K_{r_i})$ and $V(K_{n_j}) \setminus V(K_{r_i})$ respectively. Then for all $v\in V(K_{r_i})$, $v$ connects both $G_1$ and $G_2$. So, if any one of them is bipartite, then by \Cref{cremark_5}, $V(K_{r_i}) \in \mathcal{D}(G)$. If both are non-bipartite, then by fixing either $G_1$ or $G_2$ and applying \Cref{cprop_4}, we get $V(K_{r_i}) \in \mathcal{D}(G)$.
\end{proof}
\begin{proposition} \label{x and y adjacent}
    Let $i, j \in [t]$ be such that $i<j$. Then for any $x\in V(K_{n_i})$ and $y\in  V(K_{n_j})$, if $x\neq y$ and $\{x,y\} \in E(G)$, then either $x$ or $y$ must belong to $K_{r_j}$.
\end{proposition}

\begin{proof}
    Let $H$ be the induced subgraph of $G$, which is the clique sum of $K_{n_i}$ up to $j$th. Then, the vertices of $K_{n_j}$ which are adjacent to other vertices of $H$, must belong to $K_{r_j}$.
\end{proof}

Now we introduce three subclasses of chordal graphs $\mathfrak{G}_1$, $\mathfrak{G}_2$ and $\mathfrak{G}_3$  as described in the  \Cref{fig:fig 13}. Here $P_1$, $P_2$ and $P_3$ are arbitrary path graphs.

\begin{theorem} \label{I_G is unmixed only if G has there forms}
    Let $G$ be a chordal graph. Then $\II_G$ is unmixed if $G$ is any one of the following:
    \begin{enumerate}
        \item a path graph.
        \item $K_3$
        \item $G\in \mathfrak{G}_1 \cup \mathfrak{G}_2 \cup \mathfrak{G}_3$.
    \end{enumerate}
\end{theorem}

 \begin{proof}
If $G$ is bipartite, then it must be a tree graph. In this case $\II_G \equiv J_G$. So, $\II_G$ is unmixed if and only if $J_G$ is unmixed if and only if $G$ is a path graph \cite[Corollary 1.2]{EHHNMJ}. Now we take $G$ to be non-bipartite. 
If $G= K_3$, then $\II_G$ is unmixed by \cite[Theorem 3.5]{Kum21}. Next we show that if $G$ obtains the structure of \Cref{fig:fig 13}, then $\II_G$ is unmixed. 
For that purpose, it is sufficient to prove that $b_G(S) = |S|$ for all $S\in \mathcal{D}(G)$. Note that if $S\in \mathcal{D}(G)$ and $T\subseteq S$, then $S\setminus T \in \mathcal{D}(G\setminus T)$. 

\noindent \textbf{Case 1:} \new{Let $G\in \mathfrak{G}_3$. Then by \cite[Theorem 4.13 (i)]{Kum21}, $\frac{R}{\II_G}$ is Cohen-Macaulay. Therefore, $\II_G$ is unmixed.} 

\noindent \textbf{Case 2:} Let $G\in \mathfrak{G}_2$ and $S\in \mathcal{D}(G)$.

\noindent \textbf{$x_1\in S$ or $y_1\in S$:} Then both $G\setminus \{x_1\}$ and $G\setminus \{y_1\}$ have exactly one bipartite component. By Case 1, $\II_{G\setminus \{x_1\}}$ and $\II_{G\setminus \{y_1\}}$ are unmixed. So, $b_G(S) = |S|$ (\Cref{prop on the unmixed property}).\\
\noindent \textbf{ $x_1, y_1 \notin S$:} Since $S$ is a disconnector set, $\alpha_1\in S$ if and only if $\alpha_2 \in S$.
 Suppose both $\alpha_1$, $\alpha_2$ belong to $S$. Now $G\setminus \{\alpha_1, \alpha_2\}$ is union of a path graph and a singleton graph. So, by \Cref{prop on the unmixed property}, $b_G(S) = |S|$.
 Now if $\alpha_1,\alpha_2 \notin S$, then $S\subset V(P_2)\cup V(P_3)$. This implies that $S$ contains \new{no} consecutive vertices of $P_2$ and $P_3$. \new{Therefore, $b(P_2\setminus S) = c(P_2\setminus S)= |P_2\cap S|+1$ and $b(P_3\setminus S) = c(P_3\setminus S)=|P_3\cap S|+1$. Note that, in $G\setminus S$ two components of $P_2\setminus S$ and $P_3\setminus S$ containing the end vertices $y_1$ and $x_1$ are contained in a single non-bipartite component.} Hence $b_G(S) = |S|$. \\

 \noindent \textbf{Case 3:} Let $G\in \mathfrak{G}_1$ and $S\in \mathcal{D}(G)$.\\
\noindent \textbf{ $\{\beta_i$, $\beta_j\}$ $\subseteq S$ for some $i,j\in [3]$: }
 Then $G\setminus \{\beta_i$, $\beta_j\}$ is union of two path graphs and one component in the class $\mathfrak{G}_1$. Hence by \Cref{prop on the unmixed property}, with $T= \{\beta_i, \beta_j\}$ we get $b_G(S) = |S|$.\\ 
\noindent \textbf{$\beta_i \in S$ for exactly one $i\in [3]$ :} Without loss of generality assume that $\beta_1 \in S$ and $\beta_2, \beta_3 \notin S$. If $\alpha_3 \in S$, then taking $T = \{ \beta_1, \alpha_3\}$ in \Cref{prop on the unmixed property} we get $b_G(S) = |S|$. So, now we assume that $\alpha_3 \notin S$. Note that if $\alpha_1, \alpha_2 \in S$ then taking $T= \{\alpha_1, \alpha_2\}$ in \Cref{prop on the unmixed property} we get $b_G(S) = |S|$. So, \new{we may assume that} both $\alpha_1$ and $\alpha_2$ do not belong to $S$. Since $\alpha_3$, $\beta_2$ or $\beta_3$ do not belong to $S$, $\alpha_1 \in S$ if and only if $\alpha_2 \in S$. Therefore, $\alpha_1$, $\alpha_2$, $\alpha_3$, $\beta_2$, $\beta_3$ do not belong to $S$. Thus $S\subset V(P_1 \cup P_2 \cup P_3)$. Hence $b_G(S) = |S|$.\\
\noindent \textbf{$\beta_1, \beta_2, \beta_3 \notin S$:} Then $\alpha_i \in S$ if and only if $\alpha_j\in S$ for $j\neq i$ and $i,j \in [3]$. Suppose $\{ \alpha_i, \alpha_j\} \subseteq S$. Then taking $T= \{ \alpha_i, \alpha_j\}$ in \Cref{prop on the unmixed property} we get $b_G(S) = |S|$. Suppose $\alpha_1$, $\alpha_2$, $\alpha_3$ do not belong to $S$. Since we already know $\beta_1 \beta_2, \beta_3 \notin S$, $S \subseteq  V(P_1 \cup P_2 \cup P_3).$ Hence $b_G(S) = |S|$.    
\end{proof}

The above theorem gives us some classes of unmixed parity binomial edge ideals. Later, we show that those are the only possible classes \new{among chordal graphs}. The unmixed property of $\mathcal{I}_G$ depends on the sign-split disconnector sets of $G$. Unlike binomial edge ideals, for parity binomial edge ideals of chordal graphs, the disconnector sets does not behave well under the graph operations as studied in \cite{BMS22}. Also, we do not have a short exact sequence as given in the papers \cite{EHHNMJ} and \cite{AR1} which enables one to use induction on certain related graphs. In short, what we see is that the standard techniques that work for the binomial edge ideals are not really useful for parity binomial edge ideals. So, to understand the behavior of disconnector sets, we construct an algorithm based on the structure of the graph and this algorithm leads us to some specific disconnector sets. These disconnector sets satisfy certain combinatorial properties which play an important role in characterizing $G$ when $\II_G$ is unmixed.
    
To simplify further exposition, we define the following function $m$. 
\[  m: \mathcal{P}([t]) \rightarrow \mathbb{N} \cup \{0\}, \quad 
\qquad \text{where} \qquad 
 m(\Gamma) = |V(\cap_{i\in \Gamma}K_{n_i})| .\] 
\new{Here $\mathcal{P}([t])$ represents the power set of $[t]= \{1, 2, \ldots, t\}$.}

\subsection{Algorithm:}\label{algo} 
\new{Now we describe the algorithm below. Recall that $G$ is a clique sum of $K_{n_1}, \ldots, K_{n_t}.$ At each step and each iteration of the algorithm, we choose a maximal subset of $\{1,\ldots, t\}$ such that it takes a nonzero value under the function $m$. Then we choose a vertex from the intersection of all complete graphs corresponding to that set. The construction of the algorithm ensures that all these chosen vertices form a connected subgraph of $G$. Later on we show that this connected graph is an induced tree and it is a maximal tree of $G$. To get a better understanding of the algorithm, we recommend the reader to follow Example \ref{example 1} alongside the algorithm:}

\vskip 2mm
\noindent
\textbf{Global Input:} The function $m$ and the graph $G$.
\begin{align*}
\textbf{Step $0$:} \qquad &  \textit{Iteration $1$}: \text{Define} \\
& T_1^0 = [t], \\
& \Gamma_1^0 = \{L\subseteq [t] : m(L) >0 \}, \\
& M(\Gamma_1^0) = \{ \text{maximal elements of}~ \Gamma_1^0 ~ \text{with respect to the order relation $\subseteq$} \}.\\
& \text{Fix }  L_1^0 \in M(\Gamma^0_1). \\
& \textit{Outputs of iteration $1$:} \quad L_1^0.\\
& \textit{Proceed to Iteration $2$ in Step $0$.}
\end{align*}

\begin{align*}
& \textit{Iteration $q\geq 2$}: \quad \text{Inputs:} \quad L_i^0,~ i<q.\\
& \text{Define} \\
& T_q^0 = T_{q-1}^0 \setminus L_{q-1}^0 = [t]\setminus \big( \bigcup_{j=1}^{q-1} L_j^0 \big) , \\
& \Gamma_q^0 = \{ L\subseteq [t] : L \cap T_q^0 \neq \emptyset, ~ m(L) >0\},\\
& M(\Gamma_q^0) = \{ \text{maximal elements of}~ \Gamma_q^0 ~ \text{with respect to the order relation $\subseteq$} \} , \\
& M_1(\Gamma_q^0) = \{ L\in M(\Gamma_q^0): L\cap L_i^0 \neq \emptyset ~\text{for exactly one $i<q$} \} , \\
& \text{If } M_1(\Gamma_q^0) = \emptyset , \text{then the } 0^{th} \text{ step terminates}.\\
&\text{Else } M_2(\Gamma_q^0) = \{ L\in M_1(\Gamma_q^0): |L\cap L_i^0| ~ \text{is minimum}\}.\\
& \text{Choose $L_q^0 \in  M_2(\Gamma_q^0)$ and $\displaystyle{x_q^0 \in \bigcap_{i\in L_q^0}V(K_{n_i})}$.}\\
& \textit{Outputs of Iteration $q$:}\quad L_q^0. \\
& \textit{Proceed to Iteration $q+1$ in Step $0$.}
\end{align*}
Once the $0^{th}$ step terminates, define
\begin{align*}
D_0 & = \{ 1 \} \cup \{ q: M_1(\Gamma_q^0) \neq \emptyset \} ~\text{and}~ H_0 = G[\{x_q^0: q\in D_0\}].\\
A_2^0 & = \{ i\new{\in [t]}: i\in (L_{q_1}^0 \cap L_{q_2}^0) ~ \text{for } q_1 \neq q_2 \in D_0 \}, \\
A_1^0 & = \{ i\new{\in [t]}: i\in L_q^0 ~ \text{ for exactly one } q\in D_0\}, \\
A_0^0 & = \{ i\new{\in [t]}: i\notin L_q^0 ~ \text{for all } q\}.
\end{align*}  

\textbf{Outputs of Step $0$:} $A_2^0, A_1^0, A_0^0, H_0, D_0$.   \\
Proceed to Step $1$.

\begin{align*}
\textbf{Step $p$ for $p\geq 1$:} \qquad & \text{Inputs: } A_0^{p-1}, A_1^{p-1}, A_2^{p-1}. \\
& \text{Define}\\
& T_0^p = A_1^{p-1} \cup A_0^{p-1}, \\
& L_0^p = \emptyset, \\
& \textit{Iteration $q, q\geq 1:$} \quad \text{Inputs:} \quad L_i^p,~ i<q.\\ 
& \text{Define}\\
& T_q^p = T_{q-1}^p \setminus L_{q-1}^p \\
& \Gamma_q^p =  \{ L\subseteq T_q^p: L\cap A_1^{p-1} \neq 0 ~\text{and}~ \bigcap_{i\in L}V(K_{n_i}) \setminus \bigcup _{j\in A_2^{p-1}\cup_{s=0}^{q-1} L_{s}^p} V(K_{n_j}) \neq \emptyset \}. \\
& \text{If } \Gamma_q^p = \emptyset , \text{then the } p^{th} \text{ step terminates} . \\
& \text{Else } M(\Gamma_q^p)= \{ \text{maximal elements of}~ \Gamma_q^p ~ \text{with respect to the order relation } \subseteq \},\\
& \text{Choose $L_q^p \in M(\Gamma_q^p)$
and $\displaystyle{x_q^p \in \bigcap_{i\in L_q^p}V(K_{n_i}) \setminus \bigcup _{j\in A_2^{p-1}\cup_{s=0}^{q-1} L_{s}^p} V(K_{n_j}})$.}\\
& \textit{Outputs of Iteration $q$:}\quad L_q^p. \\
& \textit{Proceed to Iteration $q+1$ in Step $p$.}
\end{align*}
Once the $p^{th}$ step terminates, define
\begin{align*}
D_p &= \{q: M(\Gamma_q^p) \neq \emptyset \}
~\text{and}~ H_p = G[\{x^i_q: i\in [p], q\in \cup _{j\in [p]}D_j\}].\\
A_2^p & = A_2^{p-1} \cup A_1^{p-1}, \\
A_1^p & = \{ i\in A_0^{p-1} : i \in L_q^p ~ \text{ for some}~  q\in D_p\}, \\
 A_0^p & = [t] \setminus (A_2^p\cup A_1^p ). 
\end{align*}
 \textbf{Outputs of Step $p:$} $A_2^p, A_1^p, A_0^p, D_p, H_p$. \\
 If $A_1^{p} = \emptyset$, then the algorithm terminates. Else proceed to Step $p+1$.

\begin{remark} \label{remarks on algorithm}
    \begin{enumerate}
    \item Note that the sets $A_0^p$, $A_1^p$, $A_2^p$, the graphs $H_p$, and the number of steps of the algorithm will vary depending on the choices of $L_q^p$ and $x_q^p$. Let $\mathcal{ L}$ be the collection of all data in one run of the algorithm with fixed choices of $L_q^p$, which we denote by $L_q^p (\mathcal{L})$, and $\mathrm{n}(\mathcal{L})$ denote the last step of the algorithm for these fixed choices. Specifically $\mathcal{L}$ represents one run of iteration of the algorithm. If $\mathcal{L}$ is fixed, then we use $\n$ instead of $\n(\mathcal{L})$ and $L_q^p$ instead of $L_q^p (\mathcal{L})$. 
    
        \item Observe that $A_0^{p} \sqcup A_1^{p} = A_0^{p-1}$ and $A_2^p, A_1^p$ and $A_0^p$ is a partition of $[t]$.

        \item At each step of the algorithm we construct a subgraph $H_p$ of $G$ and for $p< p^\prime$, $H_p$ is a subgraph of $H_{p^\prime}$. 
    \end{enumerate}
\end{remark}

\new{
\begin{example} \label{example 1}
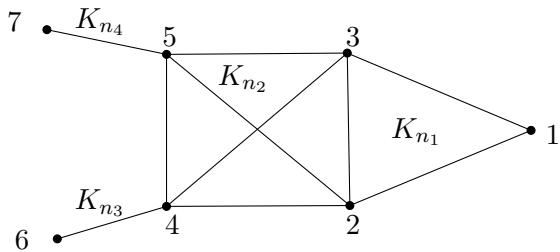
\begin{figure}[H]
    \centering
\begin{tikzpicture}[line cap=round,line join=round,>=triangle 45,x=1cm,y=1cm]
\clip(-7.5,-1.5) rectangle (2,4);
\draw  (-3.8463466690239216,2.011325981637946)-- (-3.8463466690239216,-0.011200555919785929);
\draw  (-3.8463466690239216,2.011325981637946)-- (-1.4435159961989221,2.0286125332410014);
\draw  (-1.4435159961989221,2.0286125332410014)-- (-1.408942892992807,0);
\draw  (-3.8463466690239216,-0.011200555919785929)-- (-1.408942892992807,0);
\draw  (-3.8463466690239216,2.011325981637946)-- (-1.408942892992807,0);
\draw  (-3.8463466690239216,-0.011200555919785929)-- (-1.4435159961989221,2.0286125332410014);
\draw  (-1.4435159961989221,2.0286125332410014)-- (1,1);
\draw  (-1.408942892992807,0)-- (1,1);
\draw  (-3.8463466690239216,-0.011200555919785929)-- (-5.298417003680756,-0.443364345996224);
\draw  (-3.8463466690239216,2.011325981637946)-- (-5.4367094165052166,2.339770462096037);
\draw (-5.2,2.76594230122228) node[anchor=north west] {$K_{n_4}$};
\draw (-5.2,0.35) node[anchor=north west] {$K_{n_3}$};
\draw (-3.3,2) node[anchor=north west] {$K_{n_2}$};
\draw (-0.9940656545194256,1.3) node[anchor=north west] {$K_{n_1}$};
\draw (1.0630339862444231,1.2) node[anchor=north west] {$1$};
\draw (-1.59909496062644,0) node[anchor=north west] {$2$};
\draw (-1.596455,2.5) node[anchor=north west] {$3$};
\draw (-4,0) node[anchor=north west] {$4$};
\draw (-4.036498736657555,2.5) node[anchor=north west] {$5$};
\draw (-6,-0.2) node[anchor=north west] {$6$};
\draw (-6.1,2.7) node[anchor=north west] {$7$};

\begin{scriptsize}
\draw [fill=black] (-3.8463466690239216,2.011325981637946) circle (1.5pt);
\draw [fill=black] (-3.8463466690239216,-0.011200555919785929) circle (1.5pt);
\draw [fill=black] (-1.4435159961989221,2.0286125332410014) circle (1.5pt);
\draw [fill=black] (-1.408942892992807,0) circle (1.5pt);
\draw [fill=black] (1,1) circle (1.5pt);
\draw [fill=black] (-5.298417003680756,-0.443364345996224) circle (1.5pt);
\draw [fill=black] (-5.4367094165052166,2.339770462096037) circle (1.5pt);
\end{scriptsize}
\end{tikzpicture}

    \caption{The smallest graph $G$ in the class $\mathfrak{G}_2$}
    \label{fig:exam_1}
\end{figure}  

Let $G$ be the clique sum of $K_{n_1}$, $K_{n_2}$, $K_{n_3}$ and $K_{n_4}$ as shown in \Cref{fig:exam_1}.\\
\textbf{Step $0$:} Iteration $1$: $T_1^0  = \{1,2,3,4\},$ So, $\Gamma_1^0 = \{ \{1\}, \{2\}, \{3\}, \{4\}, \{1,2\}, \{2,3\}, \{2,4\} \}$. This implies $M(\Gamma_1^0) = \{ \{1,2\}, \{2,3\}, \{2,4\}\}$.\\
Choose $L_1^0 = \{1,2\} \in M(\Gamma_1^0)$ and $x_1^0 = 2 \in  V(K_{n_1})\cap V(K_{n_2})$.\\
\textit{Outputs of Iteration 1:} $L_1^0 = \{ 1,2\}$.\\ 

Iteration $2$: Inputs: $L_1^0$.\\
$T_2^0  = T_{1}^0 \setminus L_{1}^0 = \{ 3,4\}$. So, $\Gamma_2^0 = \{ \{3\}, \{4\}, \{2,3\}, \{2,4\}$ and $M(\Gamma_2^0) = \{ \{2,3\}, \{2,4\} \}$. Therefore, $M_1(\Gamma_q^0) = \{ \{2,3\}, \{2,4\} \}$.
We Choose $L_2^0 = \{2,3\}$ and $x_2^0 = 4 \in  V(K_{n_2})\cap V(K_{n_3})$.\\ 
\textit{Outputs of Iteration 2:} $L_2^0 = \{ 2,3\}$.\\

Iteration $3$: Inputs: $L_1^0, L_2^0$.\\
$T_3^0  = T_{2}^0 \setminus L_{2}^0 = \{ 4\}$ and $M(\Gamma_3^0) = \{ \{2,4\} \}$ but $M_1(\Gamma_3^0) = \emptyset$.  So, Step $0$ terminates.\\

\begin{figure}[H]
\begin{minipage}{0.45\textwidth}
\textbf{Outputs of Step $0$:} $D_0 = \{1,2\}$, $H_0 = G[\{x_q^0: q\in D_0\}] = G[\{2,4\}]$.
$A_2^0 = \{ 2\}$, $A_1^0 = \{ 1,3 \}$ and $A_0^0 = \{4\}$,
\end{minipage}
\hfill
\begin{minipage}{0.45\textwidth}
\centering
   \begin{tikzpicture}[scale=0.9, line cap=round, line join=round, >=triangle 45, x=1cm, y=1cm]

\clip(-6.1,-1.4) rectangle (2,3.8);
\tikzset{every path/.style={black, line width=1pt, dash pattern=on 1pt off 3.5pt}}
\draw (-3.8463466690239216,2.011325981637946)-- (-1.4435159961989221,2.0286125332410014);
\draw (-1.4435159961989221,2.0286125332410014)-- (-1.408942892992807,0);
\draw (-1.4435159961989221,2.0286125332410014)-- (1,1);
\draw (-1.408942892992807,0)-- (1,1);
\draw (-3.8463466690239216,-0.011200555919785929)-- (-5.298417003680756,-0.443364345996224);
\draw (-3.8463466690239216,2.011325981637946)-- (-5.4367094165052166,2.339770462096037);

\draw[black, line width=1.5pt, dash pattern=] (-1.408942892992807,0) -- (-3.8463466690239216,-0.011200555919785929);

\draw (-3.8463466690239216,2.011325981637946)-- (-3.8463466690239216,-0.011200555919785929);
\draw (-3.8463466690239216,2.011325981637946)-- (-1.408942892992807,0);
\draw (-3.8463466690239216,-0.011200555919785929)-- (-1.4435159961989221,2.0286125332410014);

\draw (-5.2,2.76594230122228) node[anchor=north west] {$K_{n_4}$};
\draw (-5.2,0.35) node[anchor=north west] {$K_{n_3}$};
\draw (-3.3,2) node[anchor=north west] {$K_{n_2}$};
\draw (-0.9940656545194256,1.3) node[anchor=north west] {$K_{n_1}$};
\draw (-1.59909496062644,0) node[anchor=north west] {$  2$};
\draw (-4,0) node[anchor=north west] {$4$};

\begin{scriptsize}
\draw [fill=black] (1,1) circle (1pt);
\draw [fill=black] (-5.298417003680756,-0.443364345996224) circle (1pt);
\draw [fill=black] (-5.4367094165052166,2.339770462096037) circle (1pt);
\draw [fill=black] (-3.8463466690239216,2.011325981637946) circle (1.5pt);
\draw [fill=black] (-3.8463466690239216,-0.011200555919785929) circle (2.5pt);
\draw [fill=black] (-1.4435159961989221,2.0286125332410014) circle (1pt);
\draw [fill=black] (-1.408942892992807,0) circle (2.5pt);
\end{scriptsize}

\end{tikzpicture}
    \caption{Step $0$: $H_0$}
    \label{fig:step_0}
   \end{minipage} 
\end{figure}
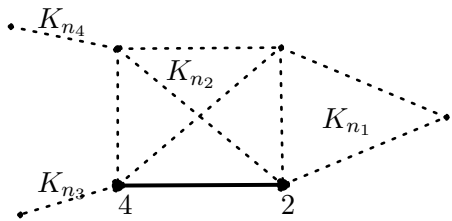

\textbf{Step $1$:} Inputs: $A_2^0, A_1^0, A_0^0.$\\ 
Iteration $1$: $T_1^1 = A_1^{0} \cup A_0^{0} = \{ 1,3,4 \}$. Then  $\Gamma_1^1= M(\Gamma_1^1) = \{ \{1\}, \{3\} \}.$\\
Choose $L_1^1 = \{1\}$ and $x_1^1 = 1\in V(K_{n_1}) \setminus V(K_{n_2})$.\\
\textit{Outputs of Iteration 1:} $L_1^1 = \{1\}$\\

Iteration $2$: Inputs: $L_1^1$.\\
$T_2^1 = T_1^1\setminus L_1^1 = \{3,4\}$. Then $\Gamma_2^1= M(\Gamma_2^1) = \{\{3\}\}$.\\
Choose  $L_2^1 = \{3\}$ and $x_2^1= 6$.\\
\textit{Outputs of Iteration 2:} $L_2^1 = \{3\}$\\

Iteration $3$: Inputs: $L_1^1, L_2^1$.\\
$T_3^1 = T_2^1\setminus L_2^1 = \{4\}$. Since $T_3^1 \cap A_1^0 = \emptyset$, $\Gamma_3^1 = \emptyset$. Hence, the Step $1$ terminates.

\begin{figure}[H]
\begin{minipage}{0.45\textwidth}
\textbf{Outputs of Step $1$:} $D_1 = \{1,2\}$ and the induced subgraph $H_1 = G[x_1^0, x_2^0, x_1^1, x_2^1] = G[\{2,4,1,6 \}]$. $A_2^1 = A_2^{0} \cup A_1^{0} = \{ 1, 2,3 \}$, $A_1^1 = \emptyset$ and $A_0^1 = \{4\}$.
    
\end{minipage}
\hfill
\begin{minipage}{0.45\textwidth}
\centering
\begin{tikzpicture}[scale=0.9, line cap=round, line join=round, >=triangle 45, x=1cm, y=1cm]

\clip(-5.9,-1.4) rectangle (2,3.8);
\tikzset{every path/.style={black, line width=1pt, dash pattern=on 1pt off 3.5pt}}
\draw (-3.8463466690239216,2.011325981637946)-- (-1.4435159961989221,2.0286125332410014);
\draw (-1.4435159961989221,2.0286125332410014)-- (-1.408942892992807,0);
\draw (-1.4435159961989221,2.0286125332410014)-- (1,1);
\draw [black, line width=1.5pt, dash pattern=](-1.408942892992807,0)-- (1,1);
\draw [black, line width=1.5pt, dash pattern=](-3.8463466690239216,-0.011200555919785929)-- (-5.298417003680756,-0.443364345996224);

\draw[black, line width=1.5pt, dash pattern=] (-1.408942892992807,0) -- (-3.8463466690239216,-0.011200555919785929);
\draw (-3.8463466690239216,2.011325981637946)-- (-5.4367094165052166,2.339770462096037);
\draw (-3.8463466690239216,2.011325981637946)-- (-3.8463466690239216,-0.011200555919785929);
\draw (-3.8463466690239216,2.011325981637946)-- (-1.408942892992807,0);
\draw (-3.8463466690239216,-0.011200555919785929)-- (-1.4435159961989221,2.0286125332410014);

\draw (-5.2,2.76594230122228) node[anchor=north west] {$K_{n_4}$};
\draw (-5.2,0.38) node[anchor=north west] {$K_{n_3}$};
\draw (-3.3,2) node[anchor=north west] {$K_{n_2}$};
\draw (-0.9940656545194256,1.3) node[anchor=north west] {$K_{n_1}$};
\draw (-1.59909496062644,0) node[anchor=north west] {$2$};
\draw (-4,0) node[anchor=north west] {$4$};
\draw (-6,-0.25) node[anchor=north west] {$6$};
\draw (0.54,1.1) node[anchor=north west] {$1$};

\begin{scriptsize}
\draw [fill=black] (1,1) circle (2.5pt);
\draw [fill=black] (-5.298417003680756,-0.443364345996224) circle (2.5pt);
\draw [fill=black] (-5.4367094165052166,2.339770462096037) circle (1pt);
\draw [fill=black] (-3.8463466690239216,2.011325981637946) circle (1.5pt);
\draw [fill=black] (-3.8463466690239216,-0.011200555919785929) circle (2.5pt);
\draw [fill=black] (-1.4435159961989221,2.0286125332410014) circle (1pt);
\draw [fill=black] (-1.408942892992807,0) circle (2.5pt);
\end{scriptsize}

\end{tikzpicture}
\caption{Step $1$: $H_1$}
    \label{fig:step_1}
\end{minipage}

\end{figure}
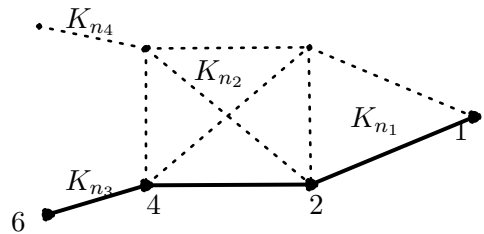
Since $A_1^1 = \emptyset$, the algorithm terminates after Step $1$. Note that the induced induced subgraph $H_1$ is a maximal tree of $G$. 
Let us construct the set $S = (\cup_{i\in A_2^1} V(K_{n_i})) \setminus V(H_1) = \{3,5 \}$. then $S$ is a sign-split disconnector set of $G$ and $H_1$ is a component of $G\setminus S$.

\end{example}

}

\begin{lemma}
The algorithm described above terminates after \new{a} finite number of steps.
\end{lemma}
\begin{proof}
Since $A_0^{p} \sqcup A_1^{p} = A_0^{p-1}$, it follows that $A_1^{p} = \emptyset$ if and only if $A_0^{p} = A_0^{p-1}$. Hence, the algorithm terminates if $A_0^{p-1} = A_0^{p}$. Since $\displaystyle{\big\{ A_0^p \big\}_{p\geq 1}}$ is a strictly descending chain of subsets of $[t]$, this implies that the algorithm must terminate in at most $t+1$ steps.
\end{proof}

Now we prove two lemmas that give some detailed information about certain outcomes of the algorithm. The proof essentially follows from our construction of the sets $\Gamma_q^p$ in the algorithm. These two results are useful in understanding disconnector sets.

\begin{lemma}\label{cremark_3}
For $p\geq 1$, if $i\in A_1^{p-1}$, then either $i\in L_q^p$ for some $q\in D_p$ or $\displaystyle{V(K_{n_i}) \subseteq \bigcup _{j\in A_2^{p-1}\cup_{q\in D_p} L_{q}^p} V(K_{n_j})}$.
\end{lemma} 

\begin{proof}
From the algorithm, it follows that $L_q^p \cap A_1^{p-1} \neq \emptyset$. So, either $i\in L_q^p$ for some $q\in D_p$ or $i\notin \cup_{q\in D_p} L_q^p$. Suppose $s$ is the highest index in $D_p$, i.e., $\Gamma_{s+1}^p = \emptyset$. Note that 
    \[ \Gamma_{s+1}^p =  \{ L\subseteq T_{s+1}^p: L\cap A_1^{p-1} \neq 0 ~\text{and}~ \bigcap_{i\in L}V(K_{n_i}) \setminus \bigcup _{j\in A_2^{p-1}\cup_{j=1}^{s} L_{j}^p} V(K_{n_j}) \neq \emptyset \}. \]
    
Now, if $i\notin L_q^p$ for all $q\in D_p$, then take the set $L=\{i\}$. Clearly $L\cap A_1^{p-1} \neq \emptyset$ and $L\subseteq T_{s+1}^p$. Since $\Gamma_{s+1}^p = \emptyset$,  $V(K_{n_i})$ should be contained in $ \bigcup _{j\in A_2^{p-1}\cup_{q\in D_p} L_{q}^p} V(K_{n_j})$.
\end{proof}

\new{\begin{remark} \label{cremark_4}
    If the algorithm terminates after the $\n$th step, then there are two possibilities.

    \begin{enumerate}
        \item either $\Gamma_1^\n = \emptyset$, i.e., for any $i\in A^{\n -1}_1$, $V(K_{n_i}) \subseteq \bigcup _{j\in A_2^{\n -1}} V(\new{K}_{n_j})$.

        \item or for any $i\in L^\n_q$, $i\notin A_0^{\n-1}$.
    \end{enumerate}
\end{remark}}


Observe that the iteration at the $0$th step is different from the iterations at $p$th step for $p\geq 1$. In fact, for $p\geq 1$, there is repetition in the algorithm. We notice that irrespective of our choices, the sets $L_q^0$ satisfy some combinatorial properties. We first prove them. Later, using those properties, we show that $ H_0 $ is a tree. 
Furthermore, we establish that $ H_p $ is a tree for all \( p \) and for any choices of \( L_q^p \).

\begin{notation}

For any set $L \subseteq [t]$, $\mu(L)$ denotes the minimum element of the set $L$.
\end{notation}

\begin{lemma}\label{cprop_1}
For every $i\in L_q^0$, if $i\neq \mu(L_q^0)$, then $\lambda(i) \subseteq L_q^0$.
\end{lemma}

\begin{proof}
Since $i\neq \mu(L_q^0)$, $i > \mu(L_q^0)$. So, from \Cref{cRemark_1}, $K_{n_{\mu(L_q^0)}} \cap K_{n_i} = K_{n_{\mu(L_q^0)}} \cap (K_{n_i}\cap_{l\in \lambda(i)} K_{n_l}) $.  This implies that $m(L_q^0) = m(L_q^0 \cup \lambda(i) ) $. Since $L_q^0$ is a maximal element of $\Gamma_q^0$, $\lambda(i) \subseteq L_q^0$.
\end{proof}

\begin{example}
    In \Cref{fig:exam_1}, $\lambda(2) = \{1\}$ and $\lambda(3) = \{2\}$. Note that $\lambda(2) \subset L_1^0$ and $\lambda(3) \subset L_2^0$.
\end{example}

\begin{lemma} \label{cprop_7}
    Let $i\in [t]$ be such that $i< \mu(L_q^0)$, for some $q\in D_0$. Then for every $j\in L_q^0$, $K_{n_j} \cap K_{n_i} \subseteq K_{n_{\mu(L_q^0)}} \cap K_{n_i}  $.
\end{lemma}
\begin{proof}
 If $j= \mu(L_q^0)$, then it is trivial. If $j > \mu(L_q^0)$, then $K_{n_j} \cap K_{n_i} \subseteq K_{n_{j_1}}\cap K_{n_i}$ for some $j_1 \in \lambda(j)$, by \Cref{cRemark_1}. Then by \Cref{cprop_1}, $j_1\in L_q^0$, and hence the proof follows by induction.
\end{proof}
According to the algorithm, every $L_i^0$ for $i>1$ intersects with some $L_j^0$ for some $j<i$. We show that the minimum element of their intersection must be the minimum element of $L_i^0$ or $L_j^0$. So, either $\mu(L_i^0)$ or $\mu(L_j^0)$ belong to $A_2^0$. This helps us to prove the next lemma.

\begin{lemma}\label{cprop_6}
If $L_i^0$ and $L_j^0$ intersect for some $i\neq j$, then $\mu(L_i^0 \cap L_j^0)$ is either $\mu(L_i^0)$ or $\mu(L_j^0)$.
\end{lemma}
\begin{proof}
Let $\mu(L_1^0 \cap L_2^0) = l$. \new{If $l=1$, then $l$ must be the minimum index of both $L_1^0$ and $L_2^0$. Therefore, $\mu(L_i^0 \cap L_j^0)= \mu(L_i^0) = \mu(L_j^0)$ = 1.} Now, suppose $l>1$ and $l\neq \mu(L_1^0)$ as well as $l\neq \mu(L_2^0)$. Then by \Cref{cprop_1}, $\lambda(l) \subseteq L_1^0 \cap L_2^0$. \new{Since $l>1$, $\lambda(l) \neq \emptyset$} and it contains indices smaller than $l$, this yields a contradiction. So, $l$ must be one of $\mu(L_i^0)$ or $\mu(L_j^0)$.   
\end{proof}

\begin{example}
    In \Cref{example 1}, the minimum element of $L_1^0\cap L_2^0$ is $2$ and it is the minimum element of $L_2^0.$
\end{example}

Recall that $x_q^0 \in \cap_{l\in L_q^0} V(K_{n_l})$ and $L_q^0$ are maximal sets that contain all $l\in[t]$ such that $x_q^0 \in V(K_{n_l})$. We also know that if two vertices $x$ and $y$ of $G$ are adjacent, then there must be some $l\in [t]$ so that $x,y \in V(K_{n_l})$. Thus, for some $r_1, r_2 >0$, if $L_{r_1}^0 \cap L_{r_2}^0 \neq \emptyset$, then $x_{r_1}^0$ and $x_{r_2}^0$ are adjacent in $H_0$ and the converse is also true. Now suppose that for some vertex $x$, $\{x,x_{r_1}^0,x_{r_2}^0\}$ is a $3$ cycle. So, there exists $i,j\in [t]$ so that $x, x_{r_1}^0 \in V(K_{n_i})$ and $x, x_{r_2}^0 \in V(K_{n_j})$. The next lemma shows that if both $i,j$ are in $A_1^0$, then there exists $u\in A_2^0$ such that $x\in K_{n_u}$. Later, using this result, we show that $H_1$ is a tree.
\begin{lemma} \label{clemma_1}
Let $i,j\in A_1^0$ be such that $i\in L_{r_1}^0$ and $j\in L_{r_2}^0$, where $r_1 \neq r_2$ and $L_{r_1}^0 \cap L_{r_2}^0 \neq \emptyset$. Then there exists $u\in A_2^0$ such that $K_{n_i} \cap K_{n_j} \subseteq K_{n_u}.$ 
\end{lemma}

\begin{proof}
Since $L_{r_1}^0 \cap L_{r_2}^0 \neq \emptyset$, by \Cref{cprop_6}, $\mu(L_{r_1}^0 \cap L_{r_2}^0)$ is either $\mu(L_{r_1}^0)$ or $\mu(L_{r_2}^0)$. Without loss of generality we assume that $\mu(L_{r_1}^0 \cap L_{r_2}^0)= \mu(L_{r_1}^0)$. Then $\mu(L_{r_1}^0)\in A_2^0$. Define $L=  L_{r_1}^0 \cap A_1^0\ $. Then $i\in L$ and  $i> \mu(L_{r_1}^0) $. We have two cases.

\begin{itemize}
    \item[\textbf{Case 1:}] $j< i$. Then $K_{n_i} \cap K_{n_j} = K_{r_i}\cap K_{n_j} \subseteq K_{n_{i_1}}\cap K_{n_j}$ for some $i_1\in \lambda(i)$. By \Cref{cprop_1}, $\lambda(i) \subseteq L_{r_1}^0$. Hence $i_1\in L_{r_1}^0$. Now we proceed by induction on $i$, considered as an element of $L$. Suppose $i= \mu(L)$. Then $i_1<i$ and $i_1\in L_{r_1}^0$. Since $i =\mu(L)$, $i_1 \notin A_1^0$ so that $i_1\in A_2^0$. Then choose $u= i_1$. Now suppose $i> \mu(L)$. 
    If $i_1\in A_2^0$, then choose $i_1= u$. If $i_1 \notin A_2^0$, then $i_1\in A_1^0$ and hence $i_1\in L$. Since $i_1 <i$, by the induction hypothesis, there exists $u\in A_2^0$ such that $ K_{n_{i_1}}\cap K_{n_j} \subseteq K_{n_u}$. Since $K_{n_i} \cap K_{n_j} \subseteq K_{n_{i_1}}\cap K_{n_j} $, the assertion follows. 

    \item[\textbf{Case 2:}] $j>i$. Let $L'= \{ s\in L_{r_2}^0 \cap A_1^0: s>i \} $. Then $j\in L'$. Note that \new{ $j> \mu( L_{r_2}^0),$ because $j \in  L_{r_2}^0 $.} Therefore, by \Cref{cprop_1}, $\lambda(j) \subseteq L_{r_2}^0$. Since $j>i$, $K_{n_i} \cap K_{n_j} \subseteq K_{n_i}\cap K_{n_{j_1}} $ for some $j_1\in \lambda(j)$. 
    We proceed by induction on $j$ considered as an element of $L'$. Suppose $j= \mu(L')$. If $j_1\in A_2^0 $, then choose $u = j_1$. If $j_1 \notin A_2^0$, then $j_1\in A_1^0$ and $j_1<i$. So, by \textbf{Case 1} there exists $u\in A_2^0$ such that $K_{n_i}\cap K_{n_{j_1}} \subseteq K_{n_u}$, hence $K_{n_i} \cap K_{n_j}  \subseteq K_{n_u} $. Now assume that $j> \mu(L')$. If $j_1\in A_2^0$, then choose $u = j_1$. If $j_1 \notin A_2^0$, then $j_1<j$ and $j_1\in A_1^0$. If $j_1 < i$, then the proof follows by \textbf{Case 1}. If $j_1>i$, then $j_1 \in L'$, hence the proof follows by induction hypothesis on $j_1$.
\end{itemize} 
\end{proof}

The next two propositions give some detailed information on how the vertices $x_q^p \in \cap_{l\in L_q^p} V(K_{n_l})$ are connected to each other. This helps us to understand the structure of $H_p$.

\begin{proposition} \label{cprop_2}
Let $p\geq 1$. Then for given $q\in D_p$, there exists $\tilde{q} \in D_{p-1}$, such that the vertices $x_q^p$ and $x_{\tilde{q}}^{p-1}$ are adjacent. Moreover, if $p' < p-1$, then $x^p_q$ and $x^{p'}_{q'}$ are not adjacent for any $q' \in D_{p'}$.
\end{proposition}

\begin{proof}
Recall that $\displaystyle{ x_q^p \in \bigcap_{i\in L_q^p}V(K_{n_i}) \setminus \bigcup _{j\in A_2^{p-1}\cup_{s=1}^{q-1} L_{s}^p} V(K_{n_j})}$ and $L_q^p \cap A_1^{p-1} \neq \emptyset$. This implies that there exists $\tilde{q}\in D_{p-1}$ such that $L_q^p \cap L_{\tilde{q}}^{p-1} \neq \emptyset$. Thus for any $i\in L_q^p \cap L_{\tilde{q}}^{p-1}$, $x_q^p, x_{\tilde{q}}^{p-1} \in V(K_{n_i}) $. Therefore, they are adjacent.

Consider $x_{q^\prime}^{p^\prime}$ where $p^\prime <p-1$ and $q^\prime \in D_{p^\prime}$. Suppose $x_{q^\prime}^{p^\prime} \in V(K_{n_j})$ for some $j\in [t]$. Then $j\in L_{q^\prime}^{p^\prime}$, which implies that $j$ is either in $A_2^{p^\prime}$ or $A_1^{p^\prime}$. Since $p^\prime <p-1$, $A_2^{p-1} \supseteq A_2^{p^\prime} \cup A_1^{p^\prime} $. Since $x_q^p \notin \cup_{i\in A_2^{p-1}} V(K_{n_i})$, it follows that $x_q^p \notin V(K_{n_j})$. Hence $x_q^p$ is not adjacent to $x_{q^\prime}^{p^\prime}$.
\end{proof}

\begin{proposition} \label{cprop_3}
 For $p\geq 1$, the vertices $x_q^p$ and $x_{q^\prime}^p$ are not adjacent for any $q \neq q^\prime \in D_p$.
\end{proposition}

\begin{proof}
Suppose $q< q^\prime$ and $x_q^p$ and $x_{q^\prime} ^p$ are adjacent. Then there exists a complete graph $K_{n_j}$ containing both the vertices. We know that $[t] = A_2^{p-1} \cup A_1^{p-1} \cup A_0^{p-1}$. Since by our choice, $x_q^p \notin \bigcup _{i\in A_2^{p-1}\cup_{s=1}^{q-1} L_{s}^p} V(K_{n_i})$, $j\notin A_2^{p-1}\cup_{s=1}^{q-1} L_{s}^p$. So, $j\in T_q^p $. But then $j$ must belong to $L_q^p$, as it is a maximal element of $\Gamma_q^p$. Hence it follows from the algorithm that $x_{q^\prime}^p \notin \cup_{i\in L_q^p} V(K_{n_i})$. In particular, $x_{q^\prime}^p \notin V(K_{n_j})$, a contradiction.
\end{proof}
Now we have all the prerequisites to prove the key outcome of the algorithm: that $H_p$ is a tree.

\begin{proposition}\label{H_0 tree}
$H_0$ is a tree.
\end{proposition} 

\begin{proof}
Recall that $H_0 = G[\{x_q^0: q\in D_0\}]$ and $x_q^0 \in \cap_{i\in L_q^0} V(K_{n_i})$. Furthermore, given any $q\geq 1$, $L_q^0 \cap L_{q^\prime}^0 \neq \emptyset$ for exactly one $q^\prime <q$. Suppose $i\in [t]$ is such that $L_q^0 \cap L_{q^\prime}^0 \ni i$, then  $V(K_{n_i})$ contains both $x_q^0$ and $x_{q^\prime}^0$. This implies that for every $q\geq 1$, $x_q^0$ is adjacent to $x_{q^ \prime}^0$ for some $q^\prime <q$, i.e., $H_0$ is connected. Now $H_0$ is an induced subgraph of the chordal graph, so it is also chordal. Let us assume that $H_0$ has a $3$-cycle formed by the vertices $x_{q_1}^0, x_{q_2}^0$ and $x_{q_3}^0$ such that $q_1>q_2>q_3$. Then $\{  x_{q_1}^0, x_{q_2}^0 \}, \{x_{q_1}^0, x_{q_3}^0 \} \in E(G)$. Thus, there exist $i_1, i_2 \in[t]$ such that $V(K_{n_{i_1}}) \ni x_{q_1}^0, x_{q_2}^0$ and $V(K_{n_{i_2}}) \ni x_{q_1}^0, x_{q_3}^0$. Since $L_q^0$ are maximal subsets of $[t]$ with $m(L_q^0) \new{> 0}$, $i_1\in  L_{q_1}^0 \cap L_{q_2}^0$ and $i_2 \in L_{q_1}^0 \cap L_{q_3}^0$. This implies that $L_{q_1}^0$ intersects with $L_{q_2}^0$ and $L_{q_3}^0$, which is a contradiction. Hence $H_0$ does not have a $3$-cycle. So that it is a tree.
\end{proof}

\begin{proposition} \label{H_1 tree}
    $H_1$ is a tree.
\end{proposition}

\begin{proof}
    We already know that $H_0$ is a tree. So, to prove that $H_1$ is a tree, it is sufficient to prove the following statements:
    \begin{itemize}
        \item[\textbf{(1)}] Each $x_q^1$ is adjacent to exactly one $x_i^0$ for $q\in D_1$ and $i\in D_0$.

        \item[\textbf{(2)}] $x_q^1$ and $x_{q^\prime} ^1$ are not adjacent, for any pair $q, q^\prime \in D_1$.
    \end{itemize} 
The statement $\textbf{(2)}$ follows from \Cref{cprop_3}. We now prove $\textbf{(1)}$. \new{Note that by \Cref{cprop_2}, $x_q^1$ is adjacent to one $x_i^0$ for some $i\in D_0$. We show that it cannot be adjacent to more than one such $x_i^0$.} Assume that $i_1,i_2 \in D_0$ are two distinct indices such that both $ x_{i_1}^0$ and $x_{i_2}^0$ are adjacent to $x_q^1$ for some $q\in D_1$. \new{We show that we can choose $x_{i_1}^0$ and $x_{i_2}^0$ such that they are adjacent. If they are not, then there exists an induced path $P:x_{i_1}^0, z_1, \ldots, z_{a-1}, z_a= x_{i_2}^0$ in $H_0$ and $P\cup \{x_q^1\}$ gives an induced cycle in $G$. Since $G$ is chordal and $H_0$ is a tree, $x_q^1$ must be adjacent to $z_1$. Then we replace $x_{i_2}^0$ by $z_1$. Therefore, $x_q^1, x_{i_1}^0, x_{i_2}^0$ forms an induced cycle in $G$.} Then there exist $l,s \in [t]$ such that $x_{i_1}^0, x_q^1 \in V(K_{n_l} )$ and $x_{i_2}^0, x_q^1 \in V(K_{n_s} )$. Then $l\in L_{i_1}^0$ and $s\in L_{i_2}^0$. Since $ x_{i_1}^0$ and $x_{i_2}^0$ are adjacent in $H_0$, $L_{i_1}^0 \cap L_{i_2}^0 \neq \emptyset$. 
From the algorithm, we get $x_q^1\notin  V(K_{n_j})$ for any $j\in A_2^0$. So, $l,s\in A_1^0$ and $l\neq s$. Then it follows from \Cref{clemma_1} that there exists $u\in A_2^0$ such that $x_q^1 \in V(K_{n_s}) \cap V(K_{n_l} ) \subseteq V(K_{n_u})$, which is a contradiction. Hence $\textbf{(1)}$ holds.
\end{proof}

\begin{proposition} \label{H_p tree}
    $H_p$ is a tree. for $p\geq 0$.
\end{proposition}

\begin{proof}
We proceed by induction on $p$. The statement is true for $p= 0,1$ by \Cref{H_1 tree} and \Cref{H_0 tree}. Take $p>1$. By induction hypothesis $H_{p-1}$ is a tree and $H_p$ is the induced subgraph of $G$ on the vertex set $V(H_{p-1}) \cup \{ x_q^p: q\in D_p\}$. The assertion follows once we prove the following:
\vskip 2mm
\noindent
\textbf{Claim}: $\deg_{H_p}(x_q^p) =1$ for any $q\in D_p$.

\noindent
\textit{Proof of the Claim}: Fix a $q\in D_p$. It follows from \Cref{cprop_2,cprop_3} that, in the graph $H_p$, $x_q^p$ is adjacent only to the vertices of the form $x_{q^\prime}^{p-1}$ for some $q^\prime \in D_{p-1}$. \new{Further, \Cref{cprop_2} ensures that $x_q^p$ is adjacent to some $x_{q^\prime}^{p-1}$.} Now suppose $x_q^p$ is adjacent to $x_{q_1}^{p-1}$ and $x_{q_2}^{p-1}$ for $q_1,q_2 \in D_{p-1}$ and $q_1 \neq q_2$. Since $H_{p-1}$ is connected, there exists a path $P_m$, for some $m>1$, in $H_{p-1}$, connecting $x_{q_1}^{p-1}$ and $x_{q_2}^{p-1}$. Without loss of generality, we assume that $x_q^p$ is adjacent only to the end vertices of $P_m$, for, if any interior vertex of $P_m$ is adjacent to $x_q^p$, then we may replace $x_{q_1}^{p-1}$ or $x_{q_2}^{p-1}$ with that vertex.  As we know $p-1 \geq 1$, by \Cref{cprop_3}, there is no edge between $x_{q_1}^{p-1}$ and $x_{q_2}^{p-1}$. This implies that $m>2$. Therefore, we see that the path $P_m$ together with $x_q^p$ form an induced cycle in $H_p$ of length grater than $3$, a contradiction. So, $x_q^p$ is adjacent to exactly one vertex in $H_p$. 
\end{proof}

\begin{example}
Note that in \Cref{example 1}, both $H_0= G[\{2,4\}]$ and $H_1 = G[\{ 1, 2, 4, 6\} ]$ are path graphs.
\end{example}

Now, based on the algorithm, we introduce the following set \[S^{\n}(G)= \big(\bigcup_{i\in A_2^{\n}} V(K_{n_i}) \big) \setminus V(H_{\n}),\] where $\n$ is the last step of the algorithm. Note that this set varies depending on the choices of $L_q^p$. 
\begin{example}
According to our choices in \Cref{example 1}, the algorithm terminates after $1$st step and $A_2^1= \{ 2,1,3\}$. So, $S^1(G) = \{ 3,5\} $. This is a sign-split disconnector set.
\end{example}
We show that $S^\n(G)$ is a disconnector set and it satisfies the sign-split property. The next two lemmas are preparations to prove this result.
First we show that the vertices of $S^\n(G)$ make a $3$-cycle with some edge of the tree $H_\n$. In fact we prove a general statement by taking any arbitrary step $p$ instead of the last step $\n$ of the algorithm.

\begin{lemma} \label{H_p with v nbipartite}
    For every vertex $v \in  \big( \bigcup_{i\in A_2^p} V(K_{n_i}) \setminus V(H_p) \big),$ there exist $v_1, v_2 \in V(H_p)$ such that $\{v,v_1,v_2\}$ forms a triangle. In particular, the induced subgraph on the vertex set $V(H_p) \cup \{v\}$ is non-bipartite.
\end{lemma}

\begin{proof}
We prove the assertion by induction on $p$. Let $S^p= \big( \bigcup_{i\in A_2^p} V(K_{n_i}) \setminus V(H_p) \big)$.
We show that for any vertex $v\in S^p$, $v$ is adjacent to two vertices of $H_p$. Then the induced subgraph $G[V(H_p) \cup \{v\}]$ has a cycle. Since $G$ is chordal \new{and $H_p$ is a tree, $G[V(H_p) \cup \{v\}]$ has a $3$-cycle of the form $\{v, v_1, v_2\}$, where $v_1, v_2 \in V(H_p)$. Hence, it is non-bipartite.}
\begin{itemize}
\item [$\mathbf{p=0}:$] Suppose $v\in V(K_{n_j}) \setminus V(H_0)$, for some $j\in A_2^0$. Hence $j\in L_{q_1}^0 \cap L_{q_2}^0$ for some $q_1,q_2 \in D_0$ and $q_1 < q_2$. Then $x_{q_1}^0$ and $x_{q_2}^0$ belong to $K_{n_j}$ which completes this case. 

\item [$\mathbf{p>0}:$] Assume by induction that the assertion is true for $r \leq p-1$. Let  $v\in V(K_{n_j})$, for some $j\in A_2^p$. We know that $A_2^p = A_2^{p-1} \cup A_1^{p-1}$ and $H_{p-1}$ is an induced subgraph of $H_p$. So, if $j\in A_2^{p-1}$, then $v\in S^{p-1}$ and by induction hypothesis, $V(H_{p-1}) \cup \{v\}$ is non-bipartite. Hence $V(H_p) \cup \{v\}$ is also non-bipartite. The remaining case is when $j\in A_1^{p-1}$. Then it follows from the definition of $A_1^{p-1}$ that $j\in L_q^{p-1}$ for some $q\in D_{p-1}$. So, $x_q^{p-1} \in V(K_{n_j})$. Hence $v$ and $x_q^{p-1}$ are adjacent. Now we find another vertex of $H_p$ that is adjacent to $v$. According to \Cref{cremark_3}, either $j\in L_{q^\prime}^p$ for some $q^{\prime}\in D_p$ or $\displaystyle{V(K_{n_j}) \subseteq \bigcup _{l\in A_2^{p-1}\cup \cup_{q\in D_p} L_{q}^p} V(K_{n_l})}$.
\begin{itemize}
    \item[Case 1:] If $j\in L_{q^\prime}^p$, then $x_{q^\prime}^p \in V(K_{n_j})$. So, $x_{q^\prime}^p$ is adjacent to $v$.  

    \item[Case 2:] If $\displaystyle{V(K_{n_j}) \subseteq \bigcup _{l\in A_2^{p-1}\cup \cup_{q\in D_p} L_{q}^p} V(K_{n_l})}$, then there exists $l\in A_2^{p-1}\cup_{s\in D_p} L_{s}^p$, such that $v\in V(K_{n_l})$.
\end{itemize}
If $l \in A_2^{p-1}$, then by induction we are done. If $l \in L_s^p$, then $v$ is adjacent to $x_s^p$. This completes the proof.
\end{itemize}
\end{proof}

The next lemma is used further to show that $H_n$ becomes a connected component of $G\setminus S^n(G)$. 

\begin{lemma}  \label{A_0 ncontain H_p}
    For $p\geq 0$ and every $i\in A_0^p$, $V(K_{n_i})$ does not contain any vertex of $H_p$.
\end{lemma}

\begin{proof}
First we show that for any $i\in A_0^p$ and $q\in D_p$, $x_q^p \notin V(K_{n_i})$. By definition $x_q^p \in \cap_{j\in L_q^p}V(K_{n_j})$. Since $i\in A_0^p$, $i\notin L_q^p$. So, if $x_q^p \in V(K_{n_i})$, then $L_q^p \cup \{i\} \in \Gamma_q^p$, which contradicts the maximality of $L_q^p$. Hence, $x_q^p \notin V(K_{n_i})$ which proves the assertion of the lemma for $p=0$. We now proceed by induction on $p$. Let $p>1$ and $i$ be some element of $A_0^p$. Note that for $p\geq 1$, $A_0^p \subseteq A_0^{p-1}$. So, by induction hypothesis $V(K_{n_i})$ does not contain any vertex of $H_{p-1}$. Since $V(H_p) = V(H_{p-1}) \cup \{ x_q^p : q\in D_p\}$, $V(K_{n_i}) \cap V(H_p) = \emptyset$.
\end{proof}

\begin{corollary} \label{ccoro_1}
    Let $i\in [t]$ be such that $V(K_{n_i}) \cap V(H_{p}) \neq \emptyset$, then $i\in A_2^{p} \cup A_1^{p}$.
\end{corollary}

\begin{proof}
    Since $[t] = A_2^{p} \sqcup A_1^{p} \sqcup A_0^{p}$, the assertion follows from \Cref{A_0 ncontain H_p}.
\end{proof}

\begin{theorem} \label{cthm_1}
 Let $\n$ be the last step of the algorithm and $S^{\n}(G)= \big(\bigcup_{i\in A_2^{\n}} V(K_{n_i}) \big) \setminus V(H_{\n})$. Then $S^{\n}(G)$ is a sign-split disconnector set of $G$.
\end{theorem}

\begin{proof}
By \Cref{H_p tree}, $H_{\n}$ is a tree and by \Cref{H_p with v nbipartite}, for every $v\in S^{\n}(G)$, $G[V(H_{\n}) \cup \{v\}]$ is non-bipartite. Hence if we show that $H_\n$ is a connected component of $G\setminus S^{\n}(G)$, then it implies that $S^\n(G)$ is a sign-split disconnector set. We now show that $H_\n$ is a connected component of $G\setminus S^{\n}(G)$. Let $v\in V(G)$ be such that $v\notin V(H_{\n})$ and $\{v, x_q^p\} \in E(G)$. Then there exists $K_{n_i}$ that contains the edge $\{v, x_q^p\}$. Since $[t] = A_2^{\n} \sqcup A_1^{\n} \sqcup A_0^{\n}$, from \Cref{A_0 ncontain H_p}, $i\in A_2^{\n} \sqcup A_1^{\n} $. Then by \Cref{cremark_4}, $i\in A_2^{\n}$. Hence $v\in V(K_{n_i}) \setminus V(H_{\n}) \subset S^{\n}(G)$. So, every vertex that does not belong to $H_{\n}$, but adjacent to some vertex of this graph, should belong to $S^{\n}(G)$. Hence $H_{\n}$ is a connected component of $G\setminus S^{\n}(G)$.
\end{proof}

Note that by definition, $S^{\n}(G) \subseteq \big(\bigcup_{i\in A_2^{\n}} V(K_{n_i}) \big)$. We know from \Cref{remarks on algorithm}(2) that $A_2^n \cap A_0^n = \emptyset$. But it may also happen that $S^{\n}(G) \cap (\cup_{i\in A_0^{\n}}V(K_{n_i}) \neq \emptyset$, as described in the proposition below.
\begin{proposition}
If $G\setminus S^{\n}(G)$ has more than one component, then  $S^{\n}(G)\cap (\cup_{i\in A_0^{\n}}V(K_{n_i}) \neq \emptyset$.
\end{proposition}

\begin{proof}
Note that $G\setminus S^{\n}(G)= H_{\n} \sqcup \{\big(\bigcup_{i\in A_0 ^{\n}} K_{n_i} \big) \setminus S^{\n}(G) \} $. We know that $H_{\n}$ is one component of $G\setminus S^{\n}(G)$. If $G_1$ is another component of $G\setminus S^{\n}(G)$, then $V(G_1) \subseteq  \big(\bigcup_{s\in A_0 ^{\n}} V(K_{n_s}) \big) \setminus S^{\n}(G)$. Since $G$ is connected, there exists $s\in S^{\n}(G)$ such that $s$ connects $H_{\n}$ to $G_1$. Let $s'\in V(G_1)$ be such that $s$ and $s'$ are adjacent. Then there exists $l\in [t]$ so that $s,s'\in V(K_{n_l})$. Since $s'\notin S^{\n}(G)$, $l\notin A_2^{\n}$ and  by \Cref{cremark_4}, $l\notin A_1^{\n}$. So, $l\in A_0^{\n}$. Hence $s\in S^{\n}(G)\cap (\cup_{i\in A_0^{\n}}V(K_{n_i})$.
\end{proof}
\begin{example}
    In \Cref{example 1}, $A_0^1 = \{4\}$ and $S^1(G) \cap V(K_{n_4}) = \{5\}$.
\end{example}

Based on the above observation we introduce the following sets. 
Let \[ S^{\n}_2(G) = \{ s\in S^{\n}(G): \mathcal{C}_{G\setminus S^{\n}(G)}(s) = \{H_{\n} \} \} \] and 
\[ S^{\n}_0(G) = S^{\n}(G) \setminus S_2^{\n} \subseteq S^{\n}(G) \cap (\cup_{i\in A_0^{\n}} V(K_{n_i})). \]
So, $S^{\n}(G) = S_2^\n(G)\sqcup S_0^{\n}(G)$. The vertices of $S_0^{\n}(G)$ connect at least two connected components of $G\setminus S^{\n}(G)$ and the vertices of $S_2^{\n}(G)$ are adjacent to vertices of $H_{\n}$ only. From now on, we suppress $\n$ and denote $S^{\n}(G), S_2^{\n}(G)$ and $S_0^{\n}(G)$ by $S(G), S_2(G)$ and  $S_0(G)$ respectively. We now show that \new{also $S_0(G)$ is a} sign-split disconnector set. 

\begin{proposition} \label{S_0 disconnector set}
The set $S_0(G)$ is a sign-split disconnector set of $G$.
\end{proposition}

\begin{proof}
If $S_2(G) = \emptyset$, then $S(G)= S_0(G)$. So, it is a sign-split disconnector set as proved in \Cref{cthm_1}.

If  $S_2(G) \neq \emptyset$, then for any $s\in S_2(G)$, $\mathcal{C}_{G\setminus S(G)}(s) = \{H_{\n}\}$. So, $G[V(H_{\n}) \cup S_2(G)]$ is a connected component of $G\setminus S_0$, which is non-bipartite (by \Cref{H_p with v nbipartite}). Other connected components of $G\setminus S_0(G)$ are same as the connected components of $G\setminus S(G)$. Then each element of $S_0(G)$ connects at least two connected components of $G\setminus S_0(G)$ and one such component is $G[V(H_{\n}) \cup S_2(G)]$. So, it is a disconnector set and using \Cref{cprop_4} we can say $S_0(G)$ is a sign-split disconnector set.
\end{proof}
We have discussed about the connected components of $G\setminus S(G)$ and $G\setminus S_0(G)$ in the proof of above proposition. The next result gives a necessary condition for the unmixedness of $\II_G$. 
\begin{proposition} \label{S_2 cardinality 1}
If $G$ is a non-bipartite chordal graph such that $\mathcal{I}_G$ is unmixed, then $|S_2(G)| \leq 1.$
\end{proposition}

\begin{proof}
Since $S_0(G) \in \mathcal{D}(G)$ and $\mathcal{I}_G$ is unmixed, $|S_0(G)|= b(G\setminus S_0(G))$. If $|S_2(G)| \neq \emptyset$, then $|S(G)| = |S_0(G)| + |S_2(G)|$ and $b(G\setminus S(G)) = b(G\setminus S_0(G))+ 1$. Again $S(G) \in \mathcal{D}(G)$, so, using the unmixed property of $\mathcal{I}_G$ we get $|S_2(G)| =1$.
\end{proof}

Now we prove the main result of this section. This also provides a sufficient condition for $\II_G$ to be unmixed.

\begin{theorem}\label{H_n is path}
For any non-bipartite chordal graph $G$, if $\mathcal{I}_G$ is unmixed, then $H_{\n}$ is a path.
\end{theorem}

\begin{proof}
By \Cref{H_p tree}, $H_{\n}$ is a tree. From \Cref{cthm_1} and \Cref{S_0 disconnector set}, we know that $S(G), S_0(G) \in \mathcal{D}(G)$. Since $\mathcal{I}_G$ is unmixed, $b(G\setminus S(G)) = |S(G)|$ and $b_G(S_0(G)) = |S_0(G)|$. We now prove that $\deg_{H_{\n}}(v) \new{\leq} 2$ for every $v \in V(H_{\n})$. Suppose there exists $x_q^p \in V(H_{\n})$ such that $\deg_{H_{\n}}(x_q^p)\geq 3$. Then $H_{\n}\setminus \{x_q^p\}$ has at least three bipartite components. \new{Note that for any $u\in S_0(G)$, $u$ connects $H_{\n}$ with some other component of $G\setminus S(G)$ and by \Cref{H_p with v nbipartite} $u$ is adjacent to two end points of an edge of $H_{\n}$. Therefore, each vertex of $ S_0(G)$ connects $H_{\n}\setminus \{x_q^p\}$ with some other component of $G\setminus S(G)$. Hence, $S_0(G) \cup \{x_q^p\} \in \mathcal{D}(G)$.} Now if $S_2(G) = \emptyset$, i.e., $S(G) = S_0(G)$, then $H_{\n}$ is a component of $G\setminus S_0(G)$. Since $\deg_{H_{\n}}(x_q^p)\geq 3$, $b_G(S_0(G) \cup \{x_q^p\}) \geq |S_0(G)| +2$, a contradiction. Therefore $S_2(G) \neq \emptyset$. Set $S_2(G) = \{\alpha\}$. We know that $\alpha$ is adjacent to at least two vertices of $H_{\n}$. So, $|V(H_{\n} \setminus \{x_q^p\}) \cap N_G(\alpha)| \geq 1$. Now we have two cases:

\noindent \textbf{$|V(H_{\n} \setminus \{x_q^p\}) \cap N_G(\alpha)| \geq 2$:} In this case $\alpha$ is either adjacent to an edge or it connects two components of $H_{\n}\setminus \{x_q^p\}$. Therefore $S(G)\cup \{x_q^p\} \in \mathcal{D}(G)$. Then $b_G(S(G) \cup \{x_q^p\}) \geq |S(G)| +2$, a contradiction to the unmixed property of $\II_G$.

\vskip 2mm \noindent \textbf{$|V(H_{\n} \setminus \{x_q^p\}) \cap N_G(\alpha)| = 1$:} Then $\alpha$ is adjacent to exactly two vertices of an edge in $H_{\n}$ and one of them is $x_q^p$. We already proved earlier that $S_0(G) \cup \{x_q^p\} \in \mathcal{D}(G)$. Now we get $b_G(S_0(G) \cup \{x_q^p\}) \geq |S_0(G)| +2$, a contradiction. 

\noindent Hence $H_n$ cannot have a vertex of degree $3$ and since $H_\n$ is a tree, it is a path.
\end{proof}

\begin{corollary} \label{ccoro_2}
    If $\mathcal{I}_G$ is unmixed, then $H_p$ is  a path graph for all $0\leq p \leq \n$.
\end{corollary}
Note that the algorithm provides us \new{with} two special disconnector sets $S(G)$ and $S_0(G)$, and using them we proved two important results  \Cref{S_2 cardinality 1} and \Cref{ccoro_2} in this section. We use all these data to characterize the unmixed chordal graphs in the next section.

\section{Conditions when non-bipartite chordal graphs are not unmixed}\label{sec:non-unmixed}

In this section we consider those chordal graphs $G$ for which $\cap_{i\in [t]}K_{n_i} = \emptyset $, i.e., $m([t]) =0$. \new{We show that for a large subclass of chordal graph for which $\II_G$, the algorithm produces a tree which is not a path and hence all such graphs are not unmixed.}

The following remark will be implicitly used in many of the subsequent hypotheses.
\begin{remark}\label{rmk-exists-L02}
Note that $m([t]) = 0$ if and only if $L_1^0$ is a proper subset of $[t]$ if and only if $L_2^0$ exists in any run of the algorithm. Therefore $m([t]) = 0$ if and only if there exist two different sets $L_1$ and $L_2$ in $M(\Gamma_1^0)$ so that $L_1\cap L_2 \neq \emptyset$.
\end{remark}
We now prove a lemma that plays a crucial role in showing that $\II_G$ is not unmixed, for a large subclass of chordal graphs.
If there exists a subset $S$ of $V(G)$ so that $G\setminus S$ has at least three components, then the following lemma gives a sufficient condition for $\II_G$ to not be unmixed. Recall that by $\mathcal{L}$ we denote one run of the algorithm.

\begin{lemma} \label{prelem:G-K-has-3-comp new}
Let $S\subset V(G)$ and $l_0, l_0', l_0'' \in [t]$ be such that $K_{n_{l_0}} \setminus S$, $ K_{n_{l_0'}} \setminus S$ and $ K_{n_{l_0''}} \setminus S$ are in different connected components of $G\setminus S$. Further suppose that there exist $L_1, L_2 \in M(\Gamma_1^0)$ such that $ l_0, l_0', l_0'' \in L_1$ and $l_0\in L_2$. Then $\II_G$ is not unmixed if any one of the following holds:
\begin{enumerate}
    \item Each run $\mathcal{L}$ of the algorithm with $L_1^0(\mathcal{L}) = L_2$ and $L_2^0(\mathcal{L}) = L_1$ satisfies  $V(H_{n(\mathcal{L})}) \cap S = \{x_2^0\}$.
    \item Each run $\mathcal{L}$ of the algorithm with $L_1^0(\mathcal{L}) = L_1$ and $L_2^0(\mathcal{L}) = L_2$ satisfies $V(H_{n(\mathcal{L})}) \cap S = \{x_1^0\}$.
\end{enumerate}
\end{lemma}
\begin{proof}
Note first that runs as in (1) and (2) exist since $L_1\neq L_2$ and $L_1\cap L_2 \neq \emptyset$. Let us assume the statement (1) in the hypothesis and that $\II_G$ is unmixed. Let $\mathcal{L}$ be a run of the algorithm such that $L_1^0 = L_2$ and $L_2^0 = L_1$. We construct a run $\mathcal{L}'$ of the algorithm which contradicts \Cref{ccoro_2}. First take $L_q^0(\mathcal{L'}) = L_q^0 (\mathcal{L})$ and $x_q^0(\mathcal{L'}) = x_q^0(\mathcal{L})$ for $q\geq 1$, i.e., $\mathcal{L}$ and $\mathcal{L'}$ have same choices at the $0$th step.
Note that no vertices $x_q^0(\mathcal{L})$ other than $x_2^0(\mathcal{L})$ are in $S$. From now on we write $x_q^0$ instead of $x_q^0(\mathcal{L'})$.

Let $C_1$, $C_2$ and $C_3$ be the  connected components of $G\setminus S$ containing  $K_{n_{l_0}} \setminus S$, $ K_{n_{l_0'}} \setminus S$ and $ K_{n_{l_0''}} \setminus S$ respectively.  Suppose there exists $l\in A_2^0(\mathcal{L'})$ such that $ V(K_{n_l}) \cap V(K_{n_{l_0'}} \setminus S)  \neq \emptyset$. Then all vertices in $ K_{n_l} \setminus S$ are in the connected component $C_2$. Since $K_{n_{l_0}} \setminus S$ is in the component $C_1$, $V(K_{n_l})\cap V(K_{n_{l_0}}) \subseteq S$. Since $x_1^0 \notin S$ by the hypothesis of $(1)$, it follows that $l\notin L_1^0(\mathcal{L'})$. Since $l\in A_2^0(\mathcal{L'})$, there exists $q \in D_0(\mathcal{L'})$ such that $l\in L_{q}^0(\mathcal{L'})$, where $q \notin \{1,2\}$. So, $x_q^0 \in V(K_{n_l} \setminus S) \subseteq V(C_2)$. Since $\II_G$ is unmixed, $H_0(\mathcal{L'})$ is a path graph (\Cref{H_n is path}). Therefore there exists a unique path from $x_1^0$ to $x_{q}^0$ in $H_0(\mathcal{L'})$. Since $x_1^0$ is in $C_1$ and $x_{q}^0$ is in $C_2$, the path must pass through $S$, and hence through $x_2^0$.

Similarly, if there exists $l' \in A_2^0(\mathcal{L'})$ such that $ V(K_{n_l'}) \cap V(K_{n_{l_0''}} \setminus S ) \neq \emptyset$, then there exists $x_{q'}^0 \in C_3$ such that the path from $x_1^0$ to $x_{q'}^0$ passes through $x_2^0$. Thus we have the following cases based on whether or not such $l$ and $l'$ exist.

\noindent \textbf{Case 1:   $V(K_{n_{l_0'}} \setminus S )\cap \left(\bigcup_{l\in A_2^0} V(K_{n_l}) \right) \neq \emptyset$, $V(K_{n_{l_0''}} \setminus S )\cap \left( \bigcup_{l\in A_2^0} V(K_{n_l}) \right) \neq \emptyset$: }\\
As discussed above, there exist paths from $x_1^0$ to $x_{q}^0$ and $x_{q'}^0$ passing through $x_2^0$. Note that these paths are distinct because 
$x_{q}^0$ and $x_{q'}^0$ are in distinct components. Hence $\deg_{H_0(\mathcal{L'})}(x_2^0) \geq 3$, a contradiction to \Cref{ccoro_2}.

\noindent \textbf{Case 2:   $V(K_{n_{l_0'}} \setminus S )\cap \left(\bigcup_{l\in A_2^0} V(K_{n_l}) \right) \neq \emptyset$, $V(K_{n_{l_0''}} \setminus S )\cap \left( \bigcup_{l\in A_2^0} V(K_{n_l}) \right) = \emptyset$ or vice versa: }\\
We consider only the first situation since the other one has the same proof with the roles of $l_0'$ and $l_0''$ interchanged. In the first situation, 
since $(K_{n_{l_0'}} \setminus S )\cap \left(\bigcup_{l\in A_2^0(\mathcal{L'})} K_{n_l}\right) \neq \emptyset$, from the above discussion there exists a path from $x_1^0$ to $x_{q}^0$ that contains $x_2^0$. Hence $\deg_{H_0(\mathcal{L'})} x_2^0 \geq 2$. Now $(K_{n_{l_0''}} \setminus S )\cap \bigcup_{l\in A_2^0} K_{n_l} = \emptyset$ implies $K_{n_{l_0''}} \nsubseteq \bigcup_{l\in A_2^0} K_{n_l} $. So, $\{ l_0''\} \in \Gamma_1^1(\mathcal{L'})$. We choose $L_1^1(\mathcal{L'}) \in M(\Gamma_1^1(\mathcal{L'}))$ such that $l_0'' \in L_1^1(\mathcal{L'})$. Since $L_2^0(\mathcal{L'})\cap L_1^1(\mathcal{L'}) \neq \emptyset$, $x_2^0$ and $x_1^1$ are adjacent. Hence $\deg_{H_1(\mathcal{L'})} x_2^0 \geq 3$, a contrad
iction to \Cref{ccoro_2}.

\noindent \textbf{Case 3: $(K_{n_{l_0'}} \setminus S )\cap \bigcup_{l\in A_2^0} K_{n_l} = \emptyset$, $(K_{n_{l_0''}} \setminus S )\cap \bigcup_{l\in A_2^0} K_{n_l} = \emptyset$: }\\
As in Case 2, $K_{n_{l_0'}} \nsubseteq \bigcup_{l\in A_2^0} K_{n_l} $ and  we choose $L_1^1(\mathcal{L'}) \in M(\Gamma_1^1(\mathcal{L'}))$ such that $l_0' \in L_1^1(\mathcal{L'})$. Since $l_0' \in L_1^1(\mathcal{L'}) \cap L_2^0(\mathcal{L'})$, we see that $x_2^0$ and $x_1^1$ are adjacent. Note that for any $l\in L_1^1(\mathcal{L'})$, $V(K_{n_l})\setminus S$ and $V(K_{n_{l_0''}}) \setminus S$ are in distinct components $C_2$ and $C_3$ respectively. This implies $K_{n_{l_0''}} \nsubseteq \bigcup_{l\in A_2^0(\mathcal{L'}) \cup L_1^1(\mathcal{L'})} K_{n_l}$. Hence $\{ l_0''\} \in \Gamma_2^1(\mathcal{L'})$. Choose $L_2^0 \in M(\Gamma_2^1(\mathcal{L'}))$ such that $l_0'' \in L_2^1(\mathcal{L'})$. So, $l''_0 \in L_2^1(\mathcal{L'}) \cap L_2^0(\mathcal{L'})$ and hence $x_2^0$ and $x_2^1$ are adjacent. Therefore $\deg_{H_1(\mathcal{L'})} x_2^0 \geq 3$, a contradiction.

This completes the proof assuming hypothesis (1). Since hypothesis (2) is obtained by interchanging the roles of $L_1$ and $L_2$, the same proof as above with the roles of $L_1$ and $L_2$ interchanged, and $x_1^0$ and $x_2^0$ interchanged, works.
\end{proof}
The following corollary gives us a combinatorial condition on the graph so that $\II_G$ is not unmixed and this condition does not depend on the choices in the algorithm. 
\begin{corollary}\label{prelem:G-K-has-3-comp}
     Let $L \subseteq [t]$ be such that $m(L) >0$. Further, let $l_0\in L$ and $S \subseteq V(K_{n_{l_0}})$ be a disconnector set such that the following conditions hold:
\begin{enumerate}
    \item[(i)] $\cap_{l \in L}V( K_{n_l}) \cap S = \emptyset$.
    \item[(ii)] there exist $l_0', l_0'' \in [t]$ such that $K_{n_{l_0}} \setminus S$ , $ K_{n_{l_0'}} \setminus S$ and $ K_{n_{l_0''}} \setminus S$ are in different connected components of $G\setminus S$.
    \item[(iii)] $V(K_{n_{l_0}}) \cap V( K_{n_{l_0'}}) \cap V(K_{n_{l_0''}}) \neq \emptyset$.
\end{enumerate}
Then $\II_G$ is not unmixed.
\end{corollary}

\begin{proof}
Since $m(L) >0$, $L\in \Gamma_1^0$. Choose $L_2 \in M(\Gamma_1^0)$ so that it contains $L$. Now $\{ l_0, l_0', l_0''\} \in M( \Gamma_1^0)$ due to condition (iii). We take $L_1 \in \Gamma_1^0$ so that it contains $\{ l_0, l_0', l_0''\}$. Now we run the algorithm with $L_1^0 = L_2$ and $L_2^0 = L_1$. Then $x_1^0 \notin S$ (by (i)). Since $S\subset V(K_{n_{l_0}})$ and $l_0\in L_1^0\cap L_2^0$, for every vertex $v\in S$, $\{v, x_1^0, x_2^0\}$ gives a $3$-cycle. Hence no vertices of $H_n$ other than $\{x_2^0\}$ belong to $S$, i.e.,  $V(H_n)\cap S= \{x_2^0\}$. Thus from \Cref{prelem:G-K-has-3-comp new}(1), $\II_G$ is not unmixed.
\end{proof}

Now we consider all those non-bipartite chordal graphs which has a clique $K_{r_i}$, $i\in [t]$, such that $G\setminus V(K_{r_i})$ has at least three components. We show that for such a graph $G$, $\II_G$ is not unmixed. We fix some notation to start with. 
Recall that for any set $S\subseteq V(G)$, $c_G(S)$ denotes the number of connected components of $G\setminus S$.
\begin{notation}\label{notn:i1i2i3}
Let 
\[ i_2 = \min \{ j\in \{2,\ldots,t\}: c_G(V(K_{r_j})) \geq 3 \}\] 
and $i_1\in \lambda(i_2)$. Let $C$ be a connected component of $G\setminus V(K_{r_{i_2}})$ such that $C \cap (K_{n_{i_1}}\cup K_{n_{i_2}}) = \emptyset$. Since $c_G(V(K_{r_{i_2}})) \geq 3$, such a component $C$ exists. Let
\[  i_3 = \min\{j\in[t]: K_{n_j}\cap C \neq \emptyset ~\text{ and }~ K_{n_j}\cap K_{r_{i_2}} \neq \emptyset \}.\]
\end{notation}

\begin{remark} \label{remark on choices of i}
   Since $i_1\in \lambda(i_2)$, $K_{n_{i_1}} \cap K_{n_{i_2}} = K_{r_{i_2}}$. From the choice of $i_3$, $K_{n_{i_3}} \cap K_{r_{i_2}} \neq \emptyset$ and hence $K_{r_{i_2}} \supseteq K_{n_{i_1}} \cap K_{n_{i_2}} \cap K_{n_{i_3}} \neq \emptyset $. Note that $K_{n_{i_1}} \setminus V(K_{r_{i_2}})$, $K_{n_{i_2}} \setminus V(K_{r_{i_2}})$ and $K_{n_{i_3}} \setminus V(K_{r_{i_2}})$ are in distinct components of $G\setminus V(K_{r_{i_2}})$. 
   \end{remark}

The next proposition, along with \Cref{x and y adjacent}, is used in the proof of \Cref{clemma_8}.

 \begin{proposition} \label{one compo non-bip}
Let $l\in [t]$ and $\mathcal{S}\subset V(K_{n_{l}})$ be such that $V(K_{n_l}\setminus \mathcal{S}) \cap V(K_{r_{i_2}}) \neq \emptyset$. If $| V(K_{r_{i_2}}) \cap \mathcal{S}| = 1$, then the connected component of $G\setminus \mathcal{S}$ containing $K_{n_{l}} \setminus \mathcal{S}$ is non-bipartite.
\end{proposition}
\begin{proof}
   Let $C'$ be the connected component containing $K_{n_{l}} \setminus \mathcal{S}$. Since $V(K_{n_l}\setminus \mathcal{S}) \cap V(K_{r_{i_2}}) \neq \emptyset$, $(K_{n_{i_1}} \cup K_{n_{i_2}}) \setminus \mathcal{S} \subseteq C'$. Since $G[\mathcal{S}]$ is a clique, and $K_{n_{i_1}} \setminus V(K_{r_{i_2}})$ and $K_{n_{i_2}} \setminus V(K_{r_{i_2}})$ are in distinct components of $G\setminus V(K_{r_{i_2}})$, it follows that $\mathcal{S}\cap V(K_{n_{i_j}} \setminus V(K_{r_{i_2}})) = \emptyset $ for some $j\in \{1,2\}$. Since $r_{i_2} \geq 3$ and $| V(K_{r_{i_2}}) \cap \mathcal{S}| = 1$, $|V(K_{r_{i_2}}\setminus \mathcal{S})| \geq 2$. Since $V(K_{r_{i_2}}) \subset V(K_{n_{i_j}})$, $K_{n_{i_j}} \setminus \mathcal{S}$ has cardinality at least $3$, and hence $C'$ is non-bipartite.
\end{proof}

Recall that $M(\Gamma_1^0)$ is the collection of all maximal subsets $L$ of $[t]$ such that $m(L)>0$. We choose and fix $L_1 \in M(\Gamma_1^0)$ such that $\{ i_1, i_2, i_3\} \subseteq L_1$ for the rest of this section. 
We define the set
\[\Lambda_c = \{l \in [t] \setminus L_1 ~:~ V(K_{n_l}) \cap V(K_{n_{i_c}} \setminus V(K_{r_{i_2}})) \neq \emptyset ~ \text{and} ~ V(K_{n_l}) \cap V(K_{r_{i_2}})= \emptyset\},\]
where $c\in [3]$, in order to study how the complete graphs $K_{n_l}$, where $l \notin L_1$, intersect with $K_{n_{i_1}}$, $K_{n_{i_2}}$ or $K_{n_{i_3}}$ under the condition that $\II_G$ is unmixed. The next lemma shows that if there is an $l \not\in L_1$ such that $K_{n_l}$ intersects one of $K_{n_{i_1}} , K_{n_{i_2}}$ or $K_{n_{i_3}}$ outside $K_{r_{i_2}}$, then there is an $l' \not\in L_1$ such that $K_{n_{l'}}$ intersects it strictly outside $K_{r_{i_2}}$, i.e. $l' \in \Lambda_c$. This statement is crucially used to invoke \Cref{prelem:G-K-has-3-comp new} in the proof of \Cref{cthm_3},  and show that if $r_i \geq 3$, then $\II_G$ is not unmixed.
\begin{lemma}\label{clemma_8}
Let $\II_G$ be unmixed and $r_{i_2} \geq 3$. If $\Lambda_c = \emptyset$, then $K_{n_l} \cap (K_{n_{i_c}} \setminus V(K_{r_{i_2}})) = \emptyset$ for every $l \notin L_1$.
\end{lemma}
\begin{proof}  
Let us assume that for some $c\in [3]$, $\Lambda_c = \emptyset$ but $K_{n_l} \cap (K_{n_{i_c}} \setminus V(K_{r_{i_2}})) \neq \emptyset$ for some $l\notin L_1$. Take $\mathcal{L}_c = \{ l\notin L_1: K_{n_l} \cap (K_{n_{i_c}} \setminus V(K_{r_{i_2}})) \neq \emptyset \}$. Since $\Lambda_c = \emptyset$, $K_{n_{l}} \cap K_{r_{i_2}} \neq \emptyset$ for all $l\in \mathcal{L}_c$. Let $l_c= \min \mathcal{L}_c$. Then $K_{n_{l_c}} \cap K_{r_{i_2}} \neq \emptyset$. Let $y \in V(K_{n_{l_c}} \cap K_{r_{i_2}}) $. Let $x \in \cap_{l\in L_1}V(K_{n_l})$. So, $x \in V(K_{r_{i_2}})$. 
\\
 \textbf{Case 1 - $l_c> i_c$:} \\
In this case, we first show that $\lambda(l_c) \subseteq L_1$.
Let $s'\in \lambda(l_c)$. Since $l_c= \min \mathcal{L}_c$ and $s'<l_c$, $s'\notin \mathcal{L}_c$. However, $K_{n_{l_c}} \cap K_{n_{i_c}} \subseteq K_{r_{l_c}} \subseteq K_{n_{s'}}$ and hence
\begin{align*}
   K_{n_{s'}} \cap (K_{n_{i_c}} \setminus V(K_{r_{i_2}})) 
 & \supseteq & K_{r_{l_c}} \cap (K_{n_{i_c}} \setminus V(K_{r_{i_2}})) & \supseteq & K_{n_{l_c}} \cap K_{n_{i_c}} \cap (K_{n_{i_c}} \setminus V(K_{r_{i_2}})) \\
 &  & &= &  K_{n_{l_c}} \cap (K_{n_{i_c}} \setminus V(K_{r_{i_2}}))
 \neq \emptyset . \label{eq:eq1}\tag{*}
 \end{align*}
Thus, $s'$ satisfies the defining identity of $\mathcal{L}_c$ and hence it follows that $s' \in L_1$. So, $\lambda(l_c) \subseteq L_1$.

We take $\mathcal{S} = V(K_{r_{l_c}})$. If $ |\mathcal{S}| = 1$, then from equation \eqref{eq:eq1}, $ \mathcal{S} =  \mathcal{S}\cap (K_{n_{i_c}} \setminus V(K_{r_{i_2}}) )$. This implies $\mathcal{S} \subset V(K_{n_{i_c}})$ and $S\cap V(K_{r_{i_2}}) = \emptyset$. Since $y \in V(K_{n_{l_c}} \cap K_{r_{i_2}}) $, $y \notin \mathcal{S}$. Observe that $L_1$ is maximal and $l_c \not\in L_1$, so, $x\notin V(K_{n_{l_c}})$. Since $\mathcal{S} \subset V(K_{n_{l_c}})$, $x\notin \mathcal{S}$. Hence we obtain $x, y \in V(K_{r_{i_2}}) \setminus \mathcal{S}$.  Invoking \Cref{x and y adjacent} with $i = i_c$ and $j = l_c$, we get a contradiction. So, $|\mathcal{S}| \geq 2$.  

Let us fix some element $s_0 \in \lambda(l_c)$.  Note that $\mathcal{S}\in \mathcal{D}(G)$, and $K_{n_{l_c}}\setminus \mathcal{S}$ and $K_{n_{s_0}}\setminus \mathcal{S}$ are in distinct connected components of $G\setminus \mathcal{S}$, say $C_1$ and $C_2$ respectively.  

\vskip 1mm \noindent
Claim: $c_G(\mathcal{S}) \geq 3$. \\
Proof of claim: Since $\II_G$ is unmixed, if $|\mathcal{S}| \geq 3$, then $b_G(\mathcal{S}) \geq 3$ and hence $ c_G(\mathcal{S}) \geq 3$. So we may assume that $|\mathcal{S}| = 2$. We know that $V(K_{n_{l_c}} \cap K_{n_{i_c}}) \subseteq \mathcal{ S}$. 
By equation \eqref{eq:eq1}, $|\mathcal{S} \cap V(K_{n_{i_c}} \setminus V(K_{r_{i_2}}))| \geq 1 $. Hence, $|\mathcal{S}\cap V(K_{r_{i_2}})| \leq 1$. \\
Subcase 1: Suppose $|\mathcal{S}\cap V(K_{r_{i_2}})| = 1$. Note that $x \in \cap_{l\in L_1}V(K_{n_l}) \subseteq K_{n_{i_1}} \cap K_{n_{i_2}} \cap K_{n_{i_3}} \subseteq K_{r_{i_2}}$. We have already seen that $x \not\in K_{n_{l_c}}$ and hence does not belong to $\mathcal{S}$. Hence, $x \in (K_{n_{i_c}} \setminus \mathcal{S}) \cap K_{r_{i_2}}$. Therefore \Cref{one compo non-bip} applies with $l= i_c$, and we see that $C_2$ is non-bipartite. Since $\II_G$ is unmixed, $b_G(\mathcal{S})= |\mathcal{S}| = 2$. So, apart from $C_1$ and $C_2$, there exists at least one more connected component of $G\setminus \mathcal{S}$ which is bipartite. Hence $c_G(\mathcal{S}) \geq 3$. \\
Subcase 2: Suppose $|\mathcal{S}\cap V(K_{r_{i_2}})| = 0$. Recall that $x\in \cap_{l\in L_1}V(K_{n_l}) \setminus \mathcal{S}$. Since $s_0\in \lambda(l_c) \subseteq L_1$, it follows that $x \in V(K_{s_0}) \setminus \mathcal{S}$ and hence that 
$\cup _{j\in L_1} K_{n_j} \setminus \mathcal{S} \subseteq C_2$. Thus $K_{r_{i_2}} \subset K_{n_{i_1}} \subset C_2$. Therefore $C_2$ is non-bipartite as $r_{i_2} \geq 3$. Similar to subcase 1, it now follows that there exists at least one more connected component of $G\setminus \mathcal{S}$ which is bipartite. Hence $c_G(\mathcal{S}) \geq 3$.

Let $C_3$ be a connected component of $G\setminus \mathcal{S}$ different from $C_1$ and $C_2$. Suppose $l_c^\prime\in [t]$ is such that $K_{n_{l_c^\prime}}\cap C_3\neq \emptyset$ and $V(K_{n_{l_c^\prime}})\cap \mathcal{S} \neq \emptyset$. Then $V(K_{n_{l_c^\prime}}) \cap V(K_{n_{l_c}}) \cap V(K_{n_{s_0}}) = V(K_{n_{l_c^\prime}}) \cap \mathcal{S} \neq \emptyset $. Recall that $L_1$ is a maximal subset of $[t]$ with $\cap_{l\in L_1}V(K_{n_l}) \neq \emptyset$. Since $l_c \notin L_1$, and $\mathcal{S} = V(K_{r_{l_c}}) \subset V(K_{n_{l_c}})$, $(\cap_{l\in L_1}V(K_{n_l}))\cap \mathcal{S} = \emptyset$. Since $\mathcal{S} \subseteq K_{n_{s_0}}$, \Cref{prelem:G-K-has-3-comp} applies with $L= L_1$, $l_0 = s_0$, $l_0'= l_c'$, $l_0''= l_c$ and $S= \mathcal{S}$ showing that $\II_G$ is not unmixed, which is a contradiction.\\
 
\noindent \textbf{Case 2 - $l_c< i_c$:} \\
\noindent\textbf{ c = 1:} Since $l_1 < i_1$, $K_{n_{l_1}} \cap K_{n_{i_1}} \subseteq K_{r_{i_1}}$. We know that $K_{r_{i_2}}\subset K_{n_{i_1}}$. Since $l_1\in \Lambda_1$, $ K_{n_{l_1}} \cap (K_{n_{i_1}} \setminus V(K_{r_{i_2}})) \neq \emptyset$ and by our assumption $K_{n_{l_1}}\cap K_{r_{i_2}} \neq \emptyset$. Therefore $|V( K_{n_{l_1}} \cap K_{n_{i_1}})| \geq 2$. Hence $|V(K_{r_{i_1}})| \geq 2$. Note that $V(K_{r_{i_1}}) \in \mathcal{D}(G)$ and $\II_G$ is unmixed. So, if $|V( K_{n_{l_1}} \cap K_{n_{i_1}})| > 2$, then $c_G(V(K_{r_{i_1}})) \geq 3$. Since $i_1 <i_2$, this contradicts the choice of $i_2$ (\Cref{notn:i1i2i3}). So, $|V(K_{r_{i_1}})| = 2$. Thus $K_{n_{l_1}} \cap K_{n_{i_1}} = K_{r_{i_1}}$. This implies $l_1\in \lambda(i_1)$ (\Cref{definition of lambda}). In this case choose $\mathcal{S} = V(K_{r_{i_1}}) \in \mathcal{D}(G)$. Then $K_{n_{i_1}} \setminus \mathcal{S}$ and $K_{n_{l_1}} \setminus \mathcal{S}$ are in different connected components of $G\setminus \mathcal{S}$ say $C_1'$ and $C_2'$ respectively. Note that $|\mathcal{S}| =2$.

We now check that the hypotheses of \Cref{one compo non-bip} are satisfied. Recall first the vertex $x$ chosen in $\cap_{l\in L_1}V(K_{n_l})$. Since $L_1$ is maximal and $l_1\notin L_1$, it follows that $ x \notin K_{n_{l_1}}$. Hence, $x \notin V(K_{r_{i_1}})$ and so $x\in (V(K_{n_{i_1}}) \setminus \mathcal{S}) \cap V(K_{r_{i_2}})$. In particular, $(V(K_{n_{i_1}}) \setminus \mathcal{S}) \cap V(K_{r_{i_2}}) \neq \emptyset$.
Note further that $$\mathcal{S} = V(K_{n_{l_1}}) \cap V(K_{n_{i_1}}) = \big(V(K_{n_{l_1}}) \cap (V(K_{n_{i_1}}) \cap V(K_{r_{i_2}})\big) \sqcup \big(V(K_{n_{l_1}}) \cap V(K_{n_{i_1}} \setminus V(K_{r_{i_2}}))\big) . $$ As observed earlier, both sets on the right are non-empty and since it is a disjoint union, and $|\mathcal{S}| = 2$, it follows that $|V(K_{n_{l_1}}) \cap V(K_{n_{i_1}}) \cap V(K_{r_{i_2}})| =1$, i.e., $|\mathcal{S}\cap V(K_{r_{i_2}})| = 1$. Hence, applying 
\Cref{one compo non-bip}, we obtain that $C_1'$ is non-bipartite. 

Since $\II_G$ is unmixed, $b_G(\mathcal{S}) = 2$. Hence $c_G(\mathcal{S}) \geq 3 $. Let $C_3'$ be a component different from $C_1'$ and $C_2'$. Choose $l_1^\prime\in [t]$ such that $K_{n_{l_1^\prime}}\cap C_3'\neq \emptyset$ and $K_{n_{l_1^\prime}}\cap \mathcal{S} \neq \emptyset$. Since $\mathcal{S}\subseteq V(K_{n_{l_1}})$, $L_1$ is maximal and $l_1\notin L_1$, it follows that $(\cap_{l\in L_1}V(K_{n_l})) \cap \mathcal{S} = \emptyset$. Now taking $L=L_1$, $l_0 = i_1$, $l_0'= l_1'$, $l_0''= l_1$ and $S= \mathcal{S}$ in \Cref{prelem:G-K-has-3-comp}, we conclude that $\II_G$ is not unmixed, which is a contradiction.

\noindent\textbf{c = 2:}  If $l_2 <i_2$, then $K_{n_{l_2}} \cap K_{n_{i_2}} \subseteq K_{r_{i_2}}$. Since $l_2 \in \mathcal{L}_2$, $K_{n_{l_2}}\cap( K_{n_{i_2}} \setminus V(K_{r_{i_2}})) \neq \emptyset$, a contradiction.

\noindent\textbf{c = 3:} Since $l_3\in \mathcal{L}_3$, $K_{n_{l_3}}\cap( K_{n_{i_3}} \setminus V(K_{r_{i_2}})) \neq \emptyset$. Since $\Lambda_c = \emptyset$, $K_{n_{l_3}} \cap K_{r_{i_2}} \neq \emptyset$. Hence from the choice of $i_3$, $l_3>i_3$. So, this case cannot occur.
\end{proof} 

\begin{corollary} \label{Coro for two Lambda_c empty}
Let $\II_G$ be unmixed and $\Lambda_1= \Lambda_2 = \emptyset$. If for some run $\mathcal{L}$ of the algorithm $L_1^0(\mathcal{L}) = L_1$, then  $V(H_{n(\mathcal{L})}) \cap V(K_{r_{i_2}}) = \{x_1^0\}$.
\end{corollary}

\begin{proof}
By choice of $L_1^0(\mathcal{L})$, $x_1^0 \in \cap_{l\in L_1}V(K_{n_l}) \subseteq V(K_{n_{i_1}} \cap K_{n_{i_2}} \cap K_{n_{i_3}})\subseteq V(K_{r_{i_2}})$. Suppose there exists $x_q^p \in V(H_{n(\mathcal{L})}) \setminus \{x_1^0\} $ such that $x_q^p \in V(K_{r_{i_2}})$. We know that $K_{r_{i_2}} \subseteq K_{n_{i_1}}\cap K_{n_{i_2}}$. So, for $j\in \{1,2\}$, if $x\in V(K_{n_{i_j}}) \setminus \{ x_1^0, x_q^p \}$, then $\{x, x_1^0, x_q^p \}$ forms a $3$-cycle. Since $H_{\n({\mathcal{L})}}$ is a path (\Cref{H_n is path}), $x\notin V(H_{\n({\mathcal{L})}})$. So, $x\in S(G)$. Hence $V(K_{n_{i_j}} \setminus V(K_{r_{i_2}})) \subseteq V(K_{n_{i_j}}) \setminus \{ x_1^0, x_q^p \} \subseteq S(G)$. \\

\noindent
Claim: $V(K_{n_{i_j}} \setminus V(K_{r_{i_2}})) \subseteq S_2(G)$ for $j\in \{1,2\}$. \\
Proof of claim: 
We have proved that $V(K_{n_{i_j}} \setminus V(K_{r_{i_2}})) \subseteq S(G)$. Now $S(G)= S_2(G) \sqcup S_0(G)$ and $S_0(G) \subseteq \cup_{l \in A^n_0} V(K_{n_l})$. So, to prove the claim, it is sufficient to show that for any $l\in A_0^{\n(\mathcal{L})}$, $K_{n_l} \cap (K_{n_{i_j}} \setminus V(K_{r_{i_2}})) = \emptyset$. Let  $l\in A_0^{\n(\mathcal{L})}$. Note that $x_1^0, x_q^p \in V(K_{r_{i_2}})$ implies $|V(K_{r_{i_2}})| \geq 2$.
 
\noindent \textbf{Case 1: $|V(K_{r_{i_2}})|\geq 3$:} Note that $l \in A_0^0$ (\Cref{remarks on algorithm}(2)). So, from the definition of $A_0^0$, $l\notin L_1^0(\mathcal{L}) = L_1$. Since $\Lambda_1 = \Lambda_2 = \emptyset$ and $l\notin L_1$, by \Cref{clemma_8} $K_{n_l} \cap (K_{n_{i_j}} \setminus V(K_{r_{i_2}})) = \emptyset$ for $j= 1,2$. 

\noindent \textbf{Case 2: $|V(K_{r_{i_2}})|= 2$:} 
 Then $V(K_{r_{i_2}}) = \{x_1^0, x_q^p\}$. Therefore $V(K_{r_{i_2}}) \subseteq V(H_{\n(\mathcal{L})})$. Then $V(K_{n_l})\cap V(K_{r_{i_2}}) = \emptyset$ ( see \Cref{ccoro_1}). Since $\Lambda_1= \Lambda_2 = \emptyset$, $V(K_{n_l}) \cap V(K_{n_{i_j}} \setminus V(K_{r_{i_2}}))$ should be empty for $j= 1,2$.

Since both $K_{n_{i_1}} \setminus V(K_{r_{i_2}})$ and $K_{n_{i_2}} \setminus V(K_{r_{i_2}})$ are non-empty, and their vertex sets are disjoint (\Cref{remark on choices of i}), $|S_2(G)| \geq 2$. This contradicts \Cref{S_2 cardinality 1} because $\II_G$ is unmixed.
\end{proof}


Now we are ready to prove one of the important theorems of this section.

\begin{theorem} \label{cthm_3}
If there exists $i\in [t]$ such that $c_G(V(K_{r_i})) \geq 3$, then $\mathcal{I}_G$ is not unmixed.
\end{theorem}

\begin{proof}
We prove the assertion by contradiction. Suppose $\mathcal{I}_G$ is unmixed. 
Let $\{i_1,i_2,i_3\}$ be as introduced in \Cref{notn:i1i2i3} and $L_1 \in \Gamma_1^0$ the set that contains $\{i_1, i_2, i_3\}$.\\ 

\noindent \textbf{Case 1 : $\Lambda_c \neq \emptyset $ for some $c\in [3]$:} Let $l_c\in \Lambda_c$. Then $l_c \notin L_1$, $ K_{n_{l_c}} \cap (K_{n_{i_c}} \setminus V(K_{r_{i_2}})) \neq \emptyset$ and $K_{n_{l_c}} \cap K_{r_{i_2}} = \emptyset$. We choose $L_2\in M(\Gamma_1^0)$ so that $i_c, l_c \in L_2$. 
Let us run the algorithm choosing $L_1^0 = L_1$ and $L_2^0 = L_2$. Then $x_1^0 \in V(K_{n_{i_1}} \cap K_{n_{i_2}} \cap K_{n_{i_3}}) \subseteq V(K_{r_{i_2}})$ and $x_2^0 \in V(K_{n_{i_c}} \cap K_{n_{l_c}})$. Since $K_{n_{l_c}} \cap K_{r_{i_2}} = \emptyset$, $x_2^0 \notin V(K_{r_{i_2}})$. The rest of the proof will involve applying \Cref{prelem:G-K-has-3-comp new} with $\{l_0,l_0',l_0''\} = \{i_1,i_2,i_3\}$ (as unordered sets), $l_0=i_c$ and $S = V(K_{r_{i_2}})$. Note that for $c\in[3]$, $K_{n_{i_c}} \setminus V(K_{r_{i_2}})$ are in three different components of $G\setminus S$ (\Cref{remark on choices of i}). To invoke \Cref{prelem:G-K-has-3-comp new}(2), we only need to verify that apart from $x_1^0$, no other $x_q^p$ are in $S$.\\
\noindent\textbf{Subcase 1 : $c= 1$ or $c=2$ :} Then $K_{r_{i_2}} \subset K_{n_{i_c}}$. So, any vertex $v\in V(K_{r_{i_2}})$ is adjacent to both $x_1^0$ and $x_2^0$. Since $H_{\n}$ is a path (\Cref{H_n is path}), $v\notin V(H_{\n})$. Therefore, $V(H_n) \cap S = \{x_1^0\}$. Thus, \Cref{prelem:G-K-has-3-comp new}(2) applies and hence $\II_G$ is not unmixed yielding a contradiction. \\
\noindent\textbf{Subcase 2 :  $c \neq 1,2$ :} This means $\Lambda_1 = \Lambda_2 =\emptyset$ and $\Lambda_3 \neq \emptyset$. Invoking \Cref{Coro for two Lambda_c empty}, we see that $V(H_n) \cap S = \{x_1^0\}$. Hence \Cref{prelem:G-K-has-3-comp new}(2) applies, resulting in a contradiction to the unmixed property of $\II_G$.\\

\noindent \textbf{Case 2 : $\Lambda_c = \emptyset $ for all $c\in [3]$:} In this case we choose $L_2\in M(\Gamma_1^0)$ such that $L_2 \neq L_1$ and $L_2\cap L_1 \neq \emptyset$. Since $m([t]) = 0$, $L_2$ exists (\Cref{rmk-exists-L02}). Let $s\in L_1\cap L_2$. Suppose $C_1$, $C_2$ and $C_3$ be the three distinct components of $G\setminus V(K_{r_{i_2}})$ containing $K_{n_{i_1}} \setminus V(K_{r_{i_2}}) $, $K_{n_{i_2}} \setminus V(K_{r_{i_2}})$ and $K_{n_{i_3}} \setminus V(K_{r_{i_2}})$. Since $K_{n_s}\setminus V(K_{r_{i_2}})$ is a connected subgraph, it can intersect at most one of $C_1$, $C_2$, or $C_3$. Without loss of generality assume that $K_{n_s}\setminus V(K_{r_{i_2}})$ does not intersect with the components $C_2$ and $C_3$. 
Then we take $l_0= s$, $\{l_0', l_0''\}= \{i_2,i_3\}$ and $S=V(K_{r_{i_2}})$ in \Cref{prelem:G-K-has-3-comp new}. For any run $\mathcal{L}$ of the algorithm with $L_1^0(\mathcal{L}) = L_1$ and $L_2^0(\mathcal{L}) = L_2$, $x_1^0 \in \cap_{l\in L_1}V(K_{n_l}) \subseteq V(K_{r_{i_2}})$. Since $\Lambda_1 = \Lambda_2 = \emptyset$, by \Cref{Coro for two Lambda_c empty}, $V(H_{n(\mathcal{L})}) \cap S = \{x_1^0\}$. Hence invoking \Cref{prelem:G-K-has-3-comp new}(2), we get a contradiction.
\end{proof}

An immediate consequence of the above theorem is that a chordal graph in which there exists $i$ such that $r_i \geq 3$ cannot be unmixed.
We now consider chordal graphs $G$ where any two complete graphs intersect on a single vertex. \new{This graph class is well known in literature as block graphs.

\begin{definition}\cite{Harary_1963}\label{def:block_graph}
    A graph is a block graph if and only if every block in the graph is a complete graph (a clique).
\end{definition}

The next theorem \Cref{r_i=1} shows that if the graph $G$ follow the above mentioned assumptions and it is a block graph, then $\II_G$ is not unmixed.} Note that if there exist at least three complete graphs with common intersection then this follows from \Cref{cthm_3}. Otherwise, at most two complete graphs intersect in the clique sum. This condition, along with the unmixed property of $\II_G$, allows us to derive the possible structures of $G$. Based on these structures, we construct certain disconnector sets and show that they do not satisfy the unmixed condition, i.e., $b_G(T) = |T|$. One can easily verify that the disconnector sets $T_0, \ldots, T_4$ and others, chosen in the proof of \Cref{r_i=1} satisfy the sign-split property due to \Cref{cremark_5}.

\begin{theorem}\label{r_i=1}
    Let $|V(K_{r_i})| = 1$ for all $i\in [t]\setminus \{1\}$. Then $\mathcal{ I}_G$ is not unmixed.
\end{theorem}

\begin{proof}
Let us assume that $\II_G$ is unmixed. If at least three complete graphs intersect together in $G$ then there exists $i\in[t]$ so that $c_G(V(K_{r_i})) \geq 3$, contradicting \Cref{cthm_3}. Therefore at most two complete graphs can intersect in $G$. Let $i\in [t]$ be such that ${n_i}= \max \{ n_j: j\in [t]\}$.
If $n_i= 3$, then $G$ is a non-bipartite cactus graph. We have proved that a non-bipartite cactus graph is unmixed if and only if $G = K_3$,  \Cref{cm Cactus graph}. Since we already assumed that $\cap_{j\in [t]} K_{n_j} = \emptyset $, the graph $G$ cannot be $K_3$. Hence, once again we see that $\mathcal{I}_G$ is not unmixed, contradicting our initial assumption. Therefore, $n_i \geq 4$. 
%

\new{Note that the degree of every vertex of $K_{n_i}$ is at least $n_i-1$. Let $\mathcal{N}=\{ v\in V(K_{n_i}) : \deg_G v =n_i-1 \} $. For any $x \in V(K_{n_i}) \setminus \mathcal{N}$,  $\deg_G x \geq n_i$} and hence there exists another complete graph that contains the vertex $x$. We denote this complete graph by $K_{n_x}$ (for example see \Cref{fig:n_i= 4}). Since $V(K_{n_{x}}) \cap V(K_{n_i}) = \{x\}$, $x$ is a cut vertex, and hence $\{x\} \in \mathcal{D}(G)$. Moreover,  $G\setminus \{x\}$ has two components, one containing $K_{n_i} \setminus \{x\}$ and the other containing $K_{n_{x}} \setminus \{x\}$. Also, $b_G(\{x\}) = 1$ since $\II_G$ is assumed to be unmixed. Since $n_i -1 \geq 3$, the component containing $K_{n_{x}} \setminus \{x\}$ must be bipartite. We denote this component by $C_x$.
Therefore, $K_{n_{x}}$ is either $K_2$ or $K_3$ and only tree graphs are be attached to  $V(K_{n_{x}} \setminus \{x\})$. Let $y\in V(K_{n_{x}} \setminus \{x\})$ and consider the tree graphs attached to $y$. Note that if a nonempty tree graph is attached to $y$, then $\{y\} \in \mathcal{D}(G)$. Since $b_G(\{y\}) = 1$, there must be exactly one such tree attached to $y$. Further, if $y'$ is a vertex of this tree and $\deg_G(y') \geq 2$, then $\{y'\} \in \mathcal{D}(G)$. So, $b_G(\{y'\}) = 1$ implies that $\deg_G(y') = 2$. Hence that tree must be a path. We denote the path attached to $y$ by $P_y$ (see \Cref{fig:n_i= 4} for an illustration).
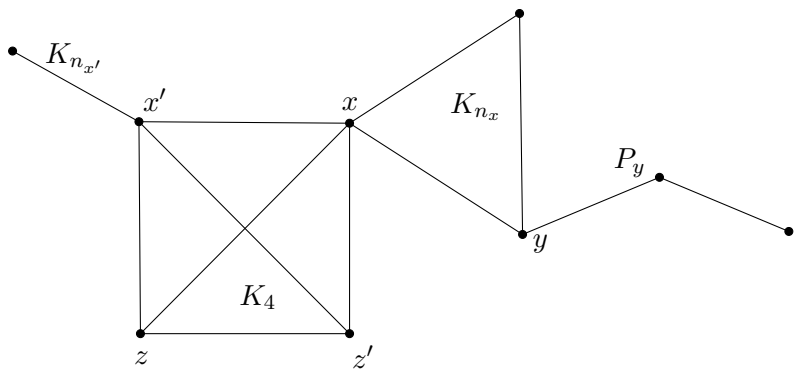
\begin{figure}[H]
    \centering
   \begin{tikzpicture}[line cap=round,line join=round,>=triangle 45,x=1cm,y=1cm]
\clip(-6,-1) rectangle (7,5);
\draw  (-2.8,2.82)-- (-2.78,0);
\draw  (-2.78,0)-- (0,0);
\draw  (-2.8,2.82)-- (0,2.8);
\draw  (0,2.8)-- (0,0);
\draw  (-2.8,2.82)-- (0,0);
\draw  (-2.78,0)-- (0,2.8);
\draw  (0,2.8)-- (2.26,4.26);
\draw  (0,2.8)-- (2.3,1.32);
\draw  (2.26,4.26)-- (2.3,1.32);
\draw  (2.3,1.32)-- (4.12,2.08);
\draw  (4.12,2.08)-- (5.84,1.36);
\draw  (-2.8,2.82)-- (-4.48,3.76);
\draw (-0.24,3.28) node[anchor=north west] {$x$};
\draw (1.2,3.34) node[anchor=north west] {$K_{n_x}$};
\draw (2.3,1.46) node[anchor=north west] {$y$};
\draw (3.38,2.6) node[anchor=north west] {$P_y$};
\draw (-2.88,3.4) node[anchor=north west] {$x'$};
\draw (-4.18,4) node[anchor=north west] {$K_{n_{x'}}$};
\draw (-1.6,0.78) node[anchor=north west] {$K_4$};
\draw (-3,-0.1) node[anchor=north west] {$z$};
\draw (-0.1,-0.02) node[anchor=north west] {$z'$};
\begin{scriptsize}
\draw [fill=black] (-2.8,2.82) circle (1.5pt);
\draw [fill=black] (-2.78,0) circle (1.5pt);
\draw [fill=black] (0,0) circle (1.5pt);
\draw [fill=black] (0,2.8) circle (1.5pt);
\draw [fill=black] (2.26,4.26) circle (1.5pt);
\draw [fill=black] (2.3,1.32) circle (1.5pt);
\draw [fill=black] (4.12,2.08) circle (1.5pt);
\draw [fill=black] (5.84,1.36) circle (1.5pt);
\draw [fill=black] (-4.48,3.76) circle (1.5pt);
\end{scriptsize}
\end{tikzpicture}
    \caption{Graph with $K_{n_i}= K_4$ and $\mathcal{N}=\{z,z'\}$}
    \label{fig:n_i= 4}
\end{figure}

\noindent \textbf{Case 1:} $ |\mathcal{N}| \geq 4$. Let $x_1, x_2, x_3, x_4 \in \mathcal{N}$. Observe that the set $T_{0} = V(K_{n_i}) \setminus \{x_1,x_2\}$ belongs to $\mathcal{D}(G)$. The bipartite components of $G\setminus T_0$ are $C_x$, $x\in T_0 \cap \mathcal{N}$ and the edge $\{x_1,x_2\}$. Thus, we get a contradiction to unmixedness since
\[ b_G(T_0) = |T_0 \cap \mathcal{N}| + 1 \leq |T_0\setminus \{x_3, x_4\}| + 1 < |T_0\setminus \{x_3, x_4\}| + 1 = |T_0| - 2 + 1 = |T_0| - 1  . \]

\noindent \textbf{Case 2:} $|\mathcal{N}| = 3$. Let $x_1,x_2,x_3 \in  \mathcal{N}$. Since $n_i \geq 4$, there exists another vertex $x_4\in V(K_{n_i})$ and hence $\deg (x_4) > n_i$. So, $K_{n_{x_4}}$ exists. If $K_{n_{x_4}}= K_2$, then $T_1 = V(K_{n_i}) \setminus \{x_1,x_4\} \in \mathcal{D}(G)$. The bipartite components of $G\setminus T_1$ are $C_x$, $x\in T_1\setminus \{x_2, x_3\}$ and the component containing $\{x_1, x_4\}$. So, $b_G(T_1) = |T_1| -2 +1= |T_1| -1$, a contradiction. If $K_{n_{x_4}}= K_3$, take $y_4 \in V(K_{n_{x_4}}) \setminus \{x_4\}$. Consider $T_2 = T_1 \cup \{y_4\}$. The connected components of $G\setminus T_2$ are $C_x, x\in T_1\setminus \{x_2, x_3\}$, the component containing $\{x_1, x_4\}$ and $P_{y_4}$, if it exists. From the previous discussions, it follows that all of them are bipartite, and that putting back any vertex of $T_2$ makes the component containing $\{x_1, x_4\}$ non-bipartite. Hence $T_2 \in \mathcal{D}(G)$. Counting the numbers now gives a contradiction that
$$b_G(T_2) \leq | \{ x\in T_1\setminus \{x_2, x_3\} \} | + 1 + 1 \leq (|T_1| -2 ) + 2 = |T_1| < |T_2|. $$

\noindent \textbf{Case 3:} $|\mathcal{N}| = 1 \text{ or } 2$. Consider $T_3 = V(K_{n_i}) \setminus \mathcal{N}$. The components of $G\setminus T_3$ are $C_x$, $x\in T_3$ and the single component containing vertices in $\mathcal{N}$. Since $|\mathcal{N}| = 1 \text{ or } 2$, all these components are bipartite. So, $T_3  \in \mathcal{D}(G)$ and $b_G(T_3) = |T_3| +1$, which yields a contradiction.\\

\noindent \textbf{Case 4:} $|\mathcal{N}| = 0$. Then there exists $K_{n_{x}}$, which is $K_2$ \text{or} $K_3$, and $C_x$ as discussed above for every $x\in V(K_{n_i})$.\\
\textbf{Subcase 4.1:} Suppose there exist at least two vertices say $x_1, x_2$ of $K_{n_i}$ such that $K_{n_{x_i}} = K_2$ for each $i$. Consider $T_4 = V(K_{n_i}) \setminus \{x_1,x_2 \} \in \mathcal{D}(G)$. For each $x\in T_4$, $C_x$ is a bipartite component of $G\setminus T_4$, and apart from these, there is another connected component \new{that} contains the the edge $\{x_1,x_2\}$. Hence, $T_4$ is a cut set and in particular $T_4 \in \mathcal{D}(G)$.
Since all of the components are bipartite, $b_G(T_4) = |T_4| + 1$, which is a contradiction to the unmixed property.\\
\textbf{Subcase 4.2:} Suppose there is at most one vertex $y$ in $K_{n_i}$ such that $K_{n_{y}} = K_2$, and for all other vertices $x$ of $K_{n_i}$, $K_{n_{x}} = K_3$. Since $n_i \geq 4$, there exist $x_1,x_2, x_3 \in V(G)$ such that $K_{n_{x_i}} = K_3$ for each $i$. We choose $y_j\in V(K_{n_{x_j}}) \setminus \{x_j\}$ for $j =1,2,3$. For $j_1,j_2\in \{1,2,3 \}$, consider $T_{j_1j_2} = (V(K_{n_i}) \setminus \{x_{j_1},x_{j_2} \}) \sqcup \{y_{j_1},y_{j_2}\}$. The connected components of $G\setminus T_{j_1j_2}$ are $C_x$, for $x\in  T_{j_1j_2} \setminus \{y_{j_1},y_{j_2}\}$, the component containing the edge $\{x_{j_1},x_{j_2}\}$, and possibly the path graphs $P_{y_{j_1}} \setminus \{y_{j_1}\}$ or $P_{y_{j_2}} \setminus \{y_{j_2}\}$, if they exist. Any vertex in $S_{j_1j_2} \setminus \{y_{j_1},y_{j_2}\}$ is a cut vertex, while the vertices $y_{j_1}$ and $y_{j_2}$ make the bipartite connected component containing the edge $\{x_{j_1},x_{j_2}\}$ non-bipartite if put back. Hence, $T_{j_1j_2} \in \mathcal{D}(G)$ and $b_G(T_{j_1j_2}) \geq |T_{j_1j_2}|-1$.  Since $\mathcal{I}_G$ is unmixed, $b_G(T_{j_1j_2}) = |T_{j_1j_2}|$. This implies exactly one of $P_{y_{j_1}}$ or $P_{y_{j_2}}$ exists. Taking $j_1 = 1$ and varying $j_2 \in \{ 2,3 \}$, we see that either both 
$P_{y_2}$ and $P_{y_3}$ exist or neither exist. Now choosing $j_1 = 2 , j_2 = 3$ gives a contradiction to the existence of exactly one of $P_{y_2}$ and $P_{y_3}$.


\end{proof}

\section{Classification of unmixed chordal graphs} \label{sec:unmixed}
Recall that the chordal graph $G$ is a clique sum of $K_{n_1}, K_{n_2} , \ldots, K_{n_t}$. We observed in \Cref{sec:non-unmixed} that the unmixed property of $\II_G$ is very rare for chordal graphs. Specifically, we saw that when $m([t]) = 0$, $\II_G$ is not unmixed under the following conditions:
\begin{enumerate}
    \item $c_G(V(K_{r_i})) \geq 3$.
    \item $|V(K_{r_i})| = 1 \quad \forall i \in [t]$.
\end{enumerate}
We begin by considering chordal graphs $G$ for which $m([t]) \neq 0$. 
\begin{theorem} \label{m([t])> 0}
    If $m([t]) \neq 0$, then $\mathcal{I}_G$ is unmixed if and only if $G = K_3$.
\end{theorem}
\begin{proof}
If $G= K_3$, then by \cite[Theorem 3.5]{Kum21} $\II_G$ is a complete intersection, and hence unmixed. Now we assume that $m([t]) \neq 0$ and $\mathcal{I}_G$ is unmixed. Since $m([t]) \neq 0$, the choice of $L_1^0$ is unique, namely $L_1^0 = [t]$. Then we have $A_1^0 = [t]$ and $A_2^0 = A_0^0 = \emptyset$. This implies $D_0= \{1\}$. So, $H_0$ has a single vertex $x_1^0$. Further, after the first step we get $A_2^1 = A_1^0 = [t]$ and $A_1^1= A_0^1= \emptyset$. Hence, the algorithm stops after the first step. So, $H_n = H_1$ and $S(G) = S_2(G)$ as $A_0^1 = \emptyset$. In this case $H_1$ is the only connected component of $G\setminus S(G)$. This implies $V(G) = V(H_1) \sqcup S(G)$. Note that $H_1$ is a path graph (by \Cref{ccoro_2}) and $V(H_1) = V(H_0) \cup \{x_q^1: q\in D_1\}$. Since $\II_G$ is unmixed and $S(G) \neq \emptyset$, $|S(G) |=1$ (by \Cref{S_2 cardinality 1}). So, it follows from \Cref{cprop_2} that $H_1$ is either $P_2$ or $P_3$. 

   \begin{itemize}
       \item[(1)] Suppose $H_1=P_2$. 
       Then $|V(G)| = |V(H_1)| + |S(G)| = 3$. Since $G$ is non-bipartite, $G= K_3$.

       \item[(2)]  Suppose $H_1 = P_3$. 
       Then $|V(G)| = |V(H_1)| + |S(G)| = 4$. Since $G$ is non-bipartite, it is either a clique sum of $K_3$ and $K_2$ at the vertex $x_1^0$ or a clique sum of two $K_3$ along an edge. In the first case $\{x_1^0\} \in \mathcal{D}(G)$, but $b(G\setminus \{x_1^0\}) =2$, a contradiction. In the second case if $u, v$ are the vertices that are not common to both $K_3$, then $\{u,v\} \in \mathcal{D}(G)$ but $b(G\setminus \{u,v\}) = 1$, a contradiction.
   \end{itemize}  
\end{proof}

Thus, the only non-bipartite chordal graph with $m([t]) \neq 0$ which is unmixed is $K_3$. We are thus left to consider the unmixedness property for non-bipartite chordal graphs which have the property that $m([t]) = 0$, $\max \lbrace |V(K_{r_i})| : i \in [t] \rbrace = 2$ for all $i\in [t]$ and $c_G(V(K_{r_i})) = 2$. We label the subclass of graphs defined by these properties by $\mathfrak{G}$. Note that the graph classes $\mathfrak{G_1}, \mathfrak{G_2}$ and $\mathfrak{G_3}$, as introduced in \Cref{fig:fig 13} are contained in $\mathfrak{G}$. \Cref{I_G unmixed 1} and \Cref{I_G unmixed 2} show that these subclasses are precisely the ones within $\mathfrak{G}$ such that $\II_G$ is unmixed.
Before proceeding to prove these theorems, we note the following special properties of $G \in \mathfrak{G}$.
\begin{remark}\label{rmk of I_G unmixed 1} 
$(1)$ Suppose $G \in \mathfrak{G}$. Let $i\in [t]$ be such that $|V(K_{r_i})| = 2$. Since $V(K_{r_i}) \in \mathcal{D}(G)$ and $c_G(V(K_{r_i})) = 2$, if $\II_G$ is unmixed, then $b_G(V(K_{r_i})) = c_G(V(K_{r_i})) = 2$. Hence, if $|V(K_{r_i})| = 2$, then $G\setminus V(K_{r_i})$ is bipartite and has two components.\\
$(2)$ Suppose $G \in \mathfrak{G}$. Let $j,j'\in [t]$ so that $j \neq j'$. Then $K_{n_j}\cap K_{n_{j'}}$ is contained either in $K_{r_j}$ or in $K_{r_{j'}}$. Since $|V(K_{r_i})| \leq 2$, for all $i\in [t]$, $|V(K_{n_j}\cap K_{n_{j'}})|\leq 2$ and if $|V(K_{n_j}\cap K_{n_{j'}})|=2$, then $K_{n_j}\cap K_{n_{j'}}$ is \new{the} same as $K_{r_j}$ or $K_{r_{j'}}$. Taken along with $(1)$, we get that if $\II_G$ is unmixed and $|V(K_{n_j}\cap K_{n_{j'}})|=2$, then $G\setminus V(K_{n_j}\cap K_{n_j'})$ must be bipartite.  
\end{remark}

\begin{theorem} \label{I_G unmixed 1}
Suppose $G \in \mathfrak{G}$. If there exists $i$ such that $|V(K_{r_i})|= 2$ and at least three complete graphs intersect with $K_{r_i}$ then $\mathcal{I}_G$ is unmixed if and only if $G\in \mathfrak{G}_1$ (\Cref{fig:fig 13}).
\end{theorem}

\begin{proof}
If $G\in \mathfrak{G}_1$, then $\II_G$ is unmixed by \Cref{I_G is unmixed only if G has there forms}. Let us assume that $\mathcal{I}_G$ is unmixed. Since $G\in \mathfrak{G}$, $c_G(V(K_{r_i})) = 2$ for all $i\in [t]$.
Let $i_2 \in [t]$ be such that it satisfies the hypothesis of the theorem, i.e., $|V(K_{r_{i_2}})|= 2$ and at least three complete graphs intersect with $K_{r_{i_2}}$. Let $i_1\in \lambda(i_2)$. Consider $T_1 = V(K_{r_{i_2}}) = \{\alpha_1, \alpha_2\}$. Then by \Cref{rmk of I_G unmixed 1}(1), $G\setminus T_1$ is bipartite and has two components. So, $K_{n_{i_1}}$ and  $ K_{n_{i_2}}$ are either $K_3$ or $K_4$. Let  $C_1$ and $C_2$ be the bipartite components of $G\setminus T_1$ containing $K_{n_{i_1}}\setminus T_1$ and  $K_{n_{i_2}}\setminus T_1$ respectively. 
 Since at least three complete graphs intersect with $K_{r_{i_2}}$, there exists $j\in [t]$ such that $j\neq i_1, i_2$ and $K_{n_j} \cap K_{r_{i_2}}\neq \emptyset$. Since $c_G(T_1) = 2$, either $K_{n_j} \cap C_1 \neq \emptyset $ or $K_{n_j} \cap C_2 \neq \emptyset $. Without loss of generality assume that $K_{n_j} \cap C_2 \neq \emptyset $.\\[2ex]
 \textbf{Claim:} There exists $i_3\in [t]$ such that $K_{n_{i_3}} \cap K_{r_{i_2}} \neq \emptyset$ and $(V(K_{n_{i_3}}) \setminus T_1) \cap (V(K_{n_{i_2}}) \setminus T_1) \neq \emptyset$.\\
 Proof of the claim: Take $K_{n_j}$ as above. Let $ \alpha_l \in V(K_{n_j} \cap K_{r_{i_2}})$ for some $l\in \{1,2\}$. Now suppose $(V(K_{n_j}) \setminus T_1) \cap (V(K_{n_{i_2}}) \setminus T_1) = \emptyset$. 
 Since $C_2$ is connected, there exists a shortest path $P : x = x_0,\ldots, x_n = y$ in $C_2$, where $x \in V(K_{n_j}) \setminus T_1$ and $y \in V(K_{n_{i_2}}) \setminus T_1$. 
 Since $x_{n-1}$ and $y$ are adjacent, there exists $j'\in [t]$ so that $x_{n-1},y \in V(K_{n_{j'}})$. If $j'<i_2$, then $V(K_{n_{j'}}\cap K_{n_{i_2}}) \subseteq V(K_{r_{i_2}}) = T_1$. But $y \notin T_1$ and $y\in V(K_{n_{j'}}\cap K_{n_{i_2}})$. So, $j'>i_2$. Now if $K_{n_{j'}} = K_2$, then $V(K_{r_{j'}}) = \{y\}$ and $i_2\in \lambda(j')$. Since both $x_{n-1}$ and $\alpha_l$ does not belongs to $V(K_{r_{j'}})$, they cannot be adjacent (\Cref{x and y adjacent}). We know that $x_0, \alpha_l \in E(G)$. Let $p = \max \{j \in \{ 0,\ldots, x_{n-2}\}: \{x_j, \alpha_l\} \in E(G) \}$. Then $\{x_p, \ldots, x_{n-1}, x_n, \alpha_l\}$ is an induced $(n-p+2)$-cycle of $G$ and $n-p+2 \geq 4$. Since $G$ is chordal, this is a contradiction. Hence $|V(K_{n_{j'}})| \geq 3$. Since $V(K_{n_j'}\setminus T_1) \subseteq V(C_2)$ and $C_2$ is bipartite, $K_{n_{j'}} \cap T_1 \neq \emptyset$. We take $i_3 = j'$.\\
 Note that in this case two complete graphs intersect in at most two vertices. Since $T_1 \subset V(K_{n_{i_2}})$, $|V(K_{n_{i_3}} \cap T_1)|=1$.  Let $\alpha_3\in V(K_{n_{i_3}} \cap K_{n_{i_2}}) \setminus T_1$ and $\alpha_2 \in V(K_{n_{i_3}}) \cap T_1$. Then
 $V(K_{n_{i_3}}) \cap V(K_{n_{i_2}}) = \{\alpha_2, \alpha_3\}$. Since $C_2$ is bipartite $K_{n_{i_3}}= K_3$. Let $T_2 = \{\alpha_2, \alpha_3\}$. By \Cref{rmk of I_G unmixed 1}(2), $G\setminus T_2$ is also bipartite. So $K_{n_{i_1}} = K_3$. But $K_{n_{i_2}}$ is either $K_3$ or $K_4$.
 Let us assume that $V(K_{n_{i_1}}) = \{ \alpha_1, \alpha_2, \beta_1 \}$, $V(K_{n_{i_3}}) = \{ \alpha_2, \alpha_3, \beta_2 \}$ and $V(K_{n_{i_2}})$ is either $\{ \alpha_1, \alpha_2, \alpha_3 \}$ or $\{  \alpha_1, \alpha_2, \alpha_3, \alpha_4\} $.
\begin{figure}[H]
\centering

\begin{tikzpicture}[line cap=round,line join=round,>=triangle 45,x=1cm,y=1cm]
\clip(-10,-2.5) rectangle (6,5);
\draw  (-7.4,1.67)-- (-5,3);
\draw  (-7.4,1.67)-- (-5.02,0.37);
\draw  (-5,3)-- (-5.02,0.37);
\draw  (-5,3)-- (-2.72,1.91);
\draw  (-5.02,0.37)-- (-2.72,1.91);
\draw  (-5.02,0.37)-- (-2.74,-1.13);
\draw  (-2.72,1.91)-- (-2.74,-1.13);
\draw  (-1.12,1.65)-- (1.2,2.99);
\draw  (1.2,2.99)-- (4,3);
\draw  (4,3)-- (4.04,0.69);
\draw  (1.2,2.99)-- (4.04,0.69);
\draw  (4.04,0.69)-- (2.8,-1.47);
\draw  (-1.12,1.65)-- (1.28,0.71);
\draw  (1.2,2.99)-- (1.28,0.71);
\draw  (1.28,0.71)-- (4.04,0.69);
\draw  (1.28,0.71)-- (2.8,-1.47);
\draw  (1.28,0.71)-- (4,3);
\draw (-5.4,3.5) node[anchor=north west] {$\alpha_1$};
\draw (0.8,3.51) node[anchor=north west] {$\alpha_1$};
\draw (-5.4,0.3) node[anchor=north west] {$\alpha_2$};
\draw (3.66,3.51) node[anchor=north west] {$\alpha_4$};
\draw (3.88,0.69) node[anchor=north west] {$\alpha_3$};
\draw (-8,1.87) node[anchor=north west] {$\beta_1$};
\draw (2.22,-1.3) node[anchor=north west] {$\beta_2$};
\draw (-1.54,2.33) node[anchor=north west] {$\beta_1$};
\draw (-2.8,2.35) node[anchor=north west] {$\alpha_3$};
\draw (-3,-1.22) node[anchor=north west] {$\beta_2$};
\draw (0.8,0.73) node[anchor=north west] {$\alpha_2$};
\draw (2,0.39) node[anchor=north west] {$K_{n_{i_3}}$};
\draw (2,2.71) node[anchor=north west] {$K_{n_{i_2}}$};
\draw (-6.12,2.15) node[anchor=north west] {$K_{n_{i_1}}$};
\draw (0,2.13) node[anchor=north west] {$K_{n_{i_1}}$};
\draw (-3.68,0.47) node[anchor=north west] {$K_{n_{i_3}}$};
\draw (-4.46,2.17) node[anchor=north west] {$K_{n_{i_2}}$};
\draw (-6.66,-0.57) node[anchor=north west] {(A)};
\draw (4.82,-0.51) node[anchor=north west] {(B)};
\begin{scriptsize}
\draw [fill=black] (-7.4,1.67) circle (2pt);
\draw [fill=black] (-5,3) circle (2pt);
\draw [fill=black] (-5.02,0.37) circle (2pt);
\draw [fill=black] (-2.72,1.91) circle (2pt);
\draw [fill=black] (-2.74,-1.13) circle (2pt);
\draw [fill=black] (-1.12,1.65) circle (2pt);
\draw [fill=black] (1.2,2.99) circle (2pt);
\draw [fill=black] (4,3) circle (2pt);
\draw [fill=black] (4.04,0.69) circle (2pt);
\draw [fill=black] (2.8,-1.47) circle (2pt);
\draw [fill=black] (1.28,0.71) circle (2pt);
\end{scriptsize}
\end{tikzpicture}
\caption{Clique sum of $K_{n_{i_1}}$, $K_{n_{i_2}}$ and $K_{n_{i_3}}$}
\label{fig 2}
\end{figure}
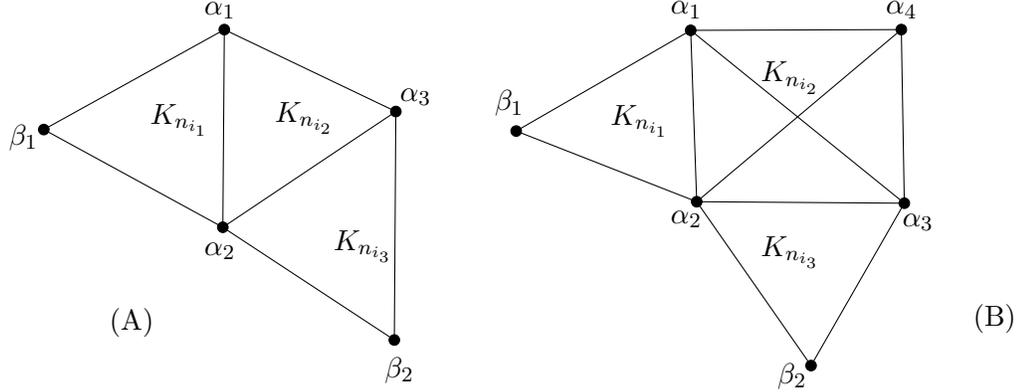
Graphs in \Cref{fig 2}(A) and (B) illustrate these possibilities and the reader may want to refer to them for the arguments that come ahead. Since $G\setminus T_1$ and $G\setminus T_2$ are bipartite, and $c_G(T_1) = c_G(T_2) = 2$, we can determine the attachments of other complete graphs at the vertices of $K_{n_{i_1}}$, $K_{n_{i_2}}$ or $K_{n_{i_3}}$ as follows. Suppose $i\notin \{i_1, i_2, i_3\}$. Then the following hold:
\begin{itemize}
    \item $K_{n_i}$ does not contain $\{ \alpha_1, \alpha_2\}$ or $\{ \alpha_2, \alpha_3\}$.
    \item If $\alpha_1 \in V(K_{n_i})$, then exactly one of $\beta_1$ or $\alpha_3$ or $\alpha_4$ belongs to $V(K_{n_i})$.
    \item If $\alpha_2 \in V(K_{n_i})$, then exactly one of $\beta_2$ or $\beta_1$ or $\alpha_4$ belongs to $V(K_{n_i})$.
    \item If $\alpha_3 \in V(K_{n_i})$ then exactly one of $\beta_2$ or $\alpha_1$ or $\alpha_4$ also belongs to $V(K_{n_i})$.
    \item Since $G\setminus T_1$ and $G\setminus T_2$ are bipartite, $K_{n_i}$ does not contain any of the edges $\{ \alpha_1, \alpha_4\}$, $\{\alpha_3, \alpha_4\}$, $\{\beta_1, \alpha_1\} $ and $\{\beta_2, \alpha_3\}$.
\end{itemize}
Let $e$ be one of the edges $\{\beta_1, \alpha_2\}, \{ \alpha_2, \beta_2\}$ or $\{\alpha_1, \alpha_3\}$. Suppose $e \in E(K_{n_i})$. Then:
\begin{itemize}
    \item $K_{n_i}= K_3$.
    \item If $x \in V(K_{n_i} \setminus e)$, then only tree graphs can be attached to $x$, and further since $\II_G$ is unmixed any such tree must be a path graph. If it exists, we denote the path attached to $x$ by $P_x$.
\end{itemize}
Note that path graphs $P_{\beta_1}$ or $P_{\beta_2}$ attached to $\beta_1$ or $\beta_2$ respectively, may exist, but there is no path graph attached to $\alpha_1$, $\alpha_2$ or $\alpha_3$.  

\noindent\textbf{Case 1: $K_{n_{i_2}} = K_3$:} (As in \Cref{fig 2}(A)). 
For any edge $\{z,w\}$ of $G$, let $L_{z,w} := \{ j\in [t] : \{z,w\}\in E(K_{n_j})\}$. First we want to show that there exists a complete graph $K_{n_i}= K_3$ that contains $\{\alpha_1, \alpha_3\}$. For that purpose we choose 
$$T_3 = \{ \alpha_1, \alpha_3\} \bigcup \bigcup_{i\in L_{\beta_1, \alpha_2}} (V(K_{n_i}) \setminus \{\beta_1, \alpha_2\}) \bigcup \bigcup_{ i \in L_{\beta_2, \alpha_2}} (V(K_{n_i}) \setminus \{\beta_2, \alpha_2\}).$$
Then the path $P$ containing the vertices $\{\beta_1, \alpha_2,\beta_2\}$ is a component of $G\setminus T_3$ and for every $s\in T_3$ $G[V(P)\cup \{s\}]$ is non-bipartite. So, $T_3\in \mathcal{D}(G)$. Then $|T_3| = b_G(T_3)$. Apart from $P$, the other possible components of $G\setminus T_3$ are $P_x\setminus \{x\}$, where $x\in V(K_{n_j}) \setminus \{\beta_1, \alpha_2, \beta_2\}$ and $j\in L_{\beta_1, \alpha_2} \cup L_{\beta_2, \alpha_2}$, and the component containing $K_{n_i} \setminus T_3$ whenever $K_{n_i}$ contains $\alpha_1, \alpha_3$. Since $|T_3| = b_G(T_3)$, there must be some $K_{n_i}= K_3$ such that $\alpha_1, \alpha_3\in V(K_{n_i})$. Let $V(K_{n_i}) = \{ \alpha_1$, $\alpha_3, \beta_3 \} $. Take $T_4 = \{\alpha_1$, $\alpha_3\} \in \mathcal{D}(G)$. Since $G\setminus T_4$ is bipartite, there is no $i'\in [t]\setminus \{i_1,i_2,i_3\}$ such that $V(K_{n_{i'}})$ contains $\{\beta_1, \alpha_2\}$ or $\{\beta_2, \alpha_2\}$. Take $T_5 = \{\beta_1, \alpha_3\} \in \mathcal{D}(G)$, then $\{ \alpha_1,  \alpha_2, \beta_2, \beta_3\}$ is part of a bipartite connected component. Since $b_G(T_5) = 2$, $P_{\beta_1}$ exists. Similarly, $\{ \beta_2, \alpha_1\}$ and $\{ \beta_3, \alpha_2\}$ are in $\mathcal{D}(G)$. So, $P_{\beta_2}$ and $P_{\beta_3}$ exist. Hence $G$ has the form as in Class $\mathfrak{G}_1$ as shown in \Cref{fig:fig 13}.\\ 
\noindent\textbf{ Case 2: $K_{n_{i_2}} = K_4$:} (As in \Cref{fig 2}(B)). We choose $T_6 = \{ \alpha_4, \alpha_3, \beta_1 \}$. Let $B_1$ be the bipartite component that contains $\{\alpha_1, \alpha_2, \beta_2\}$. Then for all $s\in T_6$, $G[V(B_1)\cup \{s\}]$ is non-bipartite. So, $T_6\in \mathcal{D}(G)$. Since $G\setminus T_1$ is bipartite, there does not exist any $K_{n_j}, j\neq i_2$, that contains both $\alpha_4$ and $\alpha_3$. Furthermore, there is no path attached to $\alpha_3$ because $c_G(T_2) = 2$. So, the possible components of $G\setminus T_6$ are $B_1$, $P_{\beta_1}\setminus \{\beta_1\}$ and $P_{\alpha_4} \setminus \{\alpha_4\}$, if $P_{\beta_1}$ and $P_{\alpha_4}$ exist. Since $b_G(T_6) = |T_6| = 3$, both $P_{\beta_1}$ and $P_{\alpha_4}$ must exist. Now if we take $T_7 = \{ \alpha_1, \alpha_3, \alpha_4 \} \in \mathcal{D}(G)$, then it can be seen that $P_{\alpha_4} \setminus \{\alpha_4\}$ and $\{ \beta_1, \alpha_2, \beta_2\}$ are part of two bipartite components of $G\setminus T_7$. So, $b(G\setminus T_7) = 2$, a contradiction to the unmixedness of $\II_G$. Hence this case cannot occur. This completes the proof.
\end{proof}
We are now left with classifying those graphs $G$ in $\mathfrak{G}$ such that for any $x\in V(G)$, $x$ belongs to at most two $K_{n_i}$ for $i\in [t]$. We classify when such graphs have the unmixedness property in the next theorem.
\begin{theorem} \label{I_G unmixed 2}
Let $G\in \mathfrak{G}$ and at most two maximal complete graphs intersect in $G$. Then $\mathcal{I}_G$ is unmixed if and only if $G \in \mathfrak{G}_2 \cup \mathfrak{G}_3$ (\Cref{fig:fig 13}).
\end{theorem}

\begin{proof}
If $G\in \mathfrak{G}_2\cup \mathfrak{G}_3$, then $\II_G$ is unmixed by \Cref{I_G is unmixed only if G has there forms}. Assume now that $G \in \mathfrak{G}$, at most two maximal complete graphs intersect in $G$ and $\mathcal{I}_G$ is unmixed. Let $i_2\in [t]$ be such that $|V(K_{r_{i_2}})| = 2$ and $i_1\in \lambda(i_2)$. Take $T_1= V(K_{r_{i_2}}) = \{\alpha_1, \alpha_2\}$. Since at most two complete graphs intersect in $G$, $K_{n_{i_1}}$ and $K_{n_{i_2}}$ are the only graphs that intersect with $K_{r_{i_2}}$. So, $G\setminus T_1$ has two components containing $K_{n_{i_1}}\setminus T_1$ and $K_{n_{i_2}}\setminus T_1$. Since $\II_G$ is unmixed, both of them are bipartite, i.e., $G\setminus T_1$ is bipartite. This condition along with the given hypothesis imply that $K_{n_{i_1}}$ and $K_{n_{i_2}}$ are either $K_3$ or $K_4$, and only tree graphs can be attached to the vertices of $V(K_{n_{i_1}}) \setminus T_1$ or $V(K_{n_{i_2}}) \setminus T_1$. If there is a tree attached to some $x\in V(K_{n_{i_1}} \cup K_{n_{i_2}}) \setminus T_1$, then $\{x\} \in \mathcal{D}(G)$. Since $\II_G$ is unmixed, $b_G(\{x\}) = 1$. So, at most one tree can be attached to $x$. Again, because of the unmixednes of $\II_G$, such a tree graph must be a path. If it exists, we denote the path attached to $x$ by $P_x$. Since we are in the case where $m([t])= 0$, one such path exists in $G$. Let $x_1 \in V(K_{n_{i_1}}) \setminus T_1 $ be such that $P_{x_1}$ exists.\\

\noindent
\textbf{Case 1: $K_{n_{i_1}}= K_{n_{i_2}}= K_3$:} Then  $V(K_{n_{i_1}})= \{ \alpha_1, \alpha_2, x_1\}$. Let $V(K_{n_{i_2}}) = \{ \alpha_1, \alpha_2, x_2 \}$. If $P_{x_2}$ exists, then the set $\{ x_1, x_2 \} \in \mathcal{D}(G)$. But then $b_G(\{ x_1, x_2 \}) = 3$, a contradiction. So, there is no path attached to $x_2$. Hence $G$ is in Class $\mathfrak{G}_3$, as shown in \Cref{fig:fig 13}.\\

\noindent
\textbf{Case 2: $K_{n_{i_1}} = K_4$ and $K_{n_{i_2}}= K_3$:}
Let $V(K_{n_{i_1}})= \{ \alpha_1, \alpha_2, x_1, y_1\}$ and $V(K_{n_{i_2}}) = \{ \alpha_1, \alpha_2, x_2\}$. Take $T_1 = \{ y_1, \alpha_1 \} \in \mathcal{D}(G)$. Then $b_G(T_1) = 2$. So, $P_{y_1}$ exists. Now if $P_{x_2}$ exists, then we choose  $T_2= \{ x_1, y_1, x_2 \} \in \mathcal{D}(G)$. But then $b_G(T_2) = 4$, a contradiction. Therefore, $P_{x_2}$ does not exists. Hence $G$ is in Class $\mathfrak{G}_2$  as shown in \Cref{fig:fig 13}. \\


\noindent
\textbf{Case 3: $K_{n_{i_1}} = K_3$ and $K_{n_{i_2}}= K_4$:}
Let $V(K_{n_{i_2}}) = \{ \alpha_1, \alpha_2, x_2, y_2\}$ and $V(K_{n_{i_1}})= \{ \alpha_1, \alpha_2, x_1\}$. We take $T_3 = \{ x_1, x_2, y_2\} \in \mathcal{D}(G)$. So, $b_G(T_3) = 3$. This implies exactly one of $P_{x_2}$ or $P_{y_2}$ exists. Without loss of generality, $P_{x_2}$ exists. Choosing $T_4 =\{y_2, \alpha_2\} \in \mathcal{D}(G)$, we get that $b_G(T_4) = 1$, which is a contradiction to the unmixed property of $\II_G$. Hence this case cannot occur.\\


\noindent
\textbf{Case 4: $K_{n_{i_1}}= K_{n_{i_2}}= K_4$:} Let $V(K_{n_{i_2}}) = \{ \alpha_1, \alpha_2, x_2, y_2\}$ and $V(K_{n_{i_1}})= \{ \alpha_1, \alpha_2, x_1, y_1\}$. The set $T_5 = \{ x_1, y_1, x_2, y_2 \} \in \mathcal{D}(G)$. So, $b_G(T_5) = 4$. This implies exactly one vertex of $T_5$ does not have a path attached. Without loss of generality, assume that $x_2$ is that vertex. Then the set $\{ x_2, \alpha_2, y_1\} \in \mathcal{D}(G)$, but $b_G(\{ x_2, \alpha_2, y_1\} ) = 2$, a contradiction. Therefore, this case cannot occur. This completes the proof.
\end{proof}
\new{Finally, we complete the classification of all chordal graphs whose corresponding parity binomial edge ideals are unmixed.
\begin{corollary}
 Let $G$ be a chordal graph. Then $\II_G$ is unmixed if and only if $G$ is a path graph or $K_3$ or $G\in \mathfrak{G}_1 \cup \mathfrak{G}_2\cup \mathfrak{G}_3$. 
 
\end{corollary}

The above result also classifies the unmixed parity binomials associated with two other subclasses of chordal graphs such as block graphs and generalized block graphs. We define block graphs in Definition \ref{def:block_graph}. The generalized block graphs are defined below. This graph class is prominently studied in the field of combinatorial commutative algebra (\cite{kiani2013regularity}, \cite{chaudhry2017generalized}, \cite{AR2}). 
\begin{definition}
    A graph $G$ is a generalized block graph if it satisfies:
   \begin{enumerate}
       \item $G$ is chordal.
       \item For any three maximal cliques $K_1, K_2, K_3$ in the graph, if their total intersection is non-empty $(K_1 \cap K_2 \cap K_3 \neq \emptyset)$, then their pairwise intersections must all be equal, i.e., $$K_1 \cap K_2 = K_2 \cap K_3 = K_1 \cap K_3.$$
    \end{enumerate} 
\end{definition}

\begin{corollary}
Let $G$ be a simple graph and $\II_G$ be unmixed. Then,
\begin{itemize}
     \item $G$ is a block graph if and only if $G$ is a path or $K_3$.
     \item $G$ is a generalized block graph if and only if $G$ is a path or $K_3$ or $G\in \mathfrak{G}_2 \cup \mathfrak{G}_3$.
 \end{itemize}

\end{corollary}

}
\section{Cohen-Macaulay parity binomial edge ideals}

In the previous section, we characterized unmixed parity binomial edge ideals of non-bipartite chordal graphs. We now characterize the Cohen-Macaulayness. Since any Cohen-Macaulay ideal is unmixed, we need to understand which are the Cohen-Macaulay ones among the graphs given in \Cref{sec:unmixed}. First we prove a result that allow us to reduce the study to the basic structures.
\begin{proposition}
Let $G$ be a graph on $[n]$ having a pendant vertex $u$ and $G'$ be the graph obtained by adding a whisker to $u$. If $R$ denotes the polynomial ring containing $\II_{G'}$, then $\depth(R/\II_{G'}) = \depth(R/\II_G)-1$.
\end{proposition}

\begin{proof}
Let $\{u,n+1\}$ be the new edge in $G'$ and $g = x_ux_{n+1}-y_uy_{n+1}$. Then $\II_{G'} = \II_G + (g)$. Since $u$ is the pendant vertex of $G$, $u\notin T$ for every $T\in \mathcal{D}(G)$. Hence $g \notin \mathfrak{p}$, for every $\mathfrak{p} \in \text{Min}(\II_G)$. Therefore $g$ is regular on $\frac{R}{\II_G}$ which yields $\depth(R/\II_{G'}) = \depth(R/\II_G)-1$.
\end{proof}

As an immediate consequence, we see that the Cohen-Macaulayness of $\II_G$ and $\II_{G'}$ are equivalent:
\begin{corollary} \label{cor:cm-whisker}
Let $G$ be a graph on $[n]$ having a pendant vertex $u$ and $G'$ be the graph obtained by adding a whisker to $u$. Then $\II_G$ is Cohen-Macaulay if and only if $\II_{G'}$ is Cohen-Macaulay.
\end{corollary}

\new{We saw in \Cref{sec:unmixed} that the only unmixed parity binomial edge ideals of non-bipartite chordal graphs are either $K_3$ or graphs in $\G_1 \cup \G_2 \cup \G_3$. If $G = K_3$ or $G \in \G_3$, then it follows from \cite[Theorem 3.5]{Kum21} and \cite[Theorem 4.13]{Kum21} that $\II_G$ is Cohen-Macaulay. Therefore, we only need to verify the Cohen-Macaulay property for the graph classes $\mathfrak{G}_1$ and $\mathfrak{G_2}$.} 

It may be observed that graphs in $\G_i$ are obtained by adding paths to the pendent vertices of $G_i$ given in \Cref{fig:cm_2}. Hence, by \Cref{cor:cm-whisker}, any graph $G \in \G_i$ is Cohen-Macaulay if and only if $G_i$ (in \Cref{fig:cm_2}) is Cohen-Macaulay. We now take the help of Macaulay 2, \cite{M2}, to check which of these are Cohen-Macaulay.

\begin{figure}[H]
\centering   
   \begin{tikzpicture}[line cap=round,line join=round,>=triangle 45,x=0.9cm,y=0.9cm]
\clip(-7.5,-4) rectangle (8,3);
\draw  (0,1.396606043771816)-- (-1,0);
\draw  (0,1.3966060437718166)-- (1,0);
\draw  (-1,0)-- (1,0);
\draw  (1,0)-- (0,-1.4044993170019082);
\draw  (-1,0)-- (0,-1.4044993170019082);
\draw  (0,1.3966060437718166)-- (1.814199525010988,1.6054603908470506);
\draw  (1,0)-- (1.814199525010988,1.6054603908470506);
\draw  (-1,0)-- (-1.9574642721711766,1.4703193427395462);
\draw  (-1.9574642721711766,1.4703193427395462)-- (0,1.3966060437718166);
\draw  (0,-1.4044993170019082)-- (0,-3);
\draw  (1.814199525010988,1.6054603908470506)-- (3.0304689579785262,2.3548789303523017);
\draw  (-1.9574642721711766,1.4703193427395462)-- (-3.21059035462258,2.293451181212527);
\draw (-1.2,-2.4) node[anchor=north west] {$G_1$};
\begin{scriptsize}
\draw [fill=black] (0,1.3966060437718166) circle (2pt);
\draw [fill=black] (-1,0) circle (2pt);
\draw [fill=black] (1,0) circle (2pt);
\draw [fill=black] (0,-1.4044993170019082) circle (2pt);
\draw [fill=black] (1.814199525010988,1.6054603908470506) circle (2pt);
\draw [fill=black] (-1.9574642721711766,1.4703193427395462) circle (2pt);
\draw [fill=black] (0,-3) circle (2pt);
\draw [fill=black] (3.0304689579785262,2.3548789303523017) circle (2pt);
\draw [fill=black] (-3.21059035462258,2.293451181212527) circle (2pt);
\end{scriptsize}
\end{tikzpicture}

    \centering
    \begin{tikzpicture}[line cap=round,line join=round,>=triangle 45,x=0.5cm,y=0.5cm]
\clip(-5,-2) rectangle (28,7);
\draw  (0,4)-- (4,4);
\draw  (4,4)-- (4,0);
\draw  (0,4)-- (0,0);
\draw  (0,0)-- (4,0);
\draw  (4,4)-- (8,2);
\draw  (4,0)-- (8,2);
\draw  (0,4)-- (4,0);
\draw  (0,0)-- (4,4);
\draw (0,4)-- (-3.839046052249362,4.783484464135771);
\draw  (0,0)-- (-3.913255750467009,-0.7636904776333466);
\draw (12,2)-- (16,0);
\draw  (12,2)-- (16,4);
\draw  (16,4)-- (16,0);
\draw (16,0)-- (20,2);
\draw  (16,4)-- (20,2);
\draw  (20,2)-- (24,2);
\draw (15.5,-0.5) node[anchor=north west] {$G_3$};
\draw (2,-1) node {$G_2$};
\begin{scriptsize}
\draw [fill=black] (4,4) circle (2pt);
\draw[color=black] (4.240248063684496,4.640696137664811);
\draw [fill=black] (0,4) circle (2pt);
\draw[color=black] (0.23090217990705908,4.640696137664811) ;
\draw[color=black] (2.115893155115854,3.8926838459152897) ;
\draw [fill=black] (4,0) circle (2pt);
\draw[color=black] (4.240248063684496,0.6313502538873781)  ;
\draw[color=black] (3.6119177386148977,2.3667387707462666)  ;
\draw [fill=black] (0,0) circle (2pt);
\draw[color=black] (0.23090217990705908,0.5715092705474165) ;
\draw[color=black] (-0.3675076534925584,2.3667387707462666) ;
\draw[color=black] (2.115893155115854,-0.11666203786214291) ;
\draw [fill=black] (8,2) circle (2pt);
\draw[color=black] (8.249593947461934,2.6360231957760947)  ;
\draw[color=black] (5.885875105533445,2.9352281124759028)  ;
\draw[color=black] (6.304761988913176,0.9305551705871865)  ;
\draw[color=black] (1.7568472550760836,2.0076928707064967) ; 
\draw[color=black] (2.4450185634856436,2.0076928707064967)  ;
\draw [fill=black] (-3.839046052249362,4.783484464135771) circle (2pt);
\draw[color=black] (-3.5989207538504924,5.418628921084313)  ;
\draw[color=black] (-1.7139297786416976,5.209185479394447)  ;
\draw [fill=black] (-3.913255750467009,-0.7636904776333466) circle (2pt);
\draw[color=black] (-3.688682228860435,-0.11666203786214291)  ;
\draw[color=black] (-1.9532937120015443,0.4518273038674931) ;
\draw [fill=black] (12,2) circle (2pt);
\draw[color=black] (12.22901933956939,2.6360231957760947)  ;
\draw [fill=black] (16,0) circle (2pt);
\draw[color=black] (16.23836522334683,0.6313502538873781)  ;
\draw[color=black] (13.904566873088319,0.9305551705871865) ;
\draw [fill=black] (16,4) circle (2pt);
\draw[color=black] (16.23836522334683,4.640696137664811)  ;
\draw[color=black] (14.323453756468052,2.9352281124759028) ;
\draw[color=black] (15.63995538994721,2.3667387707462666)  ;
\draw [fill=black] (20,2) circle (2pt);
\draw[color=black] (20.247711107124264,2.6360231957760947) ;
\draw[color=black] (18.33279964024549,0.9305551705871865)  ;
\draw[color=black] (17.883992265195776,2.9352281124759028) ;
\draw [fill=black] (24,2) circle (2pt);
\draw[color=black] (24.22713649923172,2.6360231957760947)  ;
\draw[color=black] (22.102781590663078,1.8880109040265736)  ;
\end{scriptsize}
\end{tikzpicture}
\caption{}
    \label{fig:cm_2}
\end{figure}

Macaulay 2 computations show that $\dim (R/\II_{G_1}) = 11$, the projective dimension of $R/\II_{G_1}$ is $9$ and hence the depth $(R/\II_{G_1}) = 9.$ Therefore $\II_{G_1}$ is unmixed, but not Cohen-Macaulay. Similarly, for $G_2$, it can be seen that $\dim(R/\II_{G_2}) = 7$ and depth$(R/\II_{G_2}) = 6$. Thus $\II_{G_2}$ is unmixed and not Cohen-Macaulay. 
These observations along with \Cref{cor:cm-whisker} yields the required characterization:

\begin{theorem}\label{thm:c-m}
Let $G$ be a chordal graph and $char (K) \neq 2$. Then $\frac{R}{\II_G}$ is Cohen-Macaulay if and only if $G$ is a path graph or $K_3$ or $G\in \mathfrak{G}_3$.
\end{theorem}
\begin{proof}
If $G$ is bipartite, then we already know that $\frac{R}{\II_G}$ is Cohen-Macaulay if and only if $G$ is a path graph, \cite{EHHNMJ}. Now if $G$ is non-bipartite, then it is unmixed if and only if $G= K_3$ or $G\in \mathfrak{G}_1\cup \mathfrak{G}_2\cup \mathfrak{G}_3$. Then the result follows from \Cref{cor:cm-whisker} and the discussion above.
\end{proof}

\section*{Acknowledgments}

The first named author would like to acknowledge the support from the Prime Minister's Research Fellowship (PMRF) scheme for carrying out this research work. We sincerely thank the anonymous referees who read the paper meticulously and made several comments that improved the exposition considerably.

\section*{Data Availability Statement}
This article has no associated data.

\section*{Disclosure Statement}
No potential conflict of interest was reported by the authors.

 \bibliographystyle{plain}
\bibliography{Bibliography}
\end{document}